\PassOptionsToPackage{dvipsnames,svgnames,table}{xcolor}

\documentclass[11pt]{amsart}  

\usepackage[vmargin=1in, hmargin=1.3in]{geometry}  
\usepackage[font=small,format=plain,labelfont=bf,up,textfont=it,up]{caption}
\usepackage{microtype}                  		
\usepackage{graphicx}						
\usepackage{amssymb}
\usepackage{amsmath}
\usepackage{amsthm}
\usepackage{xcolor}
\usepackage{wasysym}
\usepackage{array}
\usepackage{pst-node}
\usepackage{tikz-cd}
\usepackage{tikz}
\usepackage{tikz-qtree}
\usepackage{mathtools}
\usepackage{amsfonts}
\usepackage{bbm}
\usepackage{listings}
\usepackage{MnSymbol}
\usepackage{multirow}
\usepackage{centernot}
\usepackage{ stmaryrd }
\usepackage{subcaption}
\usepackage[backref=page, bookmarks, bookmarksdepth=2, colorlinks=true, linkcolor=blue, citecolor=blue, urlcolor=blue]{hyperref}
\usepackage{enumitem}

\usepackage[export]{adjustbox}
\usepackage{verbatim}
\usepackage{makecell}

\apptocmd{\thebibliography}{\raggedright}{}{}
\renewcommand*{\backref}[1]{}
\renewcommand*{\backrefalt}[4]{%
    \ifcase #1 (Not cited.)%
    \or        (Cited on page~#2.)%
    \else      (Cited on pages~#2.)%
    \fi}

\theoremstyle{plain}
\newtheorem{theorem}{Theorem}[section]
\newtheorem{maintheorem}{Theorem}

\newtheorem{proposition}[theorem]{Proposition}
\newtheorem{lemma}[theorem]{Lemma}

\newtheorem{corollary}[theorem]{Corollary}
\newtheorem{conjecture}[theorem]{Conjecture}

\newtheorem{problem}[theorem]{Problem}

\theoremstyle{definition}
\newtheorem{definition}[theorem]{Definition}
\theoremstyle{remark}
\newtheorem{nonexample}[theorem]{Non-Example}

\theoremstyle{remark}
\newtheorem{rmk}[theorem]{Remark}
\newenvironment{remark}[1][]{\begin{rmk}[#1]}{\end{rmk}}
\newtheorem{eg}[theorem]{Example}
\newenvironment{example}[1][]{\begin{eg}[#1]}{\end{eg}}

\definecolor{colorblind_blue}{RGB}{0,114,178}
\definecolor{colorblind_orange}{RGB}{213,94,0}
\definecolor{colorblind_green}{RGB}{0,158,115}
\definecolor{colorblind_purple}{RGB}{204,121,167}
\definecolor{colorblind_darkpurple}{RGB}{126,41,84}

\newcommand{\orange}[1]{\ensuremath{{\color{colorblind_orange} #1}}}
\renewcommand{\blue}[1]{\ensuremath{{\color{colorblind_blue} #1}}}
\renewcommand{\green}[1]{\ensuremath{{\color{colorblind_green} #1}}}
\newcommand{\purple}[1]{\ensuremath{{\color{colorblind_purple} #1}}}
\newcommand{\darkpurple}[1]{\ensuremath{{\color{colorblind_darkpurple} #1}}}

\newcommand{\arxiv}[1]{\href{http://arxiv.org/abs/#1}{{\tt arXiv:#1}}}

\newcommand{\Z}{\mathbb Z}
\newcommand{\Q}{\mathbb Q}
\newcommand{\C}{\mathbb C}

\renewcommand{\xmapsto}[1]{\mathrel{\mathop{\mapsto}\limits^{#1}}}

\newcommand\Mod{\ensuremath{\operatorname{Mod}}}
\newcommand{\UQ}{\bar{U}}

\title[The second rational cohomology of handlebody Torelli groups]{Johnson homomorphisms and the second rational cohomology of handlebody Torelli groups}

\author{Annie Holden}
\address{Dept of Mathematics; University of Notre Dame; 255 Hurley Hall; Notre Dame, IN 46556}
\email{aholden2@nd.edu}

\begin{document}

\newpage

\begin{abstract}

We introduce two Torelli subgroups of the handlebody group. The group $\mathcal{HI}_{g,p}^b$ is the subgroup of the handlebody group acting trivially on the first homology of the boundary surface, and $\mathcal{H}_B \mathcal{I}_{g,p}^b$ is the subgroup of the handlebody group acting trivially on the first homology of the handlebody. Using the symplectic representation and the Johnson homomorphisms for the Torelli subgroups of the mapping class group and of $\operatorname{Aut}(F_g)$, we define abelian quotients of these handlebody Torelli groups. In terms of the representation theory of the special linear group, we describe cup products of two classes in the first rational cohomology groups of $\mathcal{HI}_{g,p}^b$ and $\mathcal{H}_B \mathcal{I}_{g,p}^b$ obtained by the rational duals of these abelian quotients. 

\end{abstract}

\maketitle
\thispagestyle{empty}

\section{Introduction}

\subsection{The Torelli subgroup of the mapping class group} Let $\Sigma_{g,p}^b$ be an oriented surface of genus $g$ with $b$ embedded disks and $p$ marked points. In this paper, we assume $p+b \leq 1$. The mapping class group $\Mod_{g,p}^b$ is the group of isotopy classes of orientation-preserving diffeomorphisms of $\Sigma_{g,p}^b$ that fix the embedded disks and marked points pointwise. The mapping class group acts on $H := H_1(\Sigma_{g,p}^b; \Z)$ by automorphisms that fix its symplectic form, and the Torelli group $\mathcal{I}_{g,p}^b$ consists exactly of the elements acting trivially on $H$. Letting $\Psi: \Mod_{g,p}^b \rightarrow \operatorname{Sp}_{2g}(\Z)$ be the symplectic representation, this is encoded in the short exact sequence \begin{equation} \label{torelliSES}
    1 \rightarrow \mathcal{I}_{g,p}^b \rightarrow \Mod_{g,p}^b \xrightarrow{\Psi} \operatorname{Sp}_{2g}(\Z) \rightarrow 1.
\end{equation} In \cite{JohnsonHomom}, Johnson introduces the surjective Johnson homomorphism $\tau$: $$\tau: \mathcal{I}_{g,p}^b \rightarrow \begin{cases} 
     \bigwedge^3 H & (p+b = 1) \\
      (\bigwedge^3 H) / H & (p=b=0)
   \end{cases}.$$ Johnson proves in a series of papers \cite{JohnsonI, JohnsonII, JohnsonIII} that $\tau$ captures the rational abelianization of $\mathcal{I}_{g,p}^b$ when $g \geq 3$. In follow-up work that inspires this paper, Hain \cite{Hain}, Morita \cite{MoritaLinear}, Habegger-Sorger \cite{Habegger}, and Sakasai \cite{Sakasai1} use $\tau$ to study the cup product structure of $H^*(\mathcal{I}_{g,p}^b; \Q)$ in low dimensions. 

\subsection{The Torelli subgroup of $\boldsymbol{\operatorname{Aut}(F_g)}$} \label{subsection:iasubgroupintro} Let $\operatorname{Aut}(F_g)$ be the automorphism group of the free group $F_g$. The Torelli subgroup $IA_g$ is the kernel of the action of $\operatorname{Aut}(F_g)$ on the abelianization of $F_g$, which we denote by $V$. This is encoded in the short exact sequence \begin{equation} \label{iaSES} 1 \rightarrow IA_g \rightarrow \operatorname{Aut}(F_g) \rightarrow \operatorname{GL}_g(\Z) \rightarrow 1. \end{equation} Similarly, there is a Torelli subgroup $OA_g$ of $\operatorname{Out}(F_g) := \operatorname{Aut}(F_g) / \operatorname{Inn}(F_g)$. Each of these Torelli groups has a Johnson homomorphism $$J: \begin{cases} 
     IA_g \rightarrow (\bigwedge^2 V) \otimes V^* \\
     OA_g \rightarrow ((\bigwedge^2 V) \otimes V^*) / V
   \end{cases}.$$ Farb \cite{Farb}, Kawazumi \cite{Kawazumi}, and Cohen-Pakianathan \cite{CohenPakianathan} show independently that $J$ captures the entire abelianization of $IA_g$, and Pettet \cite{Pettet} uses $J$ to study cup products of two elements in $H^1(IA_g; \Q)$ or $H^1(OA_g; \Q)$.

\subsection{Torelli subgroups of the handlebody group} Now fix a handlebody $\mathcal{V}_{g,p}^b$ with $\partial \mathcal{V}_{g,p}^b = \Sigma_{g,p}^b$, and define the handlebody group $\mathcal{H}_{g,p}^b$ to be the subgroup of $\Mod_{g,p}^b$ that extends to $\mathcal{V}_{g,p}^b$. See \cite{Hensel} for a comprehensive background on $\mathcal{H}_{g,p}^b$. In this paper, we study two natural ways to define a Torelli subgroup of $\mathcal{H}_{g,p}^b$. The group $\mathcal{HI}_{g,p}^b$ is the subgroup of the handlebody group acting trivially on the first homology $H$ of the boundary surface, and $\mathcal{H}_B \mathcal{I}_{g,p}^b$ is the subgroup\footnote{In particular, $\mathcal{H}_B \mathcal{I}_{g,p}^b$ contains $\mathcal{HI}_{g,p}^b$ and the ``$B$" emphasizes that $\mathcal{H}_B \mathcal{I}_{g,p}^b$ is the bigger group.} of the handlebody group acting trivially on $V := H_1(\mathcal{V}_{g,p}^b; \Z)$. Recalling that $\Psi$ is the symplectic representation from \eqref{torelliSES}, the following short exact sequence relates our two groups: $$1 \rightarrow \mathcal{HI}_{g,p}^b \rightarrow \mathcal{H}_B \mathcal{I}_{g,p}^b  \rightarrow \Psi(\mathcal{H}_B \mathcal{I}_{g,p}^b) \rightarrow 1.$$ We prove in Proposition \ref{proposition:psihbi} that the image $\Psi(\mathcal{H}_B \mathcal{I}_{g,p}^b)$ is abelian. 

\subsection{Results for $\boldsymbol{\mathcal{HI}_{g,p}^b}$} The Johnson homomorphism $\tau$ for $\mathcal{I}_{g,p}^b$ restricts\footnote{When it is clear from context, we use $\tau$ to denote both the Johnson homomorphism for $\mathcal{I}_{g,p}^b$ and the restriction $\tau|_{\mathcal{HI}_{g,p}^b}$.} to $\mathcal{HI}_{g,p}^b$ and we denote its image by $\UQ$ $(p=b=0)$ or $U$ $(p+b =1)$. To ease notation, we let $p+b=1$. This map $\tau$ is $\mathcal{H}_{g,p}^b$-equivariant, where $\mathcal{H}_{g,p}^b$ acts on $\mathcal{HI}_{g,p}^b$ by conjugation and on $U$ through its action on $H$ via an element of $\Psi(\mathcal{H}_{g,p}^b) \subset \operatorname{Sp}_{2g}(\Z)$. Passing to duals, we can view $U_{\Q}^* := U^* \otimes \Q$ as a subspace of $H^1(\mathcal{HI}_{g,p}^b; \Q)$. 

Now let $\overline{\Psi(\mathcal{H}_{g,p}^b)}$ denote the Zariski closure of $\Psi(\mathcal{H}_{g,p}^b)$. This setup allows us to view $$\tau^*: H^2(U; \Q) \rightarrow H^2(\mathcal{HI}_{g,p}^b; \Q),$$ which is exactly the restriction of the cup product map $$\cup: \bigwedge^2 U_{\Q}^* \rightarrow H^2(\mathcal{HI}_{g,p}^b; \Q),$$ as a map of $\overline{\Psi(\mathcal{H}_{g,p}^b)}$ representations. The group $\overline{\Psi(\mathcal{H}_{g,p}^b)}$ is not semisimple but it naturally contains a subgroup isomorphic to $\operatorname{SL}_g(\Q)$, as we explain in Section \ref{section:setuphi}. By restricting this action to $\operatorname{SL}_g(\Q)$, we can view $\tau^*$ as a map of $\operatorname{SL}_g(\Q)$ representations. The analogous statements are true for $\mathcal{HI}_g$.

To describe $\operatorname{ker}(\tau^*)$, we leverage the fact that $\operatorname{SL}_g(\Q)$ modules are well understood. Briefly, an irreducible $\operatorname{SL}_g(\Q)$ representation $\Phi_{w_1, w_2, \dots, w_{g-1}}$ is uniquely determined by its highest weight $(w_1, w_2, \dots, w_{g-1})$. Any finite dimensional $\operatorname{SL}_g(\Q)$ representation, such as $H^2(U; \Q)$ or $\operatorname{ker}(\tau^*)$, decomposes into a direct sum of these irreducible representations. See Section \ref{section:reptheory} for more details. 

In the following theorem, we give a complete description of cup products of two classes in $H^1(\mathcal{HI}_{g,p}^b; \Q)$ detected by $\tau$.

\begin{table}[h!]
\centering
\begin{tabular}{|c|c|c|c|}
\hline
\multirow{2}{*}{\textbf{Highest Weight}} & \multicolumn{3}{c|}{\textbf{Multiplicity}} \\
\cline{2-4}
 & \textbf{$\operatorname{ker}(\tau^*)$ for $\mathcal{HI}_g$} & \textbf{$\operatorname{ker}(\tau^*)$ for $\mathcal{HI}_{g,1}$} & \textbf{$\operatorname{ker}(\tau^*)$ for $\mathcal{HI}_g^1$} \\
\hline
$(0, \dots, 0)$ & 1 & 1 & 2 \\
$(0, \dots, 0,1,0)$ & 0 & 1 & 1 \\
$(0, 1, 0, \dots, 0)$ & 0 & 1 & 1 \\
$(0, 1, 0, \dots, 0, 1, 0)$ & 1 & 1 & 1 \\
$(0, 2, 0, \dots, 0)$ & 1 & 1 & 1 \\
$(1, 0, \dots, 0, 1)$ & 1 & 2 & 2 \\
$(1, 0, \dots, 0, 1, 1)$ & 1 & 1 & 1 \\
$(1, 1, 0, \dots, 0, 1)$ & 1 & 1 & 1 \\
$(2, 0, \dots, 0)$ & 1 & 1 & 1 \\
$(2, 0, \dots, 0, 2)$ & 1 & 1 & 1 \\
\hline
\end{tabular}
\caption{Decomposition of $\operatorname{ker}(\tau^*)$}
\label{table:theoremB}
\end{table} 

\begin{maintheorem} \label{theorem:kerneltau} For $g \geq 6$, the kernel of $$\tau^*: \begin{cases}
    H^2(\UQ; \Q) \rightarrow H^2(\mathcal{HI}_g; \Q) \\
    H^2(U; \Q) \rightarrow H^2(\mathcal{HI}_{g,p}^b; \Q) & (p+b = 1) 
   \end{cases}$$ decomposes into irreducible $\operatorname{SL}_g(\Q)$ representations as in Table \ref{table:theoremB}. \end{maintheorem}

We also conjecture the following, which will imply that Theorem \ref{theorem:kerneltau} gives a complete description of the cup product map $$\cup: \bigwedge^2 H^1(\mathcal{HI}_{g,p}^b; \Q) \rightarrow H^2(\mathcal{HI}_{g,p}^b; \Q).$$

\begin{conjecture} For $g \gg 0$, we have 
    $H_1(\mathcal{HI}_{g,p}^b; \Q) \cong \begin{cases}
    \UQ_{\Q} & (p=b=0) \\
     U_{\Q} & (p+b = 1) 
   \end{cases}$.
\end{conjecture}

\subsection{Results for $\boldsymbol{\mathcal{H}_B \mathcal{I}_{g,p}^b}$} When $p+b=1$ (resp. $p=b=0$), the handlebody group $\mathcal{H}_{g,p}^b$ acts on $\pi_1(\mathcal{V}_{g,p}^b) = F_g$ by automorphisms (resp. outer automorphisms) and $\mathcal{H}_B \mathcal{I}_{g,p}^b$ is exactly the preimage of $IA_g$ (resp. $OA_g$). From this, we obtain a Johnson homomorphism $J$ on $\mathcal{H}_B \mathcal{I}_{g,p}^b$ as in Section \ref{subsection:iasubgroupintro}. The surjective map $$\Theta: \mathcal{H}_B \mathcal{I}_{g,p}^b \rightarrow \begin{cases}
    \overline{W} := J(\mathcal{H}_B \mathcal{I}_g) \oplus \Psi(\mathcal{H}_B \mathcal{I}_g) \\
   W := J(\mathcal{H}_B \mathcal{I}_{g,p}^b) \oplus \Psi(\mathcal{H}_B \mathcal{I}_{g,p}^b) & (p+b=1) 
   \end{cases}$$ combines the information detected by $J$ and the symplectic representation $\Psi$. To ease notation, we let $p+b=1$. The map $\Theta$ is $\mathcal{H}_{g,p}^b$-equivariant, where $\mathcal{H}_{g,p}^b$ acts on $\mathcal{H}_B \mathcal{I}_{g,p}^b$ by conjugation and on $W$ through its actions on $V$ and $H$ by elements of $\operatorname{GL}_g(\Z)$. Restricting to the action of $\operatorname{SL}_g(\Z)$, we can view the induced map  $$\Theta^*: H^2(W; \Q) \rightarrow H^2(\mathcal{H}_B \mathcal{I}_{g,p}^b; \Q)$$ as a map of $\operatorname{SL}_g(\Q)$ modules. Viewing $W_{\Q}^* := W^* \otimes \Q$ as a subspace of $H^1(\mathcal{H}_B \mathcal{I}_{g,p}^b; \Q)$, the map $\Theta^*$ is exactly the restriction of the cup product map $$\cup: \bigwedge^2 W_{\Q}^* \rightarrow H^2(\mathcal{H}_B \mathcal{I}_{g,p}^b; \Q).$$ The analogous statements hold for $\mathcal{H}_B \mathcal{I}_g$. 

   In the following theorem, we give an almost complete description of cup products of two classes in $H^1(\mathcal{H}_B \mathcal{I}_{g,p}^b; \Q)$ detected by $\Theta$.

\begin{maintheorem} \label{theorem:kerneltheta}
    When $g \geq 3$, the kernel of $$\Theta^*: \begin{cases}
    H^2(\overline{W}; \Q) \rightarrow H^2(\mathcal{H}_B \mathcal{I}_g; \Q) \\
    H^2(W; \Q) \rightarrow H^2(\mathcal{H}_B \mathcal{I}_{g,p}^b; \Q) & (p+b=1) 
   \end{cases}$$ is either: \begin{itemize}
       \item the direct sum of the modules in Table \ref{table:kerneltheta}; or
       \item the direct sum of the modules in Table \ref{table:kerneltheta} plus one copy of the module \[ \begin{cases}
    \Phi_{0,2} & g=3 \\
    \Phi_{0,1,0, \dots, 0,1} & g \geq 4
\end{cases}. \]
   \end{itemize}
    
\end{maintheorem}

\begin{table}[h!]
\centering
\begin{tabular}{|c|c|>{\centering\arraybackslash}m{2cm}|>{\centering\arraybackslash}m{2cm}|}
\hline
\multirow{2}{*}{\textbf{Genus}} & \multirow{2}{*}{\textbf{Highest Weight}} & \multicolumn{2}{c|}{\textbf{Multiplicity in $\operatorname{ker}(\Theta^*)$}} \\
\cline{3-4}
 &  & \textbf{$p=b=0$} & \textbf{$p+b=1$} \\
\hline
$g=3$ & $(1,0)$ & 0 & 2 \\
$g=3$ & $(2,1)$ & 1 & 1 \\
\hline
$g \geq 4$ & $(0,\dots,0,1,0)$ & 0 & 1 \\
$g \geq 4$ & $(1,0,\dots,0)$ & 0 & 1 \\
$g \geq 4$ & $(1,0,\dots,0,1,1)$ & 1 & 1 \\
\hline
\end{tabular}
\caption{Irreducible $\operatorname{SL}_g(\Q)$-modules contained in $\operatorname{ker}(\Theta^*)$}
\label{table:kerneltheta}
\end{table}

We conjecture that the module $\Phi_{0,1,0, \dots, 0,1}$ (resp. $\Phi_{0,2}$) is in $\operatorname{ker}(\Theta^*)$ when $g \geq 4$ (resp. $g=3$), but we cannot prove it. We also conjecture the following, which will imply that Theorem \ref{theorem:kerneltheta} describes the entire cup product map $$\cup: \bigwedge^2 H^1(\mathcal{H}_B \mathcal{I}_{g,p}^b; \Q) \rightarrow H^2(\mathcal{H}_B \mathcal{I}_{g,p}^b; \Q).$$ \begin{conjecture} For $g \gg 0$, we have 
    $H_1(\mathcal{H}_B \mathcal{I}_{g,p}^b; \Q) \cong \begin{cases}
    \overline{W}_{\Q} & (p=b=0) \\
    W_{\Q} & (p+b = 1) 
   \end{cases}$.
\end{conjecture}

\subsection{Proof outline for Theorem \ref{theorem:kerneltheta}} \label{subsection:theoremBoutine} The same outline applies to all $p+b \leq 1$, so to ease notation we let $p+b=1$. In determining the kernel of $$\Theta^*: \bigwedge^2 W^*_{\Q} \rightarrow H^2(\mathcal{H}_B \mathcal{I}_{g,p}^b; \Q),$$ our first step is to decompose $\bigwedge^2 W^*_{\Q}$ as a direct sum of irreducible $\operatorname{SL}_g(\Q)$ representations. The kernel $\operatorname{ker}(\Theta^*)$ is an $\operatorname{SL}_g(\Q)$ subrepresentation of $\bigwedge^2 W_{\Q}^*$. Given a nonzero vector $v$ belonging to some irreducible subrepresentation $\Phi_{w_1, \dots, w_{g-1}}$ of $\bigwedge^2 W_{\Q}^*$, if $v \in \operatorname{ker}(\Theta^*)$ then all of $\Phi_{w_1, \dots, w_{g-1}}$ must lie in $\operatorname{ker}(\Theta^*)$. With this setup, we reduce our problem to determining whether or not each module in the decomposition of $\bigwedge^2 W_{\Q}^*$ belongs to $\operatorname{ker}(\Theta^*)$. 

We have $H_2(W; \Q) \cong \bigwedge^2 W_{\Q}$, where the $\operatorname{SL}_g(\Q)$ decomposition of $\bigwedge^2 W_{\Q}$ consists of the duals of the modules appearing in $\bigwedge^2 W_{\Q}^*$. Our strategy is to understand the induced map on homology $$\Theta_*: H_2(\mathcal{H}_B \mathcal{I}_{g,p}^b; \Q) \rightarrow \bigwedge^2 W_{\Q}.$$ If a module is in $\operatorname{coker}(\Theta_*)$, then its dual is in $\operatorname{ker}(\Theta^*)$. If a module is in $\operatorname{im}(\Theta_*)$, then its dual is not in $\operatorname{ker}(\Theta^*)$.

In Section \ref{section:lowerboundtheta}, we find the dual of each module in Table \ref{table:kerneltheta} in $\operatorname{coker}(\Theta_*)$. We do this by considering the five term exact sequence of the extension $$1 \rightarrow \operatorname{ker}(\Theta) \rightarrow \mathcal{H}_B \mathcal{I}_{g,p}^b \rightarrow W \rightarrow 1,$$ which contains the map $\Theta_*$. We use two quotients of $\operatorname{ker}(\Theta)$ to help us detect this cokernel. One quotient is the second Johnson homomorphism $J_2$ for $\operatorname{Aut}(F_g)$, following the work of Pettet \cite{Pettet} who solved the analogous problem for $IA_g$ and $OA_g$. The other quotient is the first Johnson homomorphism $\tau$ for $\mathcal{I}_{g,p}^b$, which is a new input that deviates from the theory for $\mathcal{I}_{g,p}^b$ and $IA_g$.

In Section \ref{section:upperboundtheta}, we conclude the proof by showing that all of the modules in the decomposition of $\bigwedge^2 W_{\Q}$ except for the duals of those mentioned in Theorem \ref{theorem:kerneltheta} lie in $\operatorname{im}(\Theta_*)$. We do this using the method of abelian cycles, following the work of Hain \cite{Hain}, Sakasai \cite{Sakasai1, Sakasai2}, Pettet \cite{Pettet}, and Brendle-Farb \cite{Brendle}. 

\subsection{Proof outline for Theorem \ref{theorem:kerneltau}} For the groups $\mathcal{HI}_{g,p}^b$, we consider the three cases $p=b=0$ and $p=1$ and $b=1$ separately, building on previous work as we go. We start with $\mathcal{HI}_g$, and our strategy for this case follows that in Section \ref{subsection:theoremBoutine}. To determine the kernel of $$\tau^*: \bigwedge^2 \UQ^*_{\Q} \rightarrow H^2(\mathcal{HI}_g; \Q),$$ we first decompose $\bigwedge^2 \UQ^*_{\Q}$ into a direct sum of irreducible $\operatorname{SL}_g(\Q)$ representations. We then determine whether or not each module in this decomposition belongs to $\operatorname{ker}(\tau^*)$. A module is in $\operatorname{ker}(\tau^*)$ if its dual is in the cokernel of $$\tau_*: H_2(\mathcal{HI}_g; \Q) \rightarrow \bigwedge^2 \UQ_{\Q},$$ and a module is not in $\operatorname{ker}(\tau^*)$ if its dual is in the image $\operatorname{im}(\tau_*)$.

Let $\mathcal{K}_g$ be the kernel of the Johnson homomorphism on $\mathcal{I}_g$ and $\mathcal{HK}_g$ be the intersection $\mathcal{H}_g \cap \mathcal{K}_g$. In Section \ref{section:lowerboundtau}, we consider the five term exact sequence of the extension $$1 \rightarrow \mathcal{HK}_g \rightarrow \mathcal{HI}_g \rightarrow \UQ \rightarrow 1,$$ which contains the map $\tau_*$. We relate this exact sequence to the bracket map $$[\cdot, \cdot]: \bigwedge^2 \tau(\mathcal{I}_g) \rightarrow \tau_2(\mathcal{K}_g)$$ introduced by Morita in \cite{MoritaBracket}. Tensoring with $\Q$ and restricting to $\bigwedge^2 \UQ_{\Q}$, we show that $$[\cdot, \cdot]: \bigwedge^2 \UQ_{\Q} \rightarrow \tau_2^{\Q}(\mathcal{HK}_g)$$ detects all of $\operatorname{coker}(\tau_*)$. This closely (but not exactly) resembles the case for $\mathcal{I}_g$, where an extra input given by the Casson invariant is needed. See \cite{Hain}.

In Section \ref{section:upperboundtau}, we conclude the proof for $\mathcal{HI}_g$ by using abelian cycles to show that the remaining modules in the decomposition of $\bigwedge^2 \UQ_{\Q}$ are in $\operatorname{im}(\tau_*)$, and so their duals survive in $H^2(\mathcal{HI}_g; \Q)$.

In Sections \ref{section:taupuncture} and \ref{section:tauboundary}, we extend our results to the decorated cases of $\mathcal{HI}_{g,1}$ and $\mathcal{HI}_g^1$ by studying the Hochschild-Serre spectral sequences of the extensions $$1 \rightarrow \pi_1 (\Sigma_g) \rightarrow \mathcal{HI}_{g,1} \rightarrow \mathcal{HI}_g \rightarrow 1$$ and $$1 \rightarrow \Z \rightarrow \mathcal{HI}_g^1 \rightarrow \mathcal{HI}_{g,1} \rightarrow 1.$$ These results closely resemble those of the analogous problems for $\mathcal{I}_{g,1}$ and $\mathcal{I}_g^1$. See \cite{MoritaLinear} and \cite{Habegger}.

In all three cases, we find that $\operatorname{ker}(\tau^*)$ is exactly the dual of the rational image of the second Johnson homomorphism $\tau_2^{\Q}(\mathcal{HK}_{g,p}^b) := \tau_2(\mathcal{HK}_{g,p}^b) \otimes \Q$. Faes determines the image $\tau_2(\mathcal{HK}_g^1)$ in \cite{Faes} and describes it as the kernel of a trace map. In Theorem \ref{theorem:decomptau2}, we offer a new perspective and decompose $\tau_2^{\Q}(\mathcal{HK}_{g,p}^b)$ as a direct sum of $\operatorname{SL}_g(\Q)$ modules for all $p+b \leq 1$.

\subsection{Torsion} None of our results in this paper address torsion in the group cohomology of these handlebody Torelli groups. In \cite{Brendle}, Brendle-Farb investigates cup products of two elements in $H^1(\mathcal{I}_{g,p}^b; \Z_2)$ or $H^1(\mathcal{K}_{g,p}^b; \Z_2)$ detected by the Birman-Craggs-Johnson homomorphism $\sigma$. Work in progress by the author shows that the BCJ map detects a very similar structure in $H^*(\mathcal{HI}_{g,p}^b; \Z_2)$ and $H^*(\mathcal{HK}_{g,p}^b; \Z_2)$. \begin{conjecture} For $g \gg 0$, we have 
    $H_1(\mathcal{HI}_{g,p}^b; \Z) \cong \tau(\mathcal{HI}_{g,p}^b) \oplus \sigma(\mathcal{HK}_{g,p}^b)$.
\end{conjecture}

\subsection{Acknowledgments} I would like to thank my advisor Andy Putman for introducing me to this problem and for his continued help and guidance. I also thank Josh Lehman for helpful discussions about spectral sequences.

\section{Preliminaries}

\subsection{(Handlebody) mapping class groups} \label{subsection:mod} The mapping class groups $\Mod_{g,p}^b$ with $p+b \leq 1$ are related by the following exact sequences: \begin{align*}
    1 \rightarrow \Z \rightarrow \Mod_g^1 \rightarrow \Mod_{g,1} \rightarrow 1, \\ 1 \rightarrow \pi_1(\Sigma_g) \rightarrow \Mod_{g,1} \rightarrow \Mod_g \rightarrow 1, \\ 1 \rightarrow \pi_1(U \Sigma_g) \rightarrow \Mod_g^1 \rightarrow \Mod_g \rightarrow 1.
\end{align*} The Dehn twist $T_{\partial}$ around the embedded disk of $\Sigma_g^1$ generates the group $\Z$. The map $\pi_1(U \Sigma_g) \rightarrow \Mod_g^1$ (resp. $\pi_1(\Sigma_g) \rightarrow \Mod_{g,1}$) is induced by ``pushing the embedded disk (resp. the marked point) around curves." See \cite[Chapter 4]{Primer}. \\

A genus $g$ handlebody $\mathcal{V}_g$ is the oriented 3-manifold with $\partial \mathcal{V}_g = \Sigma_g$ obtained from a 3-ball by attaching $g$ one-handles. We denote by $\mathcal{V}_{g,p}^b$ a handlebody with $p$ marked points and $b$ embedded disks in its boundary $\partial \mathcal{V}_{g,p}^b = \Sigma_{g,p}^b$. The handlebody group $\mathcal{H}_{g,p}^b$ is the subgroup of $\Mod_{g,p}^b$ that extends to $\mathcal{V}_{g,p}^b$.

\subsection{Conventions} \label{subsection:conventions} We first define conventions for the surface with one embedded disk $\Sigma_g^1$ and then extend them to the marked and undecorated cases. \\

Place a basepoint on the boundary of the embedded disk in $\Sigma_g^1$. Viewing $\Sigma_g^1$ as the surface minus the embedded disk, $\pi_1(\Sigma_g^1)$ is the free group $F_{2g}$ on generators $\alpha_1, \alpha_2, \dots, \alpha_g, \beta_1, \beta_2, \dots, \beta_g$. The abelianization $H := H_1(\Sigma_g^1; \Z)$ is the free abelian group $\Z \langle a_1, a_2, \dots, a_g, b_1, b_2, \dots, b_g \rangle$, where $a_i$ (resp. $b_i$) is the homology class of $\alpha_i$ (resp. $\beta_i$). We include an image with curves $\mathfrak{a}_i$ and $\mathfrak{b}_{i}$ whose homology classes are $a_i$ and $b_i$, respectively. Throughout this paper, we fix a handlebody $\mathcal{V}_g^1$ such that $\partial \mathcal{V}_g^1 = \Sigma_g^1$ and the curves $\beta_i$ and $\mathfrak{b}_i$ bound disks in $\mathcal{V}_g^1$. \\ \centerline{\includegraphics[scale=.5]{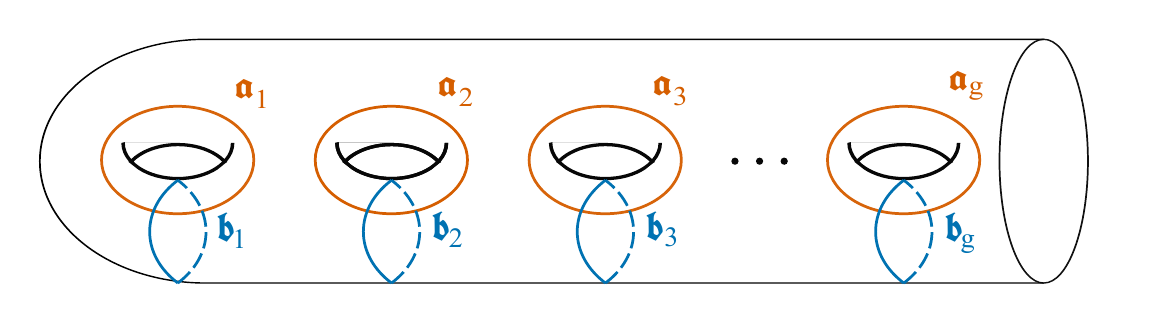}} Then $V := H_1(\mathcal{V}_g^1; \Z) = \Z \langle a_1, \dots, a_g \rangle$ and $\pi_1(\mathcal{V}_g^1) = \langle \alpha_1, \dots, \alpha_g | \rangle \cong F_g$. The curves $\beta_i$ normally generate the kernel of the surjection $\pi_1(\Sigma_g^1) \rightarrow \pi_1(\mathcal{V}_g^1)$ and the handlebody group $\mathcal{H}_g^1$ is exactly the subgroup of $\Mod_g^1$ preserving this kernel. 

In the following subsections, we further discuss the actions of $\Mod_g^1$ on these fundamental groups and homology groups.

\subsection{Action on surface homology.} \label{subsection:actiononhomology} Poincaré duality gives us an isomorphism $$H \cong H^* = H^1(\Sigma_g^1; \Z),$$ where $a_i$ (resp. $b_i$) corresponds to $-b_i^*$ (resp. $a_i^*$). We use the same symbol $H$ to denote these canonically isomorphic groups. The abelian group $H$ is equipped with a symplectic intersection form $\omega$, and the action of $\Mod_g^1$ preserves this form. Recalling that we denote $H_1(\mathcal{V}_g^1; \Z) = \Z \langle a_1, \dots, a_g \rangle$ by $V$, we get the decomposition $$H = V \oplus V^*.$$ Everything behaves the same when $b=0$: we have the isomorphisms $$H \cong H_1(\Sigma_{g,1}; \Z) \cong H_1(\Sigma_g; \Z),$$ and we have the following commutative diagram with exact rows. \[\begin{tikzcd}[ampersand replacement=\&]
	1 \& {\mathcal{I}_g^1} \& {\Mod_g^1} \& {\operatorname{Sp}_{2g}(\mathbb{Z})} \& 1 \\
	1 \& {\mathcal{I}_{g,1}} \& {\Mod_{g,1}} \& {\operatorname{Sp}_{2g}(\mathbb{Z})} \& 1 \\
	1 \& {\mathcal{I}_g} \& {\Mod_g} \& {\operatorname{Sp}_{2g}(\mathbb{Z})} \& 1
	\arrow[from=1-1, to=1-2]
	\arrow[from=1-2, to=1-3]
	\arrow[from=1-2, to=2-2]
	\arrow["\Psi", from=1-3, to=1-4]
	\arrow[from=1-3, to=2-3]
	\arrow[from=1-4, to=1-5]
	\arrow[equals, from=1-4, to=2-4]
	\arrow[from=2-1, to=2-2]
	\arrow[from=2-2, to=2-3]
	\arrow[from=2-2, to=3-2]
	\arrow["\Psi", from=2-3, to=2-4]
	\arrow[from=2-3, to=3-3]
	\arrow[from=2-4, to=2-5]
	\arrow[equals, from=2-4, to=3-4]
	\arrow[from=3-1, to=3-2]
	\arrow[from=3-2, to=3-3]
	\arrow["\Psi", from=3-3, to=3-4]
	\arrow[from=3-4, to=3-5]
\end{tikzcd}\] The map $\Mod_g^1 \rightarrow \Mod_{g,1}$ is given by forgetting all but the middle point of the embedded disk, and the map $\Mod_{g,1} \rightarrow \Mod_g$ is given by forgetting the marked point. \\

The handlebody group $\mathcal{H}_{g,p}^b$ acts on $H$, and Hirose proves in \cite{Hirose} that its image in $\operatorname{Sp}_{2g}(\Z)$ is $$\Psi(\mathcal{H}_{g,p}^b) \cong \left\{  \begin{bmatrix} A & 0 \\ C & (A^t)^{-1} \end{bmatrix} \; | \; A \in  \operatorname{GL}_{g}(\Z), \; C = C^t \right\}.$$

\noindent Hirose chose the convention that the $\mathfrak{a}_i$ curves bound disks in the fixed handlebody, so he called this image $urSp(2g)$, where {\textit{ur}} stands for ``upper right." We instead have the $\mathfrak{b}_i$ curves bound disks in our handlebody, so the nontrivial entries of our matrices lie in the lower left.

\subsection{(Outer) automorphisms of $\boldsymbol{F_g}$.} \label{subsection:autfn} Let $\operatorname{Aut}(F_g)$ be the automorphism group of the free group $F_g$ and let $IA_g$ be its Torelli subgroup as in Section \ref{subsection:iasubgroupintro}. 

In \cite{Magnus}, Magnus determines a finite set of generators for $IA_g$ and we set notation for them here. The group $IA_g$ is finitely generated by the maps \begin{itemize}
    \item $C_{ij}: x_j \mapsto x_i x_j x_i^{-1}$
    \item $M_{ijk}: x_k \mapsto x_k[x_i,x_j]$
\end{itemize} where we assume $1 \leq i,j,k \leq g$ are distinct, $i < j$, and all unmentioned generators of $F_g$ are fixed.

Similarly, let $\operatorname{Out}(F_g) := \operatorname{Aut}(F_g) / \operatorname{Inn}(F_g)$ be the outer automorphism group of $F_g$. The Torelli subgroup $OA_g$ lies in the short exact sequence $$ 1 \rightarrow OA_g \rightarrow \operatorname{Out}(F_g) \rightarrow \operatorname{GL}_g(\Z) \rightarrow 1.$$ The groups $IA_g$ and $OA_g$ give free group analogues of the Torelli subgroups of the based and unbased mapping class groups, respectively. 

\subsection{Action on fundamental groups.} \label{subsection:actiononpi1} The mapping class group $\Mod_{g,p}^b$ acts on $\pi_1(\Sigma_{g,p}^b)$ via automorphisms (if $p+b=1$) or outer automorphisms (if $p=b=0$). These actions are well understood and appear in the literature in many places (see, e.g., \cite[Section 2]{MoritaLinear}). When $p+b=1$, the fundamental group $\pi_1(\Sigma_{g,p}^b)$ is a free group and the Torelli group acts via elements of $IA_{2g}$. \\

Equally important in this paper is the action of $\mathcal{H}_{g,p}^b$ on $\pi_1(\mathcal{V}_{g,p}^b) \cong F_g$. See \cite[Section 6]{Hensel}. Define the twist group $\mathcal{T}_{g,p}^b$ to be the subgroup of $\mathcal{H}_{g,p}^b$ generated by Dehn twists about meridians (curves that bound disks in $\mathcal{V}_{g,p}^b$). Letting $p+b=1$, the handlebody group $\mathcal{H}_{g,p}^b$ surjects onto $\operatorname{Aut}(F_g)$ (see \cite{Griffiths}) and, by \cite{Luft}, we have the short exact sequence $$1 \rightarrow \mathcal{T}_{g,p}^b \rightarrow \mathcal{H}_{g,p}^b \xrightarrow{\alpha} \operatorname{Aut}(F_g) \rightarrow 1.$$ By definition, $\mathcal{H}_B \mathcal{I}_{g,p}^b$ is the preimage of $IA_g$ under $\alpha$. In fact, $\mathcal{HI}_{g,p}^b$ also surjects onto $IA_g$. See \cite[Section 2]{Om}. For the unbased case $\mathcal{H}_g$, analogous results hold regarding the map into $\operatorname{Out}(F_g)$.

\section{The Johnson homomorphisms}

In this section, we define the Johnson homomorphisms on subgroups of $\operatorname{Aut}(F_g)$ and $\operatorname{Out}(F_g)$ and $\Mod_{g,p}^b$. 

\subsection{The first Johnson homomorphism for $\boldsymbol{IA_g}$.} \label{subsection:firstjohnsonIA} 

There exists a surjective homomorphism $$J: IA_g \rightarrow (\bigwedge^2 H_1(F_g; \Z)) \otimes H_1(F_g; \Z)^*.$$ Recalling that we denote $H_1(\mathcal{V}_{g,p}^b; \Z)$ by $V$ and viewing $IA_g$ as a subgroup of $\operatorname{Aut}(\pi_1(\mathcal{V}_{g,p}^b))$, this becomes $$J: IA_g \rightarrow (\bigwedge^2 V) \otimes V^*.$$ It evaluates on generators as:\footnote{The negative signs appear in our definition but not Pettet's \cite{Pettet} because the Poincaré duality isomorphism introduces them.} \begin{itemize}
    \item $J(C_{ij}) = -(a_i \wedge a_j) \otimes b_j$
    \item $J(M_{ijk}) = -(a_i \wedge a_j) \otimes b_k.$
\end{itemize} 

The Johnson homomorphism for $OA_g$ surjects onto the quotient of $(\bigwedge^2 V) \otimes V^*$: $$J: OA_g \rightarrow ((\bigwedge^2 V) \otimes V^*) / V.$$ 

The Johnson homomorphism for $\mathcal{H}_B \mathcal{I}_{g,p}^b$ (which we also denote $J$) is the composition of the map $\alpha$ into $\operatorname{Aut}(F_g)$ or $\operatorname{Out}(F_g)$ defined in Section \ref{subsection:actiononpi1} with $J$.

\subsection{Higher Johnson homomorphisms for subgroups of $\boldsymbol{IA_g}$ and $\boldsymbol{OA_g}$.} \label{subsection:higherIA} Define $$IA_g(k) := \operatorname{ker}(IA_g \rightarrow \operatorname{Aut}(F_g / \gamma_{k+1}(F_g))),$$ where $\gamma_{k+1}(F_g)$ is the $(k+1)^{\text{st}}$ term in the lower central series of $F_g$. Under these indexing conventions, $IA_g(1) = IA_g$ and $IA_g(2) = [IA_g, IA_g]$. The $k^{\text{th}}$ Johnson homomorphism $J_k$ is an abelian quotient map defined on $IA_g(k)$ which detects the action on $F_g / \gamma_{k+2}(F_g)$. See \cite{Satoh} for this construction. In this paper, we only need the second Johnson homomorphism: $$J_2: IA_g(2) \rightarrow (((\bigwedge^2 V) \otimes V) / (\bigwedge^3 V)) \otimes V^*.$$ See \cite[Section 2.3]{Pettet} for this construction. There are analogues of the maps $J_k$ for $OA_g(k)$.

This gives us higher Johnson homomorphisms $J_k$ on the preimages of $IA_g(k)$ (resp. $OA_g(k)$) in $\mathcal{H}_B \mathcal{I}_{g,p}^b$ when $p+b =1$ (resp. $p=b=0$). 

\subsection{The first Johnson homomorphism for $\boldsymbol{\mathcal{I}_{g,p}^b}$.} \label{subsection:firstjohnsonTorelli}

When $p+b=1$, the Torelli group $\mathcal{I}_{g,p}^b$ acts on $\pi_1(\Sigma_{g,p}^b)$ via elements of $IA_{2g}$, so we can apply $J$ defined in Section \ref{subsection:firstjohnsonIA} to an element of $\mathcal{I}_{g,p}^b$. This process outputs an element of $(\bigwedge^2 H) \otimes H^*$. The image is a submodule of $(\bigwedge^2 H) \otimes H^*$ isomorphic to $\bigwedge^3 H$, so we write $$\tau: \mathcal{I}_{g,p}^b \rightarrow \bigwedge^3 H.$$

To understand $\tau$, it suffices to know how it evaluates on generators. The Torelli group $\mathcal{I}_{g,p}^b$ is generated by bounding pair maps, i.e., products $T_x T_y^{-1}$ for which $x$ and $y$ are disjoint nonseparating simple closed curves on $\Sigma_{g,p}^b$ such that $x \cup y$ separates the surface. Given a bounding pair map $T_x T_y^{-1}$, let $S_{x,y} \cong \Sigma_k^2$ be the subsurface bounded by $x \cup y$ on the side \textit{not} containing the basepoint or embedded disk. In this case, we say $T_x T_y^{-1}$ is a genus-$k$ bounding pair map. Choose a maximal symplectic subspace of $H_1(S_{x,y}; \Z)$ with basis $x_i, y_i$ for $i = 1, \dots, k$. Then $H_1(S_{x,y}; \Z)$ is a free abelian group generated by $x_i, y_i$ $(i = 1, \dots, k)$ and $c = [x]$, choosing orientation so that $S_{x,y}$ is on the left of $c$. Then the Johnson homomorphism evaluates as $$\tau(T_x T_y^{-1}) = (\sum_{i=1}^k x_i \wedge y_i) \wedge c,$$ which does not depend on our choice of symplectic basis. \\

The construction for $\mathcal{I}_g$ goes as follows. Let $\omega = \sum_{i=1}^g a_i \wedge b_i$ be the algebraic intersection form. Then $H$ is a subgroup of $\bigwedge^3 H$ by the $\operatorname{Sp}_{2g}(\Z)$-equivariant injection $$H \hookrightarrow \bigwedge^3 H \quad (x \mapsto x \wedge \omega).$$ The Johnson homomorphism is defined to be the map that makes the following diagram commute:

\[\begin{tikzcd}
	1 & {\pi_1(U\Sigma_g)} & {\mathcal{I}_g^1} & {\mathcal{I}_g} & 1 \\
	0 & H & {\bigwedge^3 H} & {(\bigwedge^3 H) / H} & 0
	\arrow[from=1-1, to=1-2]
	\arrow[from=1-2, to=1-3]
	\arrow["\tau", from=1-2, to=2-2]
	\arrow[from=1-3, to=1-4]
	\arrow["\tau", from=1-3, to=2-3]
	\arrow[from=1-4, to=1-5]
	\arrow["\tau", dashed, from=1-4, to=2-4]
	\arrow[from=2-1, to=2-2]
	\arrow[from=2-2, to=2-3]
	\arrow[from=2-3, to=2-4]
	\arrow[from=2-4, to=2-5]
\end{tikzcd}\] 

The Johnson homomorphism on $\mathcal{HI}_{g,p}^b$, which we also denote $\tau$, is simply the restriction $\tau|_{\mathcal{HI}_{g,p}^b}$. A bounding pair map $T_x T_y^{-1}$ is in $\mathcal{HI}_{g,p}^b$ if $x \sqcup y$ bounds an annulus in $\mathcal{V}_{g,p}^b$, so the discussion above allows us to evaluate $\tau$ on bounding pair annulus twists in $\mathcal{HI}_{g,p}^b$.  

\subsection{Higher Johnson homomorphisms for subgroups of $\boldsymbol{\mathcal{I}_{g,p}^b}$.} \label{subsection:higherMod} Following Section \ref{subsection:higherIA}, we define the Johnson filtration $$\mathcal{I}_g^1(k) := \operatorname{ker}(\Mod_g^1 \rightarrow \operatorname{Aut}(\pi_1(\Sigma_g^1) / \gamma_{k+1} (\pi_1(\Sigma_g^1)))$$ for $k \geq 1$. Under these indexing conventions, $\mathcal{I}_g^1(1) = \mathcal{I}_g^1$. The $k^{th}$ Johnson homomorphism $\tau_k$ is defined on $\mathcal{I}_g^1(k)$. These higher Johnson homomorphisms were first defined in \cite{MoritaHigher} and there is a Lie algebra structure on their images $$\bigoplus_k \operatorname{im}(\tau_k).$$ The bracket map $[ \cdot, \cdot ]$, introduced in \cite{MoritaBracket}, is defined as follows. For $v \in \operatorname{im}(\tau_n)$ and $w \in \operatorname{im}(\tau_m)$, there exist preimages $f \in \mathcal{I}_g^1(n)$ and $h \in \mathcal{I}_g^1(m)$. We define $$[v,w] := \tau_{n+m}(fhf^{-1}h^{-1}).$$ From this construction, note that if $v \in \tau_n(\mathcal{H}_g^1 \cap \mathcal{I}_g^1(n))$ and $w \in \tau_m(\mathcal{H}_g^1 \cap \mathcal{I}_g^1(m))$, then $[v,w] \in \tau_{n+m}(\mathcal{H}_g^1 \cap \mathcal{I}_g^1(n+m))$. Namely, it makes sense to consider this Lie algebra structure when studying the intersection of $\mathcal{H}_g^1$ with groups $\mathcal{I}_g^1(k)$ in the Johnson filtration. In Section \ref{section:upperboundtau}, we show how to explicitly calculate the bracket map on generators of $\bigwedge^2 \operatorname{im}(\tau)$. 

The second Johnson homomorphism $\tau_2$ is defined on the Johnson kernel $\mathcal{K}_g^1 := \mathcal{I}_g^1(2)$. This restricts to $\mathcal{HK}_g^1$. See \cite{Faes} and Section \ref{section:decomptau2}. 

The groups $\mathcal{I}_{g,1}(k)$ and $\mathcal{I}_g(k)$ are defined analogously.

\section{Representation theory} \label{section:reptheory}

In this paper, we work with finite dimensional representations of $\operatorname{SL}_g(\Q)$ and $\operatorname{Sp}_{2g}(\Q)$. These are in natural bijection with the representations of their Lie algebras $\mathfrak{sl}_g(\Q)$ and $\mathfrak{sp}_{2g}(\Q)$, respectively. In particular, any representation constructed from the standard representation (and its dual) via tensor, exterior, or symmetric products has the same weight vectors and highest weight decomposition as the corresponding representation of the Lie algebra. In the following subsections, we briefly review the representation theory of $\mathfrak{sl}_g(\Q)$ and $\mathfrak{sp}_{2g}(\Q)$. We will use this theory to understand our $\operatorname{SL}_g(\Q)$ and $\operatorname{Sp}_{2g}(\Q)$ representations.

This section should suffice for readers familiar with representation theory to understand our notation, which largely follows Fulton–Harris \cite[Sections 11–17]{FH}. For additional details and worked examples, we refer the reader to the Appendix. In particular, we discuss Lie algebra actions (Appendix~\ref{appendix:liealgebra}), highest weight vectors (Appendices~\ref{appendix:hwv} and~\ref{appendix:constructhwv}), and computing with LiE (Appendix~\ref{appendix:LiE}).

\subsection{$\mathfrak{sl}_g(\Q)$ and $\mathfrak{sp}_{2g}(\Q)$ maps} \label{section:reptheorymaps} For later use, we define the elements $E_{ij}$ of $\mathfrak{sl}_g(\Q)$ and $X_i$ of $\mathfrak{sp}_{2g}(\Q)$ characterized by their action on $H_{\Q}$ as follows. Letting $\delta_{ij}$ be the Kronecker delta function, we have: \begin{align*}
E_{ij}(a_k) &= \delta_{jk} a_i \quad & E_{ij}(b_k) &= - \delta_{ik} b_j \\
X_i(a_k) &= \delta_{ik} b_i \quad & X_i(b_k) &= 0
\end{align*} We also consider the handleswap\footnote{We will consider the action of $\mathcal{H}_{g,p}^b$ on $\operatorname{SL}_g(\Q)$ vector spaces, and this handleswap action comes from the element of $\mathcal{H}_{g,p}^b$ that interchanges the $i^{\text{th}}$ and $j^{\text{th}}$ handles of the handlebody. See \cite{Suzuki}.} action of $\operatorname{SL}_g(\Q)$, which permutes generators $a_i \leftrightarrow a_j$ and $b_i \leftrightarrow b_j$. We do not define a symbol for this action but we freely reindex vectors when considering their $\operatorname{SL}_g(\Q)$-orbits. \\

We also define the $\operatorname{Sp}_{2g}(\Q)$ and $\operatorname{SL}_g(\Q)$ equivariant contraction $$C_k: \bigwedge^k H_{\Q} \rightarrow \bigwedge^{k-2} H_{\Q}$$ given by $$C_k(x_1 \wedge \dots \wedge x_k) = \sum_{1 \leq i < j \leq k} (-1)^{i+j+1} \omega(x_i, x_j) x_1 \wedge \dots \wedge \hat{x_i} \wedge \dots \wedge \hat{x_j} \wedge \dots \wedge x_k$$ where $\hat{x_i}$ means to exclude $x_i$ and $\bigwedge^0 H_{\Q}$ is the trivial representation $\Q$.

\subsection{Irreducible $\mathfrak{sl}_g(\Q)$ modules} Let $\mathfrak{h}$ be the Cartan subalgebra of diagonal matrices in $\mathfrak{sl}_g(\C)$, and let $L_i \in \mathfrak{h}^*$ be the element which sends a matrix to its $i^{\text{th}}$ diagonal entry. Any irreducible representation of $\mathfrak{sl}_g(\C)$ has a unique (up to scaling) highest weight vector $v$. This highest weight vector is killed by $E_{ij} \in \mathfrak{sl}_g(\C)$ for all $i<j$ and is an eigenvector for $\mathfrak{h}$. Moreover, $H \in \mathfrak{h}$ scales $v$ with eigenvalue $$w_1 L_1 + w_2 (L_1 + L_2) + \dots + w_{g-1}(L_1 + L_2 + \dots + L_{g-1})$$ where each $w_i$ is a nonnegative integer. The tuple $(w_1, w_2, \dots, w_{g-1})$ is the \textit{highest weight} of $v$ and this weight determines the irreducible representation up to isomorphism. We denote this irreducible representation by $\Phi_{w_1, w_2, \dots, w_{g-1}}$. The Lie algebra $\mathfrak{sl}_g(\C)$ is semisimple, so every finite dimensional representation of $\mathfrak{sl}_g(\C)$ can be decomposed into a direct sum of irreducible representations. 

All of these representations are defined over $\Q$, so they can be considered as irreducible representations of $\operatorname{SL}_g(\Q)$ and $\mathfrak{sl}_g(\Q)$.

\subsection{Irreducible $\mathfrak{sp}_{2g}(\Q)$ modules} Let $\mathfrak{h}$ be the Cartan subalgebra of matrices diagonal in $\mathfrak{sp}_{2g}(\Q)$. This Cartan subalgebra is spanned by the $2g \times 2g$ matrices $M_i = E_{i,i} - E_{g+i,g+i}$ for $i=1, \dots, g$. The matrix $M_i$ acts on the standard $\mathfrak{sp}_{2g}(\Q)$ representation $H_{\Q}$ by fixing $a_i$, sending $b_i$ to $-b_i$, and killing the other basis vectors. The dual $\mathfrak{h}^*$ has basis $L_j$ where $\langle L_j, M_i \rangle = \delta_{ij}$. An irreducible $\mathfrak{sp}_{2g}(\Q)$ representation $\Gamma_{w_1, w_2, \dots, w_g}$ is determined by a unique (up to scaling) highest weight vector $v$, which is an eigenvector for $\mathfrak{h}$ with unique eigenvalue $$w_1 L_1 + w_2 (L_1 + L_2) + \dots + w_g(L_1 + L_2 + \dots + L_g).$$ Again, each $w_i$ is a nonnegative integer. We adopt the convention that omitted trailing indices are zero, so $\Gamma_{w_1, \dots, w_k} := \Gamma_{w_1, \dots, w_k, 0, \dots, 0}$. These representations are defined over $\Q$, so we consider them as irreducible representations of $\operatorname{Sp}_{2g}(\Q)$ and $\mathfrak{sp}_{2g}(\Q)$. \\

We highlight two important examples of $\operatorname{Sp}_{2g}(\Q)$ representations: \begin{itemize}
    \item $H_{\Q} = H_1(\Sigma_{g,p}^b; \Q)$ is the standard $\operatorname{Sp}_{2g}(\Q)$ representation $\Gamma_1$.
    \item $\bigwedge^3 H_{\Q} \cong \Gamma_{0,0,1} \oplus \Gamma_1$, where $H_{\Q}$ is considered a subgroup of $\bigwedge^3 H_{\Q}$ via the injection $x \mapsto x \wedge \omega$.
\end{itemize}

There is a subgroup of $\operatorname{Sp}_{2g}(\Q)$ isomorphic to $\operatorname{SL}_g(\Q)$. It consists of matrices $$\left\{ \begin{bmatrix} A & 0 \\ 0 & (A^t)^{-1} \end{bmatrix} \; | \; A \in \operatorname{SL}_g(\Q) \right\}$$ which preserve the decomposition $H_{\Q} = \Q \langle a_1, \dots, a_g \rangle \oplus \Q \langle b_1, \dots, b_g \rangle$. Letting $V_{\Q}$ be the standard $\operatorname{SL}_g(\Q)$ representation, $H_{\Q}$ decomposes into $\operatorname{SL}_g(\Q)$ representations as $$H_{\Q} = V_{\Q} \oplus V_{\Q}^*.$$ In general, we can branch any $\operatorname{Sp}_{2g}(\Q)$ representation into a direct sum of $\operatorname{SL}_g(\Q)$ representations in this way.  

\section{Setup for $\mathcal{H}_B \mathcal{I}_{g,p}^b$} \label{section:setuphbi}

In this section, we consider $\mathcal{H}_B \mathcal{I}_{g,p}^b$ which we recall is the subgroup of $\mathcal{H}_{g,p}^b$ that acts trivially on $V := H_1(\mathcal{V}_{g,p}^b; \Z)$. We define an abelian quotient map $\Theta$ which detects the action of $\mathcal{H}_B \mathcal{I}_{g,p}^b$ on both $\pi_1(\mathcal{V}_{g,p}^b)$ and $H_1(\Sigma_{g,p}^b; \Z)$, and we calculate the kernel of the induced map $\Theta^*$ on second cohomology as in Theorem \ref{theorem:kerneltheta}. We first address this problem for $\mathcal{H}_B \mathcal{I}_{g,p}^b$ when $p+b =1$. The construction is exactly the same for $\mathcal{H}_B \mathcal{I}_g^1$ and $\mathcal{H}_B \mathcal{I}_{g,1}$, so we treat both at once in the remainder of this section as well as Sections \ref{section:lowerboundtheta} and \ref{section:upperboundtheta}. To ease notation, we focus our attention on $\mathcal{H}_B \mathcal{I}_g^1$ in these sections. We then extend our results to $\mathcal{H}_B \mathcal{I}_g$ in Section \ref{section:thetaclosed}.

First, we state a result by Omori \cite{Om} which gives us generating sets for our handlebody Torelli groups.

\begin{theorem}[{Omori, \cite[Theorem 1.2 and Corollary 1.4]{Om}}] \label{theorem:Omori} Let $g \geq 3$ and let $T_{C_1} T_{C_2}^{-1}$ be a genus 1 bounding pair map such that $C_1$ and $C_2$ do not bound disks in $\mathcal{V}_g^1$ and the disjoint union $C_1 \sqcup C_2$ bounds an annulus in $\mathcal{V}_g^1$. Then \begin{itemize}
    \item $\mathcal{HI}_g^1$ is normally generated in $\mathcal{H}_g^1$ by $T_{C_1} T_{C_2}^{-1}$, and
    \item $\mathcal{H}_B \mathcal{I}_g^1$ is normally generated in $\mathcal{H}_g^1$ by $T_{C_1} T_{C_2}^{-1}$ and one nonseparating disk twist. 
\end{itemize} \end{theorem}

\subsection{The symplectic representation.} \label{subsection:hbipsi} Recall from Section \ref{subsection:actiononhomology} that the surjective symplectic representation $\Psi: \operatorname{Mod}_g^1 \rightarrow \operatorname{Sp}_{2g}(\Z)$ measures the action of $\operatorname{Mod}_g^1$ on $H$. Hirose shows in \cite{Hirose} that the restriction of $\Psi$ to the handlebody group has image $$\Psi(\mathcal{H}_g^1) \cong \left\{  \begin{bmatrix} A & 0 \\ C & (A^t)^{-1} \end{bmatrix} \; | \; A \in  \operatorname{GL}_{g}(\Z), \; C = C^t \right\}.$$

We define the map $$ul: \Psi(\mathcal{H}_g^1) \rightarrow \operatorname{GL}_g(\Z)$$ which assigns to each matrix its upper left block of size $g \times g$, so we have $$\begin{bmatrix} A & 0 \\ C & (A^t)^{-1} \end{bmatrix} \xmapsto{ul} A.$$

\noindent We recall that $V := H_1(\mathcal{V}_g^1; \Z)$ is the free abelian group $\Z \langle a_1, \dots, a_g \rangle$, and $V^* = \Z \langle b_1, \dots, b_g \rangle$.

\begin{proposition} \label{proposition:sym2SES} The group $\Psi(\mathcal{H}_g^1)$ fits into a short exact sequence $$1 \rightarrow \operatorname{Sym}^2(V^*) \rightarrow \Psi(\mathcal{H}_g^1) \xrightarrow{ul} \operatorname{GL}_g(\Z) \rightarrow 1,$$ where the induced action of $\operatorname{GL}_g(\Z)$ on $\operatorname{Sym}^2(V^*)$ is the natural one. \end{proposition}

\begin{proof}[Proof of Proposition \ref{proposition:sym2SES}]
The kernel of $ul: \Psi(\mathcal{H}_g^1) \rightarrow \operatorname{GL}_g(\Z)$ is the group of matrices $$\left\{ \begin{bmatrix} I & 0 \\ C & I \end{bmatrix} \in  \operatorname{Mat}_{2g}(\Z) \; | \; C = C^t \right\}.$$ This group is isomorphic to $\operatorname{SymMat}_g(\Z)$, the additive group of $g \times g$ symmetric matrices via the map $$\begin{bmatrix} I & 0 \\ C & I \end{bmatrix} \mapsto C.$$ Note that $$\begin{bmatrix} I & 0 \\ C_1 & I \end{bmatrix} \begin{bmatrix} I & 0 \\ C_2 & I \end{bmatrix} = \begin{bmatrix} I & 0 \\ C_1 + C_2 & I \end{bmatrix}.$$ 

View the group $\operatorname{SymMat}_g(\Z)$ as a subgroup of $\operatorname{Hom}(V,V^*)$, which is isomorphic to $V^* \otimes V^*$ via the identification $a_i^* \rightarrow -b_i$ (recall Section \ref{subsection:actiononhomology}). Since these matrices are symmetric, we get the natural $\operatorname{GL}_g(\Z)$-equivariant isomorphism $$\operatorname{SymMat}_g(\Z) \cong \operatorname{Sym}^2(V^*).$$ See \cite[Section 2]{SakasaiL}. \\

We calculate that $$\begin{bmatrix} A & 0 \\ 0 & (A^t)^{-1} \end{bmatrix} \begin{bmatrix} I & 0 \\ C & I \end{bmatrix} \begin{bmatrix} A^{-1} & 0 \\ 0 & A^t \end{bmatrix} = \begin{bmatrix} I & 0 \\ (A^t)^{-1} C A^{-1} & I \end{bmatrix},$$ so $A \in \operatorname{GL}_g(\Z)$ acts on $\operatorname{SymMat}_g(\Z)$ via $A \cdot C = (A^t)^{-1} C A^{-1}$. This is the natural action on $\operatorname{Sym}^2(V^*)$.
\end{proof}

The symplectic representation further restricts to $\mathcal{H}_B \mathcal{I}_g^1$, and we have the following proposition. 

\begin{proposition} \label{proposition:psihbi}
    We have $\Psi(\mathcal{H}_B \mathcal{I}_g^1) = \operatorname{Sym}^2(V^*) \subset \Psi(\mathcal{H}_g^1)$.
\end{proposition}

\noindent In particular, $\Psi(\mathcal{H}_B \mathcal{I}_g^1)$ is abelian and the conjugation action of $\Psi(\mathcal{H}_g^1)$ factors through $\operatorname{GL}_g(\Z)$. Before we prove Proposition \ref{proposition:psihbi}, we introduce a lemma.

\begin{lemma} \label{lemma:sym2v}
The $\operatorname{GL}_g(\Z)$-orbit of $b_1 \otimes b_1$ is $\operatorname{Sym}^2(V^*)$.   
\end{lemma}

\begin{proof}
    The group $\operatorname{Sym}^2(V^*)$ is generated by elements of the form $b_i \otimes b_i$ and $b_i \otimes b_j + b_j \otimes b_i$ for $i \neq j$. Any $b_i \otimes b_i$ differs from $b_1 \otimes b_1$ by a handleswap action which permutes the basis elements of $V^*$. Now consider the element of $\operatorname{GL}_g(\Z)$ that sends $b_i \rightarrow b_i + b_j$ and fixes all other basis elements of $V^*$. Applying this to $b_i \otimes b_i$, we get $$b_i \otimes b_i + b_i \otimes b_j + b_j \otimes b_i + b_j \otimes b_j.$$ Since $b_i \otimes b_i + b_j \otimes b_j$ is in the $\operatorname{GL}_g(\Z)$-orbit of $b_1 \otimes b_1$, so is $b_i \otimes b_j + b_j \otimes b_i$.
\end{proof}

We are now ready to prove Proposition \ref{proposition:psihbi}.

\begin{proof}[Proof of Proposition \ref{proposition:psihbi}] First we show that $\Psi(\mathcal{H}_B \mathcal{I}_g^1) \subset \left\{ \begin{bmatrix} I & 0 \\ C & I \end{bmatrix} \in  \operatorname{Mat}_{2g}(\Z) \; | \; C = C^t \right\}$, and so by Proposition \ref{proposition:sym2SES} we have $\Psi(\mathcal{H}_B \mathcal{I}_g^1) \subset \operatorname{Sym}^2(V^*)$. Given $f \in \mathcal{H}_B \mathcal{I}_g^1$, consider $\Psi(f) = \begin{bmatrix} A & 0 \\ C & (A^t)^{-1} \end{bmatrix}.$ Since $f$ acts trivially on $V = \Z \langle a_1, \dots, a_g \rangle$, it follows that $A = I$. \\

Viewing $\Psi(\mathcal{H}_B \mathcal{I}_g^1)$ as a subgroup of $\operatorname{Sym}^2(V^*)$, by Lemma \ref{lemma:sym2v} it suffices to show that $b_1 \otimes b_1 \in \Psi(\mathcal{H}_B \mathcal{I}_g^1)$. Let $\mathfrak{b}_1$ be a meridian whose homology class is $b_1$, as in the figure in Section \ref{subsection:conventions}. The disk twist $T_{\mathfrak{b}_1}$ acts on $H$ by sending $a_1 \rightarrow a_1 + b_1$ and fixing all other generators. Recall that we have identified the dual vector $a_i^*$ with $-b_i$. Then, as an element of $\operatorname{Hom}(V,V^*) \cong V^* \otimes V^*$, the element $\Psi(T_{\mathfrak{b}_1})$ is $-b_1 \otimes b_1.$
\end{proof}

We will now pass to the action of $\operatorname{SL}_g(\Z) \subset \operatorname{GL}_g(\Z)$ on $\Psi(\mathcal{H}_B \mathcal{I}_g^1)$. Extending our action of $\operatorname{SL}_g(\Z)$ to an action of $\operatorname{SL}_g(\Q)$ on $\Psi^{\Q}(\mathcal{H}_B \mathcal{I}_g^1) := \Psi(\mathcal{H}_B \mathcal{I}_g^1) \otimes \Q$, we can view $\Psi^{\Q}(\mathcal{H}_B \mathcal{I}_g^1)$ as the irreducible $\operatorname{SL}_g(\Q)$ representation $\operatorname{Sym}^2(V^*_{\Q})$.

\subsection{The Johnson homomorphism.} \label{subsection:JsetupforHBI} In Section \ref{subsection:firstjohnsonIA}, we introduced the surjective map $$J: \mathcal{H}_B \mathcal{I}_g^1 \rightarrow (\bigwedge^2 V) \otimes V^*.$$ The restriction $J|_{\mathcal{HI}_g^1}$ is also surjective \cite[Section 2]{Om}. The handlebody group $\mathcal{H}_g^1$ acts on $\mathcal{H}_B \mathcal{I}_g^1$ by conjugation and on $(\bigwedge^2 V) \otimes V^*$ through its action on $V$. The Johnson homomorphism is $\mathcal{H}_g^1$-equivariant in the sense that, for $h \in \mathcal{H}_g^1$ and $f \in \mathcal{H}_B \mathcal{I}_g^1$, we have $$J(hfh^{-1}) = h_* J(f).$$ The handlebody group acts on $V$ by integral general linear matrices. This comes from the fact that $\operatorname{Aut}(F_g)$ acts on $(\bigwedge^2 V) \otimes V^*$ by elements of $\operatorname{GL}_g(\Z)$, and the handlebody group surjects onto $\operatorname{Aut}(F_g)$. See \cite{Pettet} and the short exact sequence \eqref{iaSES}. We will now restrict this action to $\operatorname{SL}_g(\Z)$ and then extend to an action of $\operatorname{SL}_g(\Q)$ on $(\bigwedge^2 V_{\Q}) \otimes V_{\Q}^*$. This allows us to view $(\bigwedge^2 V_{\Q}) \otimes V_{\Q}^*$ as an $\operatorname{SL}_g(\Q)$ representation.

\subsection{Abelian quotient $\boldsymbol{W}$} \label{subsection:W}

We construct an abelian quotient of $\mathcal{H}_B \mathcal{I}_g^1$ by combining the information detected by $J$ and $\Psi$ into one map. Define $$\Theta: \mathcal{H}_B \mathcal{I}_g^1 \rightarrow ((\bigwedge^2 V) \otimes V^*) \oplus \operatorname{Sym}^2(V^*) =: W$$ by $$\Theta(f) = (J(f), \Psi(f)).$$ The restriction of $\Theta$ to the smaller handlebody group has image $$\Theta(\mathcal{HI}_g^1) = (\bigwedge^2 V) \otimes V^* \oplus \{ 0 \},$$ and we showed in Section \ref{subsection:hbipsi} that the image of the normal closure of $T_{\mathfrak{b}_1}$ in $\mathcal{H}_g^1$ is $\{ 0 \} \oplus \operatorname{Sym}^2(V^*)$. It follows that $\Theta$ is surjective. \\

We will determine the kernel of $$\Theta^*: H^2(W; \Q) \rightarrow H^2(\mathcal{H}_B \mathcal{I}_g^1; \Q).$$ Applying the Künneth formula, $H^2(W; \Q) \cong \bigwedge^2 W_{\Q}^*$ decomposes as: 
$$H^2((\bigwedge^2 V) \otimes V^*; \Q) \oplus [H^1((\bigwedge^2 V) \otimes V^*; \Q) \otimes H^1(\operatorname{Sym}^2(V^*); \Q)] \oplus H^2(\operatorname{Sym}^2(V^*); \Q)$$ 

This gives us \begin{equation} \label{equation:bigwedge2W} \bigwedge^2 W_{\Q}^* \cong \bigwedge^2 ((\bigwedge^2 V_{\Q}) \otimes V_{\Q}^*)^* \oplus [((\bigwedge^2 V_{\Q}) \otimes V_{\Q}^*) \otimes \operatorname{Sym}^2(V_{\Q}^*)]^* \oplus \bigwedge^2 (\operatorname{Sym}^2(V_{\Q}^*))^*. \end{equation} 

Each of these summands further decomposes into a direct sum of irreducible $\operatorname{SL}_g(\Q)$ representations. Importantly, $\operatorname{ker}(\Theta^*)$ is a subrepresentation of $H^2(W; \Q)$, and so is a direct sum of some of these irreducible modules. In the following two sections, we prove the following theorem.

\begin{theorem} \label{theorem:thetaopen}
    When $g \geq 3$ and $p+b=1$, the kernel of $\Theta^*: H^2(W; \Q) \rightarrow H^2(\mathcal{H}_B \mathcal{I}_{g,p}^b; \Q)$ contains the following irreducible $\operatorname{SL}_g(\Q)$ modules:
        \[ \begin{cases} 
      2\Phi_{1,0} \oplus \Phi_{2,1} & g = 3 \\
      \Phi_{0, \dots, 0,1,0} \oplus \Phi_{1,0, \dots, 0} \oplus \Phi_{1,0, \dots, 0,1,1} & g \geq 4 \end{cases}. \] Moreover, $\operatorname{ker}(\Theta^*)$ is either: \begin{itemize}
          \item The direct sum of the modules listed above; or
          \item The direct sum of the modules listed above plus one copy of the module \[ \begin{cases}
    \Phi_{0,2} & g=3 \\
    \Phi_{0,1,0, \dots, 0,1} & g \geq 4
\end{cases}. \]
      \end{itemize}
\end{theorem} 

In Section \ref{section:lowerboundtheta}, we will prove containment. In Section \ref{section:upperboundtheta}, we will show that the remaining modules in the decomposition of $H^2(W; \Q)$, except for possibly for $\Phi_{0,1,0, \dots, 0,1}$ (resp. $\Phi_{0,2}$) when $g \geq 4$ (resp. $g=3$), are {\textit{not}} in $\operatorname{ker}(\Theta^*)$.

\section{Lower bound on $\operatorname{ker}(\Theta^*)$ for $\mathcal{H}_B \mathcal{I}_g^1$} \label{section:lowerboundtheta}

In this section, we prove the following proposition.

\begin{proposition} \label{proposition:lowerboundHBI}
    When $g \geq 3$ and $p+b=1$, the kernel of $\Theta^*: H^2(W; \Q) \rightarrow H^2(\mathcal{H}_B \mathcal{I}_{g,p}^b; \Q)$ contains the following direct sum of $\operatorname{SL}_g(\Q)$ representations: \[ \begin{cases} 
      2\Phi_{1,0} \oplus \Phi_{2,1} & g = 3 \\
      \Phi_{0, \dots, 0,1,0} \oplus \Phi_{1,0, \dots, 0} \oplus \Phi_{1,0, \dots, 0,1,1} & g \geq 4
   \end{cases}.
\]
\end{proposition} 

\noindent We will see that the second Johnson homomorphism $J_2$ for $IA_g(2) := \operatorname{ker}(J) \subset IA_g$ (see Section \ref{subsection:higherIA}) detects two modules and the first Johnson homomorphism $\tau$ for $\mathcal{I}_g^1$ (see Section \ref{subsection:firstjohnsonTorelli}) detects one module in $\operatorname{ker}(\Theta^*)$.

\subsection{Five term exact sequence} \label{subsection:5tesHBI} A module is in $\operatorname{ker}(\Theta^*)$ if its dual module is in $\operatorname{coker}(\Theta_*)$. From the five term exact sequence (see \cite{Stallings}) of the extension $$1 \rightarrow \operatorname{ker}(\Theta) \rightarrow \mathcal{H}_B \mathcal{I}_g^1 \rightarrow W \rightarrow 1,$$ we get an exact sequence $$H_2(\mathcal{H}_B \mathcal{I}_g^1; \Q) \xrightarrow{\Theta_*} H_2(W; \Q) \xrightarrow{b} H_1(\operatorname{ker}(\Theta); \Q)_{\mathcal{H}_B \mathcal{I}_g^1}.$$ Identifying $H_2(W; \Q) \cong \bigwedge^2 W_{\Q}$ and $H_1(\operatorname{ker}(\Theta); \Q)_{\mathcal{H}_B \mathcal{I}_g^1} \cong (\operatorname{ker}(\Theta) / [\mathcal{H}_B \mathcal{I}_g^1, \operatorname{ker}(\Theta)]) \otimes \Q$ turns this into an exact sequence $$H_2(\mathcal{H}_B \mathcal{I}_g^1; \Q) \xrightarrow{\Theta_*} \bigwedge^2 W_{\Q} \xrightarrow{b} (\operatorname{ker}(\Theta) / [\mathcal{H}_B \mathcal{I}_g^1, \operatorname{ker}(\Theta)]) \otimes \Q.$$ We therefore have $$\operatorname{coker}(\Theta_*) \cong \operatorname{im}(b).$$ 
This map $b$ is induced from the following map. Given $x \wedge y$ in $\bigwedge^2 W$, let $\varphi$ and $\psi$ in $\mathcal{H}_B \mathcal{I}_g^1$ be preimages of $x$ and $y$, respectively. Then we send $$x \wedge y \mapsto \overline{[\varphi,\psi]},$$ where $\overline{[\varphi,\psi]}$ is the class of the commutator $[\varphi,\psi]$ in the quotient $\operatorname{ker}(\Theta) / [\mathcal{H}_B \mathcal{I}_g^1, \operatorname{ker}(\Theta)]$. 

The key to detecting modules in $\operatorname{im}(b)$ is examining the image of the composition of $b$ with another $\operatorname{SL}_g(\Q)$-equivariant map. By Schur's Lemma, any module in the image of such a composition must also lie in $\operatorname{im}(b)$.

\subsection{Modules detected by $\boldsymbol{J_2}$} \label{subsection:lowerboundthetaJ2}

In this subsection we prove Proposition \ref{proposition:kernelj2}, from which part of Proposition \ref{proposition:lowerboundHBI} will follow. 

\begin{proposition} \label{proposition:kernelj2}
    When $g \geq 3$ and $p+b=1$, the kernel of $\Theta^*: H^2(W; \Q) \rightarrow H^2(\mathcal{H}_B \mathcal{I}_{g,p}^b; \Q)$ contains the following direct sum of $\operatorname{SL}_g(\Q)$ representations: 
    \[ \begin{cases} 
      \Phi_{1,0} \oplus \Phi_{2,1}  & g = 3 \\
      \Phi_{0, \dots, 0, 1,0} \oplus \Phi_{1,0, \dots, 0,1,1} & g \geq 4 
   \end{cases}, \] which lies in the $\bigwedge^2 ((\bigwedge^2 V_{\Q}) \otimes V^*)^*$ summand of $\bigwedge^2 W^*_{\Q}$.
\end{proposition}

See Equation \eqref{equation:bigwedge2W}.

\begin{remark}
    When $g \geq 3$, Pettet \cite[Theorem 1.2]{Pettet} shows that the kernel of $$J^*: H^2((\bigwedge^2 V) \otimes V^*; \Q) \rightarrow H^2(IA_g; \Q)$$ is isomorphic to the direct sum in Proposition \ref{proposition:kernelj2}.
\end{remark}

From Section \ref{subsection:W}, we see that $\operatorname{ker}(\Theta)$ is exactly the intersection $\operatorname{ker}(J) \cap \operatorname{ker}(\Psi)$, where $J$ is the first Johnson homomorphism as defined in Sections \ref{subsection:firstjohnsonIA} and \ref{subsection:JsetupforHBI}, and $\operatorname{ker}(\Psi)$ is the smaller handlebody group $\mathcal{HI}_g^1$. Recall from Section \ref{subsection:higherIA} that there is a second Johnson homomorphism $J_2$ defined on $IA_g(2) := \operatorname{ker}(J)$ $$J_2: IA_g(2) \rightarrow (((\bigwedge^2 V) \otimes V)/(\bigwedge^3 V)) \otimes V^*.$$ Since $\operatorname{ker}(\Theta) \subset IA_g(2)$, we can restrict $J_2$ to $\operatorname{ker}(\Theta)$. Moreover, $J_2$ factors through the quotient $\operatorname{ker}(\Theta) / [\mathcal{H}_B \mathcal{I}_g^1, \operatorname{ker}(\Theta)]$ because $[\mathcal{H}_B \mathcal{I}_g^1, \operatorname{ker}(\Theta)] \subset \operatorname{ker}(J_2)$. Tensoring with $\Q$, the map $J_2^{\Q}$ is $\operatorname{SL}_g(\Q)$-equivariant. \\

Proposition \ref{proposition:kernelj2} will follow from Lemmas \ref{lemma:imageBj2} and \ref{lemma:J2pettet}, which we now introduce.

\begin{lemma} \label{lemma:imageBj2} Consider the composition of $\operatorname{SL}_g(\Q)$ maps $$J_2^{\Q} \circ b: \bigwedge^2 W_{\Q} \rightarrow (((\bigwedge^2 V_{\Q}) \otimes V_{\Q})/(\bigwedge^3 V_{\Q})) \otimes V_{\Q}^*.$$ When $g \geq 3$, the following direct sum of $\operatorname{SL}_g(\Q)$ modules lies in the image of $J_2^{\Q} \circ b$: 
    \[ \begin{cases} 
      \Phi_{0,1} \oplus \Phi_{1,2}  & g = 3 \\
      \Phi_{0,1,0, \dots, 0} \oplus \Phi_{1,1,0, \dots, 0,1} & g \geq 4 
   \end{cases} .
\]
\end{lemma}

To prove Lemma \ref{lemma:imageBj2}, we introduce a lemma from \cite{Pettet}.

\begin{lemma}[{Pettet, \cite[Section 3.2]{Pettet}}] \label{lemma:J2pettet} The $\operatorname{SL}_g(\Q)$ orbit of the vector $$-(a_1 \wedge a_2) \otimes a_1 \otimes b_1 - (a_1 \wedge a_2) \otimes a_2 \otimes b_2$$ in $(((\bigwedge^2 V_{\Q}) \otimes V_{\Q})/(\bigwedge^3 V_{\Q})) \otimes V^*_{\Q}$ contains the modules \[ \begin{cases} 
      \Phi_{0,1} \oplus \Phi_{1,2}  & g = 3 \\
      \Phi_{0,1,0, \dots, 0} \oplus \Phi_{1,1,0, \dots, 0,1} & g \geq 4 
   \end{cases}.
\]
\end{lemma}

\begin{proof}[Proof of Lemma \ref{lemma:imageBj2}] 

By Lemma \ref{lemma:J2pettet}, it suffices to find the vector $$-(a_1 \wedge a_2) \otimes a_1 \otimes b_1 - (a_1 \wedge a_2) \otimes a_2 \otimes b_2$$ in the image of $J_2^{\Q} \circ b.$

Recall that the Torelli subgroup $IA_g$ of $\operatorname{Aut}(F_g) = \operatorname{Aut}(\pi_1(\mathcal{V}_g^1))$ contains elements $C_{ij}$ which send $\alpha_i \mapsto \alpha_j \alpha_i \alpha_j^{-1}$ and fix all other basis elements. Let $h$ and $k$ be preimages of $C_{12}$ and $C_{21}$ respectively in the smaller handlebody group $\mathcal{HI}_g^1 = \mathcal{H}_g^1 \cap \mathcal{I}_g^1$, which we recall surjects onto $IA_g$. Because $\mathcal{HI}_g^1$ acts trivially on surface homology $H$, the vectors $\Theta(h)$ and $\Theta(k)$ lie entirely in the $(\bigwedge^2 V) \otimes V^*$ summand of $W$. Then we have $$b(\Theta(h) \wedge \Theta(k)) =  \overline{[h,k]}.$$ Recall from Section \ref{subsection:higherIA} that the second Johnson homomorphism $J_2$ is originally defined on $IA_g(2) \subset IA_g$, and we evaluate it on $\overline{[h,k]}$ by first mapping to $IA_g(2)$. The map $[h,k]$ acts on $\pi_1(\mathcal{V}_g^1)$ via the automorphism $[C_{12}, C_{21}]$. By \cite[Section 3.2]{Pettet}, we have $$J_2([C_{12}, C_{21}]) = -(a_1 \wedge a_2) \otimes a_1 \otimes b_1 - (a_1 \wedge a_2) \otimes a_2 \otimes b_2.$$ It follows that $$J_2^{\Q} \circ b(\Theta(h) \wedge \Theta(k)) = -(a_1 \wedge a_2) \otimes a_1 \otimes b_1 - (a_1 \wedge a_2) \otimes a_2 \otimes b_2. \qedhere $$ \end{proof}

\subsection{Module detected by $\boldsymbol{\tau}$} \label{subsection:lowerboundthetatau}  

In this subsection we will prove Proposition \ref{proposition:kernelThetaTau}, from which the last piece of Proposition \ref{proposition:lowerboundHBI} will follow. 

\begin{proposition} \label{proposition:kernelThetaTau}
    When $g \geq 3$ and $p+b=1$, the kernel of $\Theta^*: H^2(W; \Q) \rightarrow H^2(\mathcal{H}_B \mathcal{I}_{g,p}^b; \Q)$ contains the irreducible $\operatorname{SL}_g(\Q)$ representation $\Phi_{1,0, \dots, 0}$. This lies in the $$((\bigwedge^2 V_{\Q}) \otimes V_{\Q}^*)^* \otimes (\operatorname{Sym}^2(V_{\Q}^*))^*$$ summand of $\bigwedge^2 W^*_{\Q}$.
\end{proposition}

See Equation \eqref{equation:bigwedge2W}. In genus 3, this gives us a second copy of $\Phi_{1,0}$ in $\operatorname{ker}(\Theta^*)$ since we already detected one in Section \ref{subsection:lowerboundthetaJ2}. At the end of this subsection, in Remark \ref{remark:twocopies01}, we verify that these modules are distinct. \\

To detect $\Phi_{1,0, \dots, 0}$ in $\operatorname{ker}(\Theta^*)$, we find its dual $\Phi_{0, \dots, 0,1}$ in $\operatorname{coker}(\Theta_*)$. We do this by composing the map $b$ from Section \ref{subsection:5tesHBI} with the first Johnson homomorphism $\tau: \mathcal{I}_g^1 \rightarrow \bigwedge^3 H$ defined in Section \ref{subsection:firstjohnsonTorelli}. Since we have $$\operatorname{ker}(\Theta) \subset \operatorname{ker}(\Psi) = \mathcal{HI}_g^1 = \mathcal{H}_g^1 \cap \mathcal{I}_g^1,$$ we can restrict the Johnson homomorphism $\tau$ (which a priori was not defined on all of $\mathcal{H}_B \mathcal{I}_g^1$) to $\operatorname{ker}(\Theta)$. This does not factor through the quotient $\operatorname{ker}(\Theta) / [\mathcal{H}_B \mathcal{I}_g^1, \operatorname{ker}(\Theta)]$, so we take coinvariants and tensor with $\Q$ to get an $\operatorname{SL}_g(\Q)$ map $$\tau^{\Q}: (\operatorname{ker}(\Theta) / [\mathcal{H}_B \mathcal{I}_g^1, \operatorname{ker}(\Theta)]) \otimes \Q \rightarrow (\bigwedge^3 H_{\Q})_{\mathcal{H}_B \mathcal{I}_g^1}.$$ We introduce the following lemma, which makes sense of $(\bigwedge^3 H_{\Q})_{\mathcal{H}_B \mathcal{I}_g^1}$.

\begin{lemma} \label{lemma:coinvariants}
    The quotient $(\bigwedge^3 H_{\Q})_{\mathcal{H}_B \mathcal{I}_g^1}$ of $\bigwedge^3 H_{\Q}$ contains the module $\Phi_{0, \dots, 0,1}$ for all $g \geq 3$.
\end{lemma}

\begin{proof} We have the decomposition $$\bigwedge^3 H_{\Q} \cong \Phi_{0,0,1,0, \dots, 0} \oplus \Phi_{0,1,0, \dots, 0,1} \oplus \Phi_{1,0, \dots, 0} \oplus \Phi_{1,0, \dots, 0,1,0} \oplus \Phi_{0, \dots, 0,1} \oplus \Phi_{0, \dots, 0,1,0,0}.$$ See Section \ref{section:moduledecomp}. In particular, $\bigwedge^3 H_{\Q}$ contains one copy of $\Phi_{0, \dots, 0,1}$ in its decomposition.

Taking coinvariants, $(\bigwedge^3 H_{\Q})_{\mathcal{H}_B \mathcal{I}_g^1}$ is the biggest quotient of $\bigwedge^3 H_{\Q}$ on which $\mathcal{H}_B \mathcal{I}_g^1$ acts trivially. We recall from Section \ref{subsection:hbipsi} that $\mathcal{H}_B \mathcal{I}_g^1$ acts on $H_{\Q}$ via the matrices $$\left\{ \begin{bmatrix} I & 0 \\ C & I \end{bmatrix} \in  \operatorname{Mat}_{2g}(\Q) \; | \; C = C^t \right\},$$ and in particular acts trivially on $V_{\Q}^*$.

In Section \ref{section:reptheorymaps}, we defined the contraction $$C_3: \bigwedge^3 H_{\Q} \rightarrow H_{\Q} = V_{\Q} \oplus V_{\Q}^*$$ which is $\mathcal{H}_g^1$-equivariant and hence $\mathcal{H}_B \mathcal{I}_g^1$-equivariant. The projection from $H_{\Q}$ to $V^*_{\Q} \cong \Phi_{0, \dots, 0,1}$ is an $\mathcal{H}_B \mathcal{I}_g^1$-equivariant map, and $\mathcal{H}_B \mathcal{I}_g^1$ acts trivially on its image. It follows that $(\bigwedge^3 H_{\Q})_{\mathcal{H}_B \mathcal{I}_g^1}$ contains the module $\Phi_{0, \dots, 0,1}$ in its decomposition. \end{proof} 

Recall that a module is in $\operatorname{ker}(\Theta^*)$ if its dual module is in $\operatorname{coker}(\Theta_*) \cong \operatorname{im}(b)$. By Schur's Lemma, Proposition \ref{proposition:kernelThetaTau} will follow from the following lemma. 

\begin{lemma} \label{lemma:cokernelThetaTau} Consider the composition of $\operatorname{SL}_g(\Q)$ maps $$\tau^{\Q} \circ b: \bigwedge^2 W_{\Q} \rightarrow (\bigwedge^3 H_{\Q})_{\mathcal{H}_B \mathcal{I}_g^1}.$$  When $g \geq 3$, the image of $\tau^{\Q} \circ b$ contains the irreducible $\operatorname{SL}_g(\Q)$ representation $\Phi_{0, \dots, 0,1}$. 
\end{lemma}

\begin{proof}[Proof of Lemma \ref{lemma:cokernelThetaTau}]

Consider the following bounding pair annulus twist $T_x T_{x'}^{-1}$ and the disk twist $T_{\mathfrak{b}_2}$ in $\mathcal{H}_B \mathcal{I}_g^1$.

\centerline{\includegraphics[scale=.5]{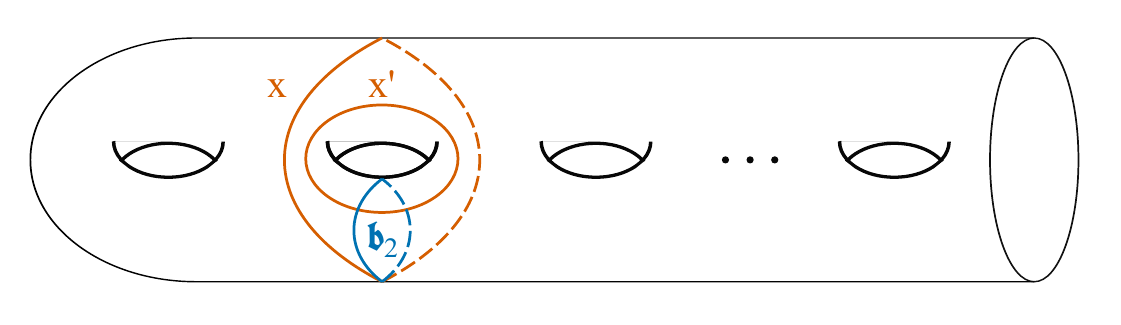}}

\noindent Note that $T_x T_{x'}^{-1}$ acts trivially on surface homology $H$, so $\Theta^{\Q}(T_x T_{x'}^{-1})$ lies entirely in the $(\bigwedge^2 V_{\Q}) \otimes V_{\Q}^*$ summand of $W_{\Q}$. On the other hand, $T_{\mathfrak{b}_2}$ acts trivially on $\pi_1(\mathcal{V}_g^1)$ and so $\Theta^{\Q}(T_{\mathfrak{b}_2})$ lies entirely in the $\operatorname{Sym}^2(V_{\Q}^*)$ summand of $W_{\Q}$. Hence, $$\Theta^{\Q}(T_x T_{x'}^{-1}) \wedge \Theta^{\Q}(T_{\mathfrak{b}_2}) \in ((\bigwedge^2 V_{\Q}) \otimes V_{\Q}^*) \otimes \operatorname{Sym}^2(V_{\Q}^*) \subset \bigwedge^2 W_{\Q}.$$

We want to understand the $\operatorname{SL}_g(\Q)$ module $\operatorname{im}(\tau^{\Q} \circ b)$ as a quotient of $((\bigwedge^2 V_{\Q}) \otimes V_{\Q}^*) \otimes \operatorname{Sym}^2(V_{\Q}^*)$. To do this, we calculate $(\tau^{\Q} \circ b)(\Theta^{\Q}(T_x T_{x'}^{-1}) \wedge \Theta^{\Q}(T_{\mathfrak{b}_2}))$. Then we use the fact that $\operatorname{SL}_g(\Q)$ representations are in natural bijection with $\mathfrak{sl}_g(\Q)$ representations, and we find a highest weight vector for $\Phi_{0, \dots, 0,1}$ in the $\mathfrak{sl}_g(\Q)$-orbit of $(\tau^{\Q} \circ b)(\Theta^{\Q}(T_x T_{x'}^{-1}) \wedge \Theta^{\Q}(T_{\mathfrak{b}_2}))$. \\

First we apply the map $b$ $$b(\Theta^{\Q}(T_x T_{x'}^{-1}) \wedge \Theta^{\Q}(T_{\mathfrak{b}_2})) = \overline{[T_x T_{x'}^{-1},T_{\mathfrak{b}_2}]}.$$ To compute $\tau^{\Q}(\overline{[T_x T_{x'}^{-1},T_{\mathfrak{b}_2}]})$, we use the following naturality property of $\tau^{\Q}$. If $h \in \mathcal{H}_g^1$ and $k \in \operatorname{ker}(\Theta)$, then $$\tau^{\Q}(hkh^{-1}) = h_*\tau^{\Q}(k).$$ See \cite[Equation 6.1]{Primer}. So then we have \begin{align*} & \tau^{\Q}([T_x T_{x'}^{-1},T_{\mathfrak{b}_2}]) \\ &= \tau^{\Q}(T_x T_{x'}^{-1}) + (T_{\mathfrak{b}_2})_*\tau^{\Q}((T_x T_{x'}^{-1})^{-1}) \\ &= \tau^{\Q}(T_x T_{x'}^{-1}) - (T_{\mathfrak{b}_2})_*\tau^{\Q}(T_x T_{x'}^{-1}). \end{align*}

Noting that the curve $x$ represents $a_2$ in homology and recalling from Section \ref{subsection:firstjohnsonIA} how $\tau$ evaluates on bounding pair maps, we calculate $$\tau^{\Q}(T_x T_{x'}^{-1}) = a_1 \wedge b_1 \wedge a_2.$$ Since $T_{\mathfrak{b}_2}$ is a disk twist around the curve $\mathfrak{b}_2$ whose homology class is $b_2$, the map $T_{\mathfrak{b}_2}$ acts on $H$ by sending $$(T_{\mathfrak{b}_2})_*(a_2) = a_2 + b_2$$ and fixing all other homology generators. Thus we have $$(T_{\mathfrak{b}_2})_*(a_1 \wedge b_1 \wedge a_2) = a_1 \wedge b_1 \wedge a_2 + a_1 \wedge b_1 \wedge b_2.$$ It follows then that $\tau^{\Q}([T_x T_{x'}^{-1},T_{\mathfrak{b}_2}]) = a_1 \wedge b_1 \wedge b_2$.

Recall that we defined highest weight vectors and the $\mathfrak{sl}_g(\Q)$ maps $E_{ij}$ and $C_k$ in Section \ref{section:reptheory}, and we provide more details and examples in the Appendix. We calculate that $$-E_{2g} \circ C_3 (a_1 \wedge b_1 \wedge b_2) = -E_{2g}(b_2) = b_g$$ is a highest weight vector for $\Phi_{0, \dots, 0,1}$ for all $g \geq 3$. By Lemma \ref{lemma:coinvariants}, this module is nonzero after taking coinvariants. Since we found a highest weight vector from $\Phi_{0, \dots, 0,1}$ in the image $\operatorname{im}(\tau^{\Q} \circ b)$, all of $\Phi_{0, \dots, 0,1}$ must lie in $\operatorname{im}(\tau^{\Q} \circ b)$.  \end{proof}

\begin{remark} \label{remark:twocopies01}
    When $g=3$, there are two copies of the module $\Phi_{0,1}$ in $\bigwedge^2 W_{\Q}$ because $\Phi_{0,1,0, \dots, 0} \cong \bigwedge^2 V_{\Q}$ and $\Phi_{0, \dots, 0,1} \cong V_{\Q}^*$ both have weight $(0,1)$.
    
    In Section \ref{subsection:lowerboundthetaJ2}, we detected the module $\Phi_{0,1}$ in $\operatorname{coker}(\Theta_*)$ by applying $\operatorname{SL}_3(\Q)$ maps to the vector $J_2^{\Q}(\overline{[h,k]})$. Since $\overline{[h,k]} \in \operatorname{ker}(\tau^{\Q})$, we detected a second copy of $\Phi_{0,1}$ in $\operatorname{coker}(\Theta_*)$ in this section.
\end{remark}

Propositions \ref{proposition:kernelj2} and \ref{proposition:kernelThetaTau} give us Proposition \ref{proposition:lowerboundHBI}.

\section{Upper bound on $\operatorname{ker}(\Theta^*)$ for $\mathcal{H}_B \mathcal{I}_g^1$} \label{section:upperboundtheta}

If a module is in the image of the induced map on homology $$\Theta_*: H_2(\mathcal{H}_B \mathcal{I}_g^1; \Q) \rightarrow H_2(W; \Q),$$ then its dual is not in $\operatorname{ker}(\Theta^*)$. In this section, we finish the proof of Theorem \ref{theorem:thetaopen} by showing that $\operatorname{im}(\Theta_*)$ contains all modules in the decomposition of $H_2(W; \Q)$ except for the duals of: \begin{itemize}
    \item the modules we found in $\operatorname{ker}(\Theta^*)$ in Section \ref{section:lowerboundtheta}; and
    \item the module \[ \begin{cases}
    \Phi_{0,2} & g=3 \\
    \Phi_{0,1,0, \dots, 0,1} & g \geq 4
\end{cases}, \] which we conjecture is in $\operatorname{ker}(\Theta^*)$.
\end{itemize} Following the methods of \cite{Hain}, \cite{Sakasai1}, and \cite{Brendle}, we use abelian cycles to find vectors in $\operatorname{im}(\Theta_*)$.

\subsection{Abelian cycles} \label{subsection:abeliancycles} If $f,h \in \mathcal{H}_B \mathcal{I}_g^1$ commute, there is a map $i: \Z^2 \rightarrow \mathcal{H}_B \mathcal{I}_g^1$ sending the standard generators of $\Z^2$ to $f$ and $h$. This induces a map on homology $$i_*: H_2(\Z^2; \Q) \rightarrow H_2(\mathcal{H}_B \mathcal{I}_g^1; \Q).$$ The image of the fundamental class $1$ of $H_2(\Z^2; \Q) \cong \Q$ is the \textit{abelian cycle} $$\{ f,h \} := i_*(1).$$ The upshot is that commuting elements of $\mathcal{H}_B \mathcal{I}_g^1$ give us (possibly trivial) elements of $H_2(\mathcal{H}_B \mathcal{I}_g^1; \Q)$. To find elements in $\operatorname{im}(\Theta_*)$, we apply $\Theta_*$ to abelian cycles. By \cite[Lemma 2.1]{Sakasai1}, this calculation is $$\Theta_*(\{f,h \}) = \Theta^{\Q}(f) \wedge \Theta^{\Q}(h) \in \bigwedge^2 W_{\Q} \cong H_2(W; \Q).$$
 
\subsection{Module decomposition} We decompose $\bigwedge^2 W_{\Q} \cong H_2(W; \Q)$ as an $\operatorname{SL}_g(\Q)$ representation. We have $$H_2(W; \Q) \cong \bigwedge^2 [((\bigwedge^2 V_{\Q}) \otimes V_{\Q}^*) \oplus \operatorname{Sym}^2(V_{\Q}^*)],$$ which expands as $$\bigwedge^2 ((\bigwedge^2 V_{\Q}) \otimes V_{\Q}^*) \oplus [((\bigwedge^2 V_{\Q}) \otimes V_{\Q}^*) \otimes \operatorname{Sym}^2(V_{\Q}^*)] \oplus \bigwedge^2 \operatorname{Sym}^2(V_{\Q}^*).$$ Each of these three pieces decomposes further into irreducible $\operatorname{SL}_g(\Q)$ representations, and we address each piece separately in Sections \ref{subsection:wedge2imJ}, \ref{subsection:imJtensorimPsi}, and \ref{subsection:wedge2sym2}. 

\subsection{The module $\boldsymbol{\bigwedge^2 ((\bigwedge^2 V_{\Q}) \otimes V_{\Q}^*)}$} \label{subsection:wedge2imJ} In this section, we prove the following proposition.

\begin{proposition} \label{proposition:wedge2J}
    When $g \geq 3$, the modules \[ \begin{cases} 
      \Phi_{1,0} \oplus \Phi_{2,1}  & g = 3 \\
      \Phi_{0, \dots, 0, 1,0} \oplus \Phi_{1,0, \dots, 0,1,1} & g \geq 4
   \end{cases}
\] are the only modules in $\bigwedge^2 ((\bigwedge^2 V_{\Q}) \otimes V_{\Q}^*)^* \subset H^2(W; \Q)$ that lie in $\operatorname{ker}(\Theta^*)$.
\end{proposition}

To prove Proposition \ref{proposition:wedge2J}, we decompose $\bigwedge^2 ((\bigwedge^2 V_{\Q}) \otimes V_{\Q}^*) \subset H_2(W; \Q)$ into a direct sum of irreducible $\operatorname{SL}_g(\Q)$ modules and find all modules except for \[ \begin{cases} 
      \Phi_{0,1} \oplus \Phi_{1,2}  & g = 3 \\
      \Phi_{0,1,0, \dots, 0} \oplus \Phi_{1,1,0, \dots, 0,1} & g \geq 4 
   \end{cases} \] (the duals of those in the statement of Proposition \ref{proposition:wedge2J}) in the image of the induced map on homology $$\Theta_*: H_2(\mathcal{H}_B \mathcal{I}_g^1; \Q) \rightarrow H_2(W; \Q).$$ When $g=3$, we have the decomposition $$\bigwedge^2 ((\bigwedge^2 V_{\Q}) \otimes V_{\Q}^*) \cong 2 \Phi_{0,1} \oplus 2 \Phi_{1,2}.$$ Tables \ref{table:bigwedge2Jgenus4}, \ref{table:bigwedge2Jgenus5}, \ref{table:bigwedge2Jgenus6}, and \ref{table:bigwedge2Jgenus7} contain decompositions of $\bigwedge^2 ((\bigwedge^2 V_{\Q}) \otimes V_{\Q}^*) \subset H_2(W; \Q)$ for all $g \geq 4$.

\begin{center}
% --- First Row ---
\begin{minipage}{0.45\textwidth}
\centering
\begin{tabular}{|c|c|}
\hline
\textbf{Highest Weight} & \textbf{Multiplicity} \\
\hline
$(0,0,2)$ & 1 \\
$(0,1,0)$ & 3 \\
$(0,3,0)$ & 1 \\
$(1,0,3)$ & 1 \\
$(1,1,1)$ & 2 \\
\hline
\end{tabular}
\captionof{table}{$g = 4$}
\label{table:bigwedge2Jgenus4}
\end{minipage}
\hspace{0.05\textwidth}
\begin{minipage}{0.45\textwidth}
\centering
\begin{tabular}{|c|c|}
\hline
\textbf{Highest Weight} & \textbf{Multiplicity} \\
\hline
$(0,0,1,1)$ & 2 \\
$(0,1,0,0)$ & 3 \\
$(0,2,1,0)$ & 1 \\
$(1,0,1,2)$ & 1 \\
$(1,1,0,1)$ & 2 \\
\hline
\end{tabular}
\captionof{table}{$g = 5$}
\label{table:bigwedge2Jgenus5}
\end{minipage}

\vspace{1em}

% --- Second Row ---
\begin{minipage}{0.45\textwidth}
\centering
\begin{tabular}{|c|c|}
\hline
\textbf{Highest Weight} & \textbf{Multiplicity} \\
\hline
$(0,0,0,2,0)$ & 1 \\
$(0,0,1,0,1)$ & 2 \\
$(0,1,0,0,0)$ & 3 \\
$(0,2,0,1,0)$ & 1 \\
$(1,0,1,0,2)$ & 1 \\
$(1,1,0,0,1)$ & 2 \\
\hline
\end{tabular}
\captionof{table}{$g = 6$}
\label{table:bigwedge2Jgenus6}
\end{minipage}
\hspace{0.05\textwidth}
\begin{minipage}{0.45\textwidth}
\centering
\begin{tabular}{|c|c|}
\hline
\textbf{Highest Weight} & \textbf{Multiplicity} \\
\hline
$(0,0,0,1,0,\dots,0,1,0)$ & 1 \\
$(0,0,1,0,\dots,0,1)$ & 2 \\
$(0,1,0,\dots,0)$ & 3 \\
$(0,2,0,\dots,0,1,0)$ & 1 \\
$(1,0,1,0,\dots,0,2)$ & 1 \\
$(1,1,0,\dots,0,1)$ & 2 \\
\hline
\end{tabular}
\captionof{table}{$g \geq 7$}
\label{table:bigwedge2Jgenus7}
\end{minipage}
\end{center}

To introduce a lemma, we define the following vectors in $\bigwedge^2 ((\bigwedge^2 V_{\Q}) \otimes V_{\Q}^*)$:

\begin{itemize}
    \item $v_1 := [(a_1 \wedge a_2) \otimes b_2 + (a_1 \wedge a_g) \otimes b_g] \wedge [(a_2 \wedge a_g) \otimes b_g]$
    \item $v_2 := [(a_1 \wedge a_3) \otimes b_3] \wedge [(a_2 \wedge a_g) \otimes b_g]$
    \item $v_3 := [(a_1 \wedge a_2) \otimes b_g] \wedge [(a_3 \wedge a_{g-1}) \otimes b_{g-1}]$
\end{itemize}

\begin{lemma}[{Pettet, \cite[Section 3.3]{Pettet}}] \label{lemma:pettet} When $g=3$, every module in $\bigwedge^2 ((\bigwedge^2 V_{\Q}) \otimes V_{\Q}^*)$ except for $\Phi_{0,1} \oplus \Phi_{1,2}$ is in the $\operatorname{SL}_g(\Q)$ orbit of $v_1$.

When $g=4$, every module in $\bigwedge^2 ((\bigwedge^2 V_{\Q}) \otimes V_{\Q}^*)$ except for $\Phi_{0,1,0} \oplus \Phi_{1,1,1}$ is in the $\operatorname{SL}_g(\Q)$ orbit of $v_1$ and $v_2$.

When $g \geq 5$, every module in $\bigwedge^2 ((\bigwedge^2 V_{\Q}) \otimes V_{\Q}^*)$ except for $\Phi_{0,1,0, \dots, 0} \oplus \Phi_{1,1,0, \dots, 0,1}$ is in the $\operatorname{SL}_g(\Q)$ orbit of $v_1$ and $v_2$ and $v_3$.
\end{lemma}

The modules not included in the orbits in Lemma \ref{lemma:pettet} are exactly the duals of the modules we found in $\operatorname{ker}(\Theta^*)$ in Proposition \ref{proposition:kernelj2}. We are now ready to prove Proposition \ref{proposition:wedge2J}.

\begin{proof}[Proof of Proposition \ref{proposition:wedge2J}]

Since $\operatorname{im}(\tau_*)$ is closed under the action of $\operatorname{SL}_g(\Q)$, by Lemma \ref{lemma:pettet} it suffices to find the vectors \[ \begin{cases}
    v_1 & g=3 \\
    v_1, v_2 & g=4 \\
    v_1, v_2, v_3 & g \geq 5
\end{cases} \] in $\operatorname{im}(\Theta_*)$. We do this by applying the map $\Theta_*$ to abelian cycles. \\

When $g \geq 3$, we show that $v_1 \in \operatorname{im}(\Theta_*)$. Consider the commuting bounding pair annulus twists $T_x T_{x'}^{-1}$ and $T_y T_{y'}^{-1}$. 

\centerline{\includegraphics[scale=.5]{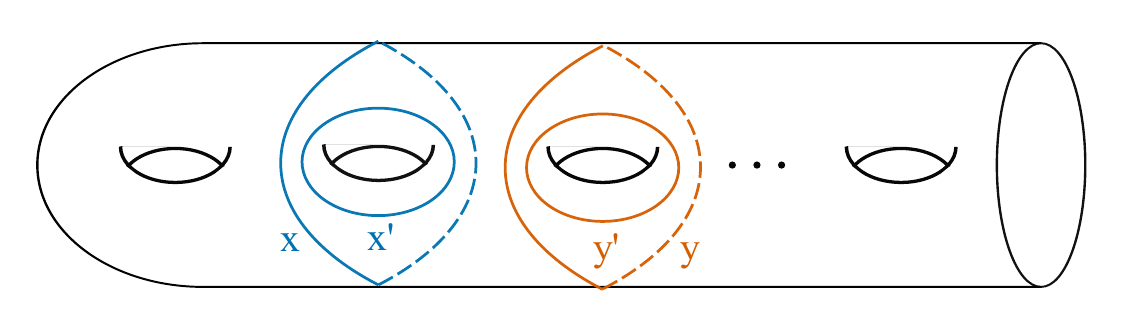}} \noindent 
Since $T_x T_{x'}^{-1}$ and $T_y T_{y'}^{-1}$ are bounding pair maps, they lie in the smaller handlebody Torelli group $\mathcal{HI}_g^1$ and the symplectic representation $\Psi$ kills both of them. It follows that $\Theta$ maps each to the $(\bigwedge^2 V) \otimes V^*$ summand of $W$. The map $T_x T_{x'}^{-1}$ acts on $\pi_1(\mathcal{V}_g^1)$ by conjugating $\alpha_2 \rightarrow \alpha_1 \alpha_2 \alpha_1^{-1}$ and fixing all other generators, acting via the element $C_{21} \in IA_g$. The map $T_y T_{y'}^{-1}$ conjugates $\alpha_1 \rightarrow \alpha_3 \alpha_1 \alpha_3^{-1}$ and $\alpha_2 \rightarrow \alpha_3 \alpha_2 \alpha_3^{-1}$, acting via the element $C_{32} C_{31} \in IA_g$. The Johnson homomorphism $J$ evaluates on these elements as $$J(T_x T_{x'}^{-1}) = -[(a_2 \wedge a_1) \otimes b_1]$$ and $$J(T_y T_{y'}^{-1}) = -[(a_3 \wedge a_2) \otimes b_2 + (a_3 \wedge a_1) \otimes b_1]$$ (recall Section \ref{subsection:firstjohnsonIA}). It follows that, as an element of $\bigwedge^2 ((\bigwedge^2 V_{\Q}) \otimes V_{\Q}^*) \subset H_2(W; \Q)$, the image of the abelian cycle $\{ T_x T_{x'}^{-1}, T_{y} T_{y'}^{-1} \}$ is $$\Theta_*(\{T_x T_{x'}^{-1}, T_y T_{y'}^{-1} \}) = [(a_2 \wedge a_1) \otimes b_1] \wedge [(a_3 \wedge a_2) \otimes b_2 + (a_3 \wedge a_1) \otimes b_1].$$ This differs from $v_1$ by the action of a handleswap in $\mathcal{H}_g^1$, which permutes the indices. Since $\operatorname{im}(\Theta_*)$ is preserved by the action of $\mathcal{H}_g^1$, we have shown that $v_1 \in \operatorname{im}(\Theta_*)$ when $g \geq 3$. \\

When $g \geq 4$, we show that $v_2 \in \operatorname{im}(\Theta_*)$. Consider the commuting bounding pair annulus twists $T_w T_{w'}^{-1}$ and $T_z T_{z'}^{-1}$. 

\centerline{\includegraphics[scale=.35]{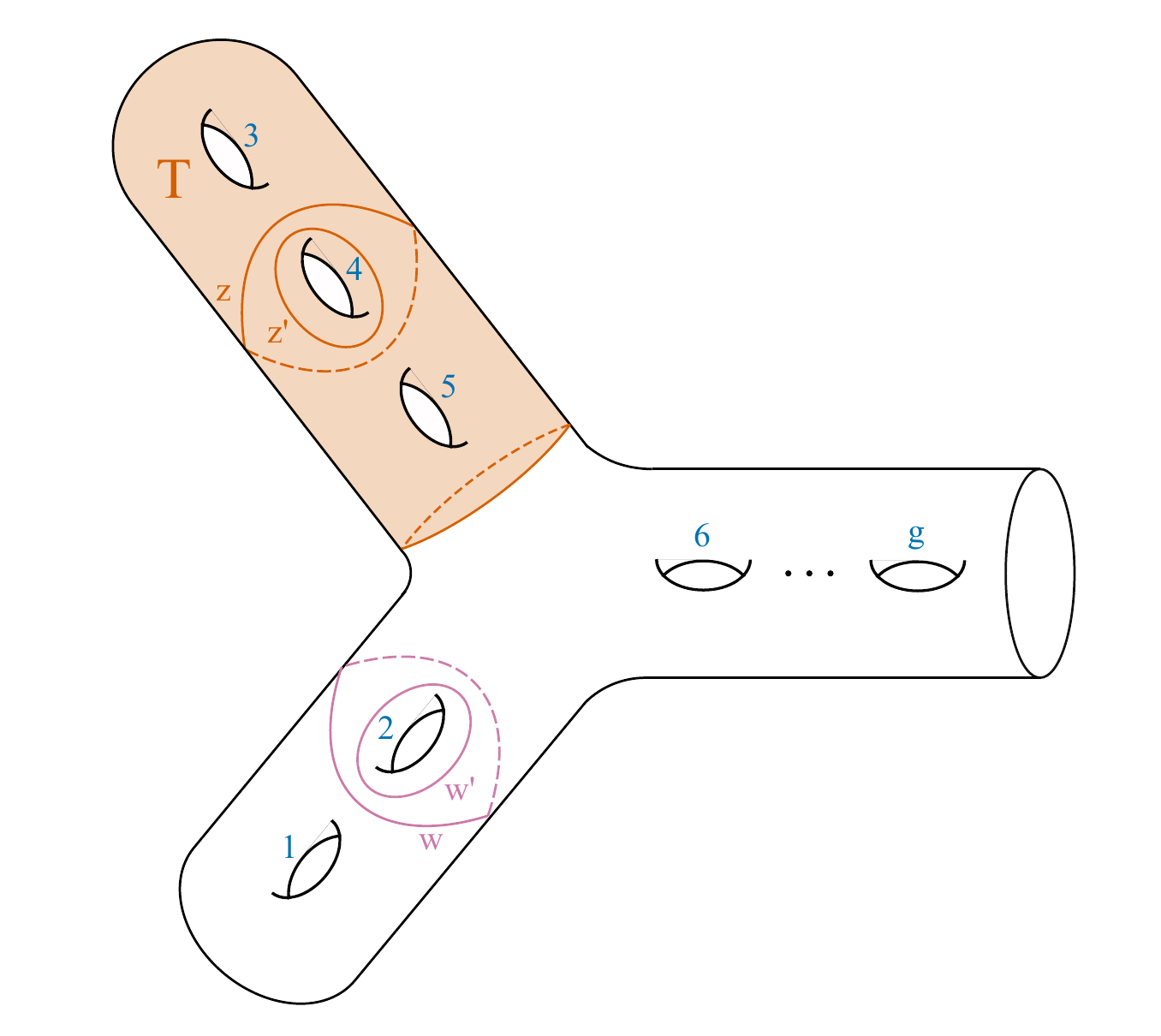}} \noindent These maps act on $\pi_1(\mathcal{V}_g^1)$ via the elements $C_{21}$ and $C_{43}$, respectively. As an element of $\bigwedge^2 ((\bigwedge^2 V_{\Q}) \otimes V_{\Q}^*) \subset H_2(W; \Q)$, the image of the abelian cycle $\{ T_w T_{w'}^{-1}, T_{z} T_{z'}^{-1} \}$ is $$\Theta_*(\{T_w T_{w'}^{-1}, T_z T_{z'}^{-1} \}) = [(a_2 \wedge a_1) \otimes b_1] \wedge [(a_4 \wedge a_3) \otimes b_3],$$ which differs from $v_2$ by the action of a handleswap in $\mathcal{H}_g^1$. \\

When $g \geq 5$, we show that $v_3 \in \operatorname{im}(\Theta_*)$. Following arguments as above, we can find bounding pair maps that act on $\pi_1(\mathcal{V}_g^1)$ via $C_{34}$ and $C_{35}$, and we can arrange for these maps to lie in the submanifold $T \subset \mathcal{V}_g^1$. The commutator $[C_{34}, C_{35}]$ acts on $\pi_1(\mathcal{V}_g^1)$ by sending $\alpha_3 \rightarrow \alpha_3 [\alpha_4, \alpha_5]$ and fixing all other generators. In the notation of Section \ref{subsection:autfn}, we have $[C_{34}, C_{35}] = M_{345}$ and so $$J([C_{34}, C_{35}]) = - (a_3 \wedge a_4) \otimes b_5.$$ Letting $f$ be an element of $\mathcal{HI}_g^1$ that is supported in $T$ (and so commutes with 
$T_w T_{w'}^{-1}$) and acts via $[C_{34}, C_{35}]$, we evaluate $$\Theta_*(\{ T_w T_{w'}^{-1}, f \}) = [(a_2 \wedge a_1) \otimes b_1] \wedge [(a_3 \wedge a_4) \otimes b_5].$$ This differs from $v_3$ by the action of a handleswap in $\mathcal{H}_g^1$. \end{proof}

\subsection{The module $\boldsymbol{((\bigwedge^2 V_{\Q}) \otimes V_{\Q}^*) \otimes \operatorname{Sym}^2(V_{\Q}^*)}$} \label{subsection:imJtensorimPsi} In this section, we prove the following proposition.

\begin{proposition} \label{proposition:imJtensorimPsi} In the $((\bigwedge^2 V_{\Q}) \otimes V_{\Q}^*)^* \otimes (\operatorname{Sym}^2(V_{\Q}^*))^*$ summand of $H^2(W; \Q)$, the modules \[ \begin{cases} 
      \Phi_{0,2} \oplus \Phi_{1,0} & g = 3 \\
      \Phi_{0,1,0, \dots, 0,1} \oplus \Phi_{1,0, \dots, 0} & g \geq 4
   \end{cases} \] are the only irreducible $\operatorname{SL}_g(\Q)$ modules that can be in $\operatorname{ker}(\Theta^*)$ when $g \geq 3$.
\end{proposition}

To prove this, we decompose $((\bigwedge^2 V_{\Q}) \otimes V_{\Q}^*) \otimes \operatorname{Sym}^2(V_{\Q}^*) \subset H_2(W; \Q)$ and find all modules {\textit{except}} for \[ \begin{cases} 
    \Phi_{0,1} \textnormal{{ and }} \Phi_{2,0} & g=3 \\
    \Phi_{0, \dots, 0,1} \textnormal{{ and }} \Phi_{1,0, \dots, 0, 1,0} & g \geq 4 \end{cases} \] (the duals of those in the statement of Proposition \ref{proposition:imJtensorimPsi}) in the image of the induced map $$\Theta_*: H_2(\mathcal{H}_B \mathcal{I}_g^1; \Q) \rightarrow H_2(W; \Q).$$ Tables \ref{table:JtensorPhiGenus3}, \ref{table:JtensorPhiGenus4}, and \ref{table:JtensorPhiGenus5} contain decompositions of $((\bigwedge^2 V_{\Q}) \otimes V_{\Q}^*) \otimes \operatorname{Sym}^2(V_{\Q}^*) \subset H_2(W; \Q)$ for all $g \geq 3$.

\begin{table}[h!]
  \centering
  \begin{minipage}{0.45\textwidth}
    \centering
    \begin{tabular}{|c|c|}
      \hline
      \textbf{Highest Weight} & \textbf{Multiplicity} \\
      \hline
      $(0,1)$ & 1 \\
      $(0,4)$ & 1 \\
      $(1,2)$ & 2 \\
      $(2,0)$ & 1 \\
      \hline
    \end{tabular}
    \caption{$g=3$}
    \label{table:JtensorPhiGenus3}
  \end{minipage}%
  \hspace{0.05\textwidth}
  \begin{minipage}{0.45\textwidth}
    \centering
    \begin{tabular}{|c|c|}
      \hline
      \textbf{Highest Weight} & \textbf{Multiplicity} \\
      \hline
      $(0,0,1)$ & 1 \\
      $(0,1,3)$ & 1 \\
      $(0,2,1)$ & 1 \\
      $(1,0,2)$ & 2 \\
      $(1,1,0)$ & 1 \\
      \hline
    \end{tabular}
    \caption{$g=4$}
    \label{table:JtensorPhiGenus4}
  \end{minipage}
\end{table}

\vspace{0.5em}

\begin{table}[h!]
  \centering
  \begin{tabular}{|c|c|}
    \hline
    \textbf{Highest Weight} & \textbf{Multiplicity} \\
    \hline
    $(0, \dots, 0,1)$ & 1 \\
    $(0,1,0, \dots, 0,3)$ & 1 \\
    $(0,1,0, \dots, 0,1,1)$ & 1 \\
    $(1,0, \dots, 0,2)$ & 2 \\
    $(1,0, \dots, 0,1,0)$ & 1 \\
    \hline
  \end{tabular}
  \caption{$g \geq 5$}
  \label{table:JtensorPhiGenus5}
\end{table}

In the following lemma, we find types of vectors in $\operatorname{im}(\Theta_*)$ using abelian cycles.

\begin{lemma} \label{lemma:JtensorPsiTypes} Let $1 \leq i,j,k \leq g$ be distinct. Vectors of the following types are in $\operatorname{im}(\Theta_*)$: \textnormal{\begin{align*}  
\text{Type 1:} \quad & [(a_i \wedge a_j) \otimes b_i] \otimes [b_i \otimes b_i] \quad (g \geq 2) \\
\text{Type 2:} \quad & [(a_i \wedge a_j) \otimes b_i] \otimes [b_k \otimes b_k] \quad (g \geq 3) \\
\text{Type 3:} \quad & [(a_i \wedge a_j) \otimes b_k] \otimes [b_k \otimes b_k] \quad (g \geq 3) \\
\text{Type 4:} \quad & [(a_i \wedge a_j) \otimes b_k] \otimes [b_{\ell} \otimes b_{\ell}] \quad (g \geq 4).
\end{align*}}

\end{lemma}
 
\begin{proof} Consider the bounding pair annulus twist $T_x T_{x'}^{-1}$ and the disk twists $T_{\mathfrak{b}_1}$ and $T_{\mathfrak{b}_3}$. Note that $x$ and $x'$ both represent the homology class $a_2$, and $\mathfrak{b}_1$ and $\mathfrak{b}_3$ represent the homology classes $b_1$ and $b_3$, respectively. 

\centerline{\includegraphics[scale=.5]{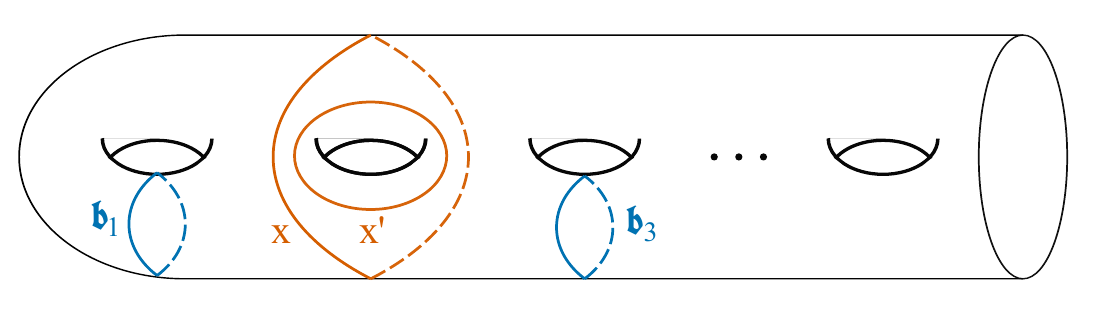}}  When $i \neq 2$, we evaluate: $$\Theta_*(\{ T_x T_{x'}^{-1}, T_{\mathfrak{b}_i} \} ) = [(a_1 \wedge a_2) \otimes b_1] \otimes [b_i \otimes b_i].$$ We obtain all Type 1 and Type 2 vectors by choosing the correct $i \in \{1,3\}$ and then permuting indices by applying handleswaps. Since representations of $\operatorname{SL}_g(\Q)$ are in natural bijection with those of $\mathfrak{sl}_g(\Q)$, we can view $\operatorname{im}(\Theta_*)$ as an $\mathfrak{sl}_g(\Q)$ representation. We apply $\mathfrak{sl}_g(\Q)$ actions to Type 2 vectors to find Type 3 and Type 4 vectors in $\operatorname{im}(\Theta_*)$:

\[
-E_{ik} \underbrace{[(a_i \wedge a_j) \otimes b_i] \otimes [b_k \otimes b_k]}_{\text{Type 2}} = \underbrace{[(a_i \wedge a_j) \otimes b_k] \otimes [b_k \otimes b_k]}_{\text{Type 3}}.  \]

\[
-E_{ik} \underbrace{[(a_i \wedge a_j) \otimes b_i] \otimes [b_{\ell} \otimes b_{\ell}]}_{\text{Type 2}} = \underbrace{[(a_i \wedge a_j) \otimes b_k] \otimes [b_{\ell} \otimes b_{\ell}]}_{\text{Type 4}}.  \]

\end{proof}

Now we are ready to prove Proposition \ref{proposition:imJtensorimPsi}. In this section, we define the $\mathfrak{sl}_g(\Q)$-equivariant map $p: (\bigwedge^2 V_{\Q}) \otimes V_{\Q}^* \rightarrow \bigwedge^3 H_{\Q}$ which sends $a_i \wedge a_j \otimes b_k \mapsto a_i \wedge a_j \wedge b_k$. Recalling that we defined contractions $C_k$ in Section \ref{section:reptheorymaps}, the composition $C_3 \circ p$ sends $$a_i \wedge a_j \otimes b_k \mapsto \omega(a_j, b_k) a_i - \omega(a_i, b_k)a_j.$$ 

\begin{proof}[Proof of Proposition \ref{proposition:imJtensorimPsi}] We find all of the necessary modules in $\operatorname{im}(\Theta_*)$.  \\

$\boxed{(0,1,0, \dots ,0,3)}$: The Type 3 vector $$[(a_1 \wedge a_2) \otimes b_g] \otimes [b_g \otimes b_g]$$ is a highest weight vector for \[ \begin{cases} 
\Phi_{0,4} & g = 3 \\
\Phi_{0,1,0, \dots ,0,3} & g \geq 4 \end{cases} . \]

\noindent By Lemma \ref{lemma:JtensorPsiTypes}, we have $\Phi_{0,4} \subset \operatorname{im}(\Theta_*)$ when $g=3$ and $\Phi_{0,1,0, \dots ,0,3} \subset \operatorname{im}(\Theta_*)$ when $g \geq 4$. \\

$\boxed{(0,1,0, \dots ,0,1,1)}$: When $g \geq 4$, the vector 
$$[(a_1 \wedge a_2) \otimes b_g] \otimes [b_{g-1} \otimes b_g + b_g \otimes b_{g-1}] - 2[(a_1 \wedge a_2) \otimes b_{g-1}] \otimes [b_g \otimes b_g]$$ is a highest weight vector for \[ \begin{cases} 
\Phi_{0,2,1} & g = 4 \\
\Phi_{0,1,0, \dots ,0,1,1} & g \geq 5 \end{cases} . \]

Since $$-2[(a_1 \wedge a_2) \otimes b_{g-1}] \otimes [b_g \otimes b_g]$$ is a Type 4 vector, it is in $\operatorname{im}(\Theta_*)$. It remains to show that $$[(a_1 \wedge a_2) \otimes b_g] \otimes [b_{g-1} \otimes b_g + b_g \otimes b_{g-1}]$$ is in $\operatorname{im}(\Theta_*)$, and we show this by applying $\mathfrak{sl}_g(\Q)$ actions to a Type 3 vector. We calculate: 

\hspace*{-1.3cm}
\begin{minipage}{\textwidth} \begin{align*}
    -E_{g(g-1)}\underbrace{[(a_1 \wedge a_2) \otimes b_g] \otimes [b_g \otimes b_g]}_{\text{Type 3}} = \underbrace{[(a_1 \wedge a_2) \otimes b_{g-1}] \otimes [b_g \otimes b_g]}_{\text{Type 4}} + [(a_1 \wedge a_2) \otimes b_g] \otimes [b_{g-1} \otimes b_g + b_g \otimes b_{g-1}].
\end{align*} \end{minipage} Since the Type 4 vector belongs to $\operatorname{im}(\Theta_*)$, it follows that $[(a_1 \wedge a_2) \otimes b_g] \otimes [b_{g-1} \otimes b_g + b_g \otimes b_{g-1}]$ is in $\operatorname{im}(\Theta_*)$ as well. \\

$\boxed{(1,0, \dots ,0,2)}$: We find two copies of $\Phi_{1,0, \dots ,0,2}$ in $\operatorname{im}(\Theta_*)$ when $g \geq 3$. First, $$ ((C_3 \circ p) \otimes \text{id}) \underbrace{[(a_1 \wedge a_2) \otimes b_2] \otimes [b_g \otimes b_g]}_{\text{Type 2}} = a_1 \otimes b_g \otimes b_g$$ is a highest weight vector for $\Phi_{1,0, \dots ,0,2}$.  \\

The second copy of $\Phi_{1,0, \dots ,0,2}$ has highest weight vector $$\sum_{i=2}^{g-1} ([(a_1 \wedge a_i) \otimes b_g] \otimes [b_i \otimes b_g + b_g \otimes b_i] - 2[(a_1 \wedge a_i)\otimes b_i] \otimes [b_g \otimes b_g]).$$ Note that $(C_3 \circ p) \otimes \text{id}$ kills this vector so we are indeed detecting a second copy of $\Phi_{1,0, \dots ,0,2}$. We rewrite our highest weight vector as $$\sum_{i=2}^{g-1} (x_i + y_i)$$ for $x_i := [(a_1 \wedge a_i) \otimes b_g] \otimes [b_i \otimes b_g + b_g \otimes b_i]$ and $y_i := \underbrace{-2[(a_1 \wedge a_i) \otimes b_i] \otimes [b_g \otimes b_g]}_{\text{Type 2}}$. It remains to show that $x_i \in \operatorname{im}(\Theta_*)$ for $i \in \{2, \dots, g-1 \}$. We calculate 
    \[ -E_{gi} \underbrace{[(a_1 \wedge a_i) \otimes b_g] \otimes [b_g \otimes b_g]}_{\text{Type 3}} \] \begin{align*} = \underbrace{-[(a_1 \wedge a_g) \otimes b_g] \otimes [b_g \otimes b_g]}_{\text{Type 1}} + \underbrace{[(a_1 \wedge a_i) \otimes b_i] \otimes [b_g \otimes b_g]}_{\text{Type 2}} + \underbrace{[(a_1 \wedge a_i) \otimes b_g] \otimes [b_i \otimes b_g + b_g \otimes b_i]}_{x_i}.
\end{align*} 

\end{proof}

\subsection{The module $\boldsymbol{\bigwedge^2 \operatorname{Sym}^2(V_{\Q}^*)}$} \label{subsection:wedge2sym2} In this subsection, we prove the following proposition. 

\begin{proposition} \label{proposition:sym2notinkernel}
    When $g \geq 3$, the irreducible module $\bigwedge^2 \operatorname{Sym}^2(V_{\Q}^*)^* \subset H^2(W; \Q)$ is not in $\operatorname{ker}(\Theta^*)$.
\end{proposition} See Equation \ref{equation:bigwedge2W}. Proposition \ref{proposition:sym2notinkernel} will follow from Proposition \ref{proposition:sym2surjective}, which we now introduce.

\begin{proposition} \label{proposition:sym2surjective}
    When $g \geq 3$, the image of the map $$\Theta_*: H_2(\mathcal{H}_B \mathcal{I}_g^1; \Q) \rightarrow H_2(W; \Q)$$ contains $\bigwedge^2 \operatorname{Sym}^2(V_{\Q}^*) \subset H_2(W; \Q)$.
\end{proposition}

\begin{proof}[Proof of Proposition \ref{proposition:sym2surjective}] The representation $\bigwedge^2 \operatorname{Sym}^2(V_{\Q}^*) \cong \Phi_{0, \dots, 0,1,2}$ is irreducible, and we will find a highest weight vector with the desired weight in $\operatorname{im}(\Theta_*)$. Let $\mathfrak{b}_{g-1}$ and $\mathfrak{b}_g$ be meridians whose homology class are $b_{g-1}$ and $b_g$, respectively, as in the figure in Section \ref{subsection:conventions}. When $g \geq 3$, we calculate that \begin{align*} E_{(g-1)g} \circ \Theta_*(\{ T_{\mathfrak{b}_g}, T_{\mathfrak{b}_{g-1}} \}) = E_{(g-1)g} [b_g \otimes b_g] \wedge [b_{g-1} \otimes b_{g-1}] = - [b_g \otimes b_g] \wedge [b_g \otimes b_{g-1} + b_{g-1} \otimes b_g] \end{align*} is a highest weight vector for $\Phi_{0, \dots, 0,1,2}$. \end{proof}

Propositions \ref{proposition:lowerboundHBI}, \ref{proposition:wedge2J}, \ref{proposition:imJtensorimPsi}, and \ref{proposition:sym2notinkernel} combine to prove Theorem \ref{theorem:thetaopen}.

\section{Determination of $\operatorname{ker}(\Theta^*)$ for $\mathcal{H}_B \mathcal{I}_g$} \label{section:thetaclosed}

Now we consider the undecorated case $\mathcal{H}_B \mathcal{I}_g$. This group acts on $\pi_1(\mathcal{V}_g) \cong F_g$ via outer automorphisms and acts trivially on $V := H_1(\mathcal{V}_g; \Z)$, giving us the surjective map to the Torelli subgroup of $\operatorname{Out}(F_g)$ $$\alpha: \mathcal{H}_B \mathcal{I}_g \rightarrow OA_g.$$ From this, we get the surjective Johnson homomorphism $$J: \mathcal{H}_B \mathcal{I}_g \rightarrow ((\bigwedge^2 V) \otimes V^*) / V.$$ The groups $\mathcal{H}_B \mathcal{I}_g^1$ and $\mathcal{H}_B \mathcal{I}_g$ act on surface homology $H$ in the same way so our previous calculations with the symplectic representation $\Psi$ do not change. Following our work in Section \ref{subsection:W}, we construct the abelian quotient map $$\Theta: \mathcal{H}_B \mathcal{I}_g \rightarrow (((\bigwedge^2 V) \otimes V^*) / V) \oplus \operatorname{Sym}^2(V^*) =: \overline{W}.$$ We have the following theorem. 

\begin{theorem} \label{theorem:ThetaClosed}
    When $g \geq 3$, the kernel of the induced map $\Theta^*: H^2(\overline{W}; \Q) \rightarrow H^2(\mathcal{H}_B \mathcal{I}_g; \Q)$ is either: 
    \begin{itemize}
        \item the module  $\begin{cases}
    \Phi_{2,1} & g=3 \\
    \Phi_{1,0, \dots, 0,1,1} & g \geq 4
\end{cases};$ or \vspace{2mm}
    \item the direct sum $ \begin{cases}
    \Phi_{0,2} \oplus \Phi_{2,1} & g=3 \\
    \Phi_{0,1,0, \dots, 0,1} \oplus \Phi_{1,0, \dots, 0,1,1} & g \geq 4
\end{cases}. $
    \end{itemize}
\end{theorem}

In particular, $\operatorname{ker}(\Theta^*)$ contains the module $\Phi_{1,0, \dots, 0,1,1}$ (resp. $\Phi_{2,1}$) when $g \geq 4$ (resp. $g=3$), and contains at most one additional module. Since the arguments closely follow those used for $\mathcal{H}_B \mathcal{I}_g^1$, we only sketch the proof of Theorem \ref{theorem:ThetaClosed} and assume familiarity with the notation and methods from Sections \ref{section:lowerboundtheta} and \ref{section:upperboundtheta}. 

To show that $\Phi_{1,0, \dots, 0,1,1}$ (resp. $\Phi_{2,1}$) belongs to $\operatorname{ker}(\Theta^*)$ when $g \geq 4$ (resp. $g=3$), we show that its dual module is in $\operatorname{coker}(\Theta_*)$. As we did in Section \ref{subsection:5tesHBI}, we construct a five term exact sequence that contains a map $b$ with $\operatorname{coker}(\Theta_*) \cong \operatorname{im}(b)$. We then detect the desired module in $\operatorname{im}(b)$ by composing $b$ with the second Johnson homomorphism $J_2^{\Q}$ for $OA_g(2)$ as we did in Section \ref{subsection:lowerboundthetaJ2}. In this case, the Johnson homomorphism $\tau$ does not detect any modules in $\operatorname{ker}(\Theta^*)$. 

It follows from our work in Section \ref{section:upperboundtheta} that all remaining modules that survive in the quotient $\bigwedge^2 \overline{W}_{\Q}^*$ of $\bigwedge^2 W_{\Q}^*$ are not in $\operatorname{ker}(\Theta^*)$, except for possibly the module $\Phi_{0,1,0, \dots, 0,1}$ (resp. $\Phi_{0,2}$) when $g \geq 4$ (resp. $g=3$). \\

Theorem \ref{theorem:kerneltheta} follows from Theorem \ref{theorem:thetaopen} and Theorem \ref{theorem:ThetaClosed}.

\section{Setup for $\mathcal{HI}_{g,p}^b$} \label{section:setuphi}

We now study $\mathcal{HI}_{g,p}^b$ which we recall is the intersection $\mathcal{H}_{g,p}^b \cap \mathcal{I}_{g,p}^b$. The Johnson homomorphism $\tau$ captures the rational abelianization of $\mathcal{I}_{g,p}^b$, and we study $\tau|_{\mathcal{HI}_{g,p}^b}$ to get a handle on $H^*(\mathcal{HI}_{g,p}^b; \Q)$. From now on, we use $\tau$ to denote this restriction and $U$ or $\UQ$ to denote its image:
$$\tau: \mathcal{HI}_{g,p}^b \rightarrow \begin{cases}
    \UQ & (p=b=0) \\
    U & (p+b = 1) 
   \end{cases}.$$

\noindent Recall that we fixed a symplectic basis for $H$ in Section \ref{subsection:conventions}. We have the following proposition.
   
\begin{proposition}[{Morita, \cite[Lemma 2.5]{MoritaCasson}}] \label{proposition:U} When $p+b=1$, the image $U := \tau(\mathcal{HI}_{g,p}^b)$ is the submodule of $\bigwedge^3 H$ generated by elements of the following form:  \begin{itemize}
    \item $a_i \wedge b_i \wedge a_j$
    \item $a_i \wedge b_i \wedge b_j$
    \item $a_i \wedge a_j \wedge b_k$
    \item $a_i \wedge b_j \wedge b_k$
    \item $b_i \wedge b_j \wedge b_k$
\end{itemize} \end{proposition} \noindent In particular, there are no ``triple-$a$" terms of the form $a_i \wedge a_j \wedge a_k$. \\

For the undecorated case of $\mathcal{HI}_g$, the image $\bar{U}$ lies inside of $(\bigwedge^3 H) / H$ and the construction follows that from Section \ref{subsection:firstjohnsonTorelli}. The quotient $\UQ$ of $U$ fits into the following commutative diagram with exact rows:

\[\begin{tikzcd}
	1 & {\pi_1(U\Sigma_g)} & {\mathcal{HI}_g^1} & {\mathcal{HI}_g} & 1 \\
	0 & H & U & {\UQ} & 0
	\arrow[from=1-1, to=1-2]
	\arrow[from=1-2, to=1-3]
	\arrow["\tau", from=1-2, to=2-2]
	\arrow[from=1-3, to=1-4]
	\arrow["\tau", from=1-3, to=2-3]
	\arrow[from=1-4, to=1-5]
	\arrow["\tau", dashed, from=1-4, to=2-4]
	\arrow[from=2-1, to=2-2]
	\arrow[from=2-2, to=2-3]
	\arrow[from=2-3, to=2-4]
	\arrow[from=2-4, to=2-5]
\end{tikzcd}\] Therefore, $U \cong \UQ \oplus H$, with the decomposition given explicitly in the following way. Recall that $C_3$ is the contraction defined in Section \ref{section:reptheorymaps}, and define $q: \bigwedge^3 H \rightarrow \bigwedge^3 H$ by $$q(v) = v - \frac{1}{g-1}C_3(v) \wedge \omega.$$ Then the direct sum decomposition of $\UQ \oplus H$ is given by $$v \mapsto (q(v), \frac{1}{g-1} C_3(v) \wedge \omega).$$ See \cite[Section 4]{Sakasai1}. \\

In Theorem \ref{theorem:kerneltau}, we decompose the kernels of $$\tau^*: \begin{cases}
     H^2(\UQ; \Q) \rightarrow H^2(\mathcal{HI}_g; \Q) \\
    H^2(U; \Q) \rightarrow H^2(\mathcal{HI}_{g,p}^b; \Q) & (p+b = 1) 
   \end{cases}$$ as direct sums of $\operatorname{SL}_g(\Q)$ representations. The $\operatorname{SL}_g(\Q)$ action comes from the $\mathcal{H}_{g,p}^b$-equivariance of $\tau$. Recall that $\mathcal{H}_{g,p}^b$ acts on $H$ by the symplectic matrices $$\Psi(\mathcal{H}_{g,p}^b) = \left\{  \begin{bmatrix} A & 0 \\ C & (A^t)^{-1} \end{bmatrix} \; | \; A \in \operatorname{GL}_g(\Z), \; C = C^t \right\}.$$ We then get an action of $\mathcal{H}_{g,p}^b$ on $U \subset \bigwedge^3 H$ (when $p+b=1$) or $\UQ \subset (\bigwedge^3 H) / H$ (when $p=b=0$) through its action on $H$. The handlebody group $\mathcal{H}_{g,p}^b$ also acts on its normal subgroup $\mathcal{HI}_{g,p}^b$ by conjugation, and the Johnson homomorphism $\tau$ is $\mathcal{H}_{g,p}^b$-equivariant in the following sense. For $h \in \mathcal{H}_{g,p}^b$ and $f \in \mathcal{HI}_{g,p}^b$, we have $$\tau(hfh^{-1}) = h_* \tau(f).$$ This action extends to an action of $\overline{\Psi(\mathcal{H}_{g,p}^b)}$, the Zariski closure of $\Psi(\mathcal{H}_{g,p}^b)$, on $U_{\Q}$ and $\UQ_{\Q}$. The algebraic group $\overline{\Psi(\mathcal{H}_{g,p}^b)}$ is not semisimple but does contain the semisimple subgroup $$ \left\{ \begin{bmatrix} A & 0 \\ 0 & (A^t)^{-1} \end{bmatrix} \; | \; A \in \operatorname{SL}_g(\Q) \right\}.$$ By restricting our action to this subgroup, we can view $U_{\Q}$ and $\UQ_{\Q}$ as $\operatorname{SL}_g(\Q)$ representations. In the following sections, we use the representation theory of $\operatorname{SL}_g(\Q)$ to prove Theorem \ref{theorem:kerneltau} as we did for Theorem \ref{theorem:kerneltheta}. We will use the $\operatorname{SL}_g(\Q)$ and $\mathfrak{sl}_g(\Q)$ maps from the previous sections. On top of this, we make use of the entire action of $\overline{\Psi(\mathcal{H}_{g,p}^b)}$. If $\mathfrak{b}_i$ is a curve representing $b_i$ in homology, the Dehn twist $T_{\mathfrak{b}_i}$ acts on $H_{\Q}$ by sending $a_i \rightarrow a_i + b_i$ and fixing all other generators. This gives us the $\mathfrak{sp}_{2g}(\Q)$ action $X_i$ (defined in Section \ref{section:reptheorymaps}) which is not in the subalgebra $\mathfrak{sl}_g(\Q)$. The map $X_i$ sends $a_i \rightarrow b_i$ and kills all other generators. 

\section{Module Decomposition} \label{section:moduledecomp}

Letting $V_{\Q}$ be the standard $\operatorname{SL}_g(\Q)$ representation and $V_{\Q}^*$ be its dual, we have the decomposition $H_{\Q} = V_{\Q} \oplus V_{\Q}^*$. The images $\tau(\mathcal{I}_{g,p}^b)$ decompose as $$\bigwedge^3 H_{\Q} = \bigwedge^3 V_{\Q} \oplus [(\bigwedge^2 V_{\Q}) \otimes V_{\Q}^*] \oplus [V_{\Q} \otimes (\bigwedge^2 V_{\Q}^*)] \oplus \bigwedge^3 V_{\Q}^*$$ and $$(\bigwedge^3 H_{\Q}) / H_{\Q} = \bigwedge^3 V_{\Q} \oplus \Phi_{0,1,0, \dots, 0,1} \oplus \Phi_{1,0, \dots, 0,1,0} \oplus \bigwedge^3 V_{\Q}^*.$$ The $\bigwedge^3 V_{\Q}$ summand contains exactly the ``triple-$a$" terms of the form $a_i \wedge a_j \wedge a_k$. Then, by Proposition \ref{proposition:U}, we decompose $U_{\Q} \subset \bigwedge^3 H_{\Q}$ as $$U_{\Q} = [(\bigwedge^2 V_{\Q}) \otimes V_{\Q}^*] \oplus [V_{\Q} \otimes (\bigwedge^2 V_{\Q}^*)] \oplus \bigwedge^3 V_{\Q}^*.$$ This decomposes into irreducible $\operatorname{SL}_g(\Q)$ modules as $$U_{\Q} = \Phi_{0,1,0, \dots, 0,1} \oplus \Phi_{1,0, \dots, 0} \oplus \Phi_{1,0, \dots, 0,1,0} \oplus \Phi_{0, \dots, 0,1} \oplus \Phi_{0, \dots, 0,1,0,0}.$$ Similarly, $\UQ_{\Q} \subset (\bigwedge^3 H_{\Q}) / H_{\Q}$ decomposes as $$\UQ_{\Q} = \Phi_{0,1,0, \dots, 0,1} \oplus \Phi_{1,0, \dots, 0,1,0} \oplus \Phi_{0, \dots, 0,1,0,0}.$$

Following the ideas in Sections \ref{section:lowerboundtheta} and \ref{section:upperboundtheta}, we will determine the kernel of $$\tau^*: H^2(\UQ; \Q) \rightarrow H^2(\mathcal{HI}_g; \Q)$$ by finding $\operatorname{SL}_g(\Q)$ modules in the cokernel and image of the induced map on homology $$\tau_*: H_2(\mathcal{HI}_g; \Q) \rightarrow H_2(\UQ; \Q).$$ For this reason, we decompose $H_2(\UQ; \Q) \cong \bigwedge^2 \UQ_{\Q}$ and $H_2(U; \Q) \cong \bigwedge^2 U_{\Q}$ in Table \ref{table:decompAlt2U}. This can be done using LiE. See Appendix \ref{appendix:LiE}. To obtain decompositions of $H^2(\UQ; \Q) \cong \bigwedge^2 \UQ_{\Q}^*$ and $H^2(U; \Q) \cong \bigwedge^2 U_{\Q}^*$, take the dual of each irreducible representation in Table \ref{table:decompAlt2U}, where we have $$\Phi_{a,b,c, \dots, x,y,z}^* = \Phi_{z,y,x, \dots, c,b,a}.$$ The decomposition in Table \ref{table:decompAlt2U} stabilizes when $g \geq 6$. 

\begin{table}[htbp]
    \centering
    \begin{tabular}{|l|l|l|}
    \hline
    \multirow{2}{*}{\textbf{Highest Weight}} & \multicolumn{2}{c|}{\textbf{Multiplicity}} \\ 
    \cline{2-3}
    & \rule{0pt}{2.6ex} \textbf{$\bigwedge^2 \UQ_{\Q}$} & \textbf{$\bigwedge^2 U_{\Q}$} \\ \hline
        $(0, \dots ,0)$  & 1 & 2 \\ \hline
        $(0, \dots ,0,2)$ & 1 & 1 \\ \hline
        $(0, \dots ,0,1,0)$ & 2 & 5 \\ \hline
        $(0, \dots ,0,2,0)$ & 1 & 1 \\ \hline
        $(0, \dots ,0,1,0,1)$ & 1 & 2 \\ \hline
        $(0, \dots ,0,1,0,0,0)$  & 1 & 2 \\ \hline
        $(0, \dots ,0,1,0,1,0)$  & 1 & 1 \\ \hline
        $(0, \dots, 0,1,0,0,0,0,0)$ & 1 & 1 \\ \hline
        $(0,0,0,1,0, \dots ,0,1,0)$  & 1 & 1 \\ \hline
        $(0,0,1,0, \dots ,0,1)$ & 1 & 2 \\ \hline
        $(0,0,1,0, \dots ,0,1,1)$  & 1 & 1 \\ \hline
        $(0,0,1,0, \dots ,0,1,0,0)$ & 1 & 1 \\ \hline
        $(0,1,0, \dots ,0)$ & 1 & 3 \\ \hline
        $(0,1,0,\dots,0,2)$  & 1 & 2 \\ \hline
        $(0,1,0, \dots ,0,1,0)$ & 2 & 4 \\ \hline
        $(0,1,0, \dots ,0,2,0)$ & 1 & 1 \\ \hline
        $(0,1,0, \dots ,0,1,0,1)$ & 1 & 1 \\ \hline
        $(0,1,0, \dots ,0,1,0,0,0)$ & 2 & 2 \\ \hline
        $(0,2,0, \dots, 0,1,0)$ & 1 & 1 \\ \hline
        $(1,0, \dots ,0,1)$ & 2 & 5 \\ \hline
        $(1,0, \dots ,0,1,1)$ & 2 & 3 \\ \hline
        $(1,0,\dots ,0,1,0,0)$ & 2 & 4 \\ \hline
        $(1,0, \dots ,0,1,1,0)$ & 1 & 1 \\ \hline
        $(1,0, \dots ,0,1,0,0,1)$ & 1 & 1 \\ \hline
        $(1,0, \dots ,0,1,0,0,0,0)$ & 1 & 1 \\ \hline
        $(1,0,1,0, \dots ,0,2)$ & 1 & 1 \\ \hline
        $(1,1,0, \dots ,0,1)$ & 1 & 2 \\ \hline
        $(1,1,0, \dots ,0,1,1)$ & 1 & 1 \\ \hline
        $(1,1,0, \dots ,0,1,0,0)$ & 1 & 1 \\ \hline
        $(2,0, \dots ,0,2)$ & 1 & 1 \\ \hline
        $(2,0, \dots ,0,1,0)$ & 1 & 2 \\ \hline
        $(2,0, \dots ,0,1,0,1)$ & 1 & 1 \\ \hline
    \end{tabular} \caption{Decomposition of $\bigwedge^2 \UQ_{\Q}$ and $\bigwedge^2 U_{\Q}$} \label{table:decompAlt2U}
\end{table} 

\subsection{Decomposition via branching} \label{subsection:branching} We can organize Table \ref{table:decompAlt2U} by remembering that $\bigwedge^2 \UQ_{\Q}$ is a submodule of $\bigwedge^2 ((\bigwedge^3 H_{\Q}) / H_{\Q})$, which has the structure of an $\operatorname{Sp}_{2g}(\Q)$ module. We have the decomposition into irreducible $\operatorname{Sp}_{2g}(\Q)$ modules $$\bigwedge^2 ((\bigwedge^3 H_{\Q}) / H_{\Q}) = \Gamma_{0,1,0,1} \oplus \Gamma_{0,2} \oplus \Gamma_{0,0,0,0,0,1} \oplus \Gamma_{0,0,0,1} \oplus \Gamma_{0,1} \oplus \mathbb{Q}.$$ See \cite[Lemma 10.2]{Hain} and \cite[Lemma 6.1]{MoritaLinear}. With LiE (see Appendix \ref{appendix:LiE}), we can branch each of these $\operatorname{Sp}_{2g}(\Q)$ modules into a direct sum of $\operatorname{SL}_g(\Q)$ modules. This gives us another way to decompose $\bigwedge^2 \UQ_{\Q}$, which is the submodule of $\bigwedge^2 ((\bigwedge^3 H_{\Q}) / H_{\Q})$ that does not include $$(\bigwedge^2 (\bigwedge^3 V_{\Q})) \oplus [\bigwedge^3 V_{\Q} \otimes (\Phi_{0,1,0, \dots, 0,1} \oplus \Phi_{1,0, \dots, 0,1,0} \oplus \bigwedge^3 V_{\Q}^*)].$$ By branching $\bigwedge^2 ((\bigwedge^3 H_{\Q}) / H_{\Q})$ and excluding the modules listed above, we recover the decomposition in Table \ref{table:decompAlt2U}. \\

This approach is useful for organizing the modules in Table \ref{table:decompAlt2U} for the following reason. Let \[ \pi: \bigwedge^2 ((\bigwedge^3 H_{\Q})/H_{\Q}) \rightarrow \bigwedge^6 H_{\Q} \] be the projection map sending $[a \wedge b \wedge c] \wedge [d \wedge e \wedge f] \mapsto a \wedge b \wedge c \wedge d \wedge e \wedge f$. Recall the contractions \[ C_k: \bigwedge^k H_{\Q} \rightarrow \bigwedge^{k-2} H_{\Q} \] as defined in Section \ref{section:reptheorymaps}. Then it follows from \cite[Theorem 17.5]{FH} that we have the following maps between $\operatorname{Sp}_{2g}(\Q)$ representations: \[\begin{tikzcd}
	{\bigwedge^2 ((\bigwedge^3 H_{\Q}) / H_{\Q}) = \Gamma_{0,1,0,1} \oplus \Gamma_{0,2} \oplus \Gamma_{0,0,0,0,0,1} \oplus \Gamma_{0,0,0,1} \oplus \Gamma_{0,1} \oplus \mathbb{Q}} \\
	{\bigwedge^6 H_{\Q} = \Gamma_{0,0,0,0,0,1} \oplus \Gamma_{0,0,0,1} \oplus \Gamma_{0,1} \oplus \mathbb{Q}} \\
	{\bigwedge^4 H_{\Q} = \Gamma_{0,0,0,1} \oplus \Gamma_{0,1} \oplus \mathbb{Q}} \\
	{\bigwedge^2 H_{\Q} = \Gamma_{0,1} \oplus \mathbb{Q}} \\
	{\mathbb{Q}}
	\arrow["\pi", from=1-1, to=2-1]
	\arrow["C_6", from=2-1, to=3-1]
	\arrow["C_4", from=3-1, to=4-1]
	\arrow["C_2", from=4-1, to=5-1]
\end{tikzcd}\] Additionally, we have the bracket map from Section \ref{subsection:higherMod} $$[\cdot, \cdot]: \bigwedge^2 ((\bigwedge^3 H_{\Q})/H_{\Q}) \rightarrow \Gamma_{0,2}.$$ These maps give us information about the modules in Table \ref{table:decompAlt2U}. By keeping track of which $\operatorname{Sp}_{2g}(\Q)$ representation each $\operatorname{SL}_g(\Q)$ module came from, we know whether it belongs to the kernels of the maps $\pi$ (composed with contractions) and $[\cdot, \cdot]$.
 
\section{Lower bound on $\operatorname{ker}(\tau^*)$ for $\mathcal{HI}_g$} \label{section:lowerboundtau}

We start with the following theorem, which addresses the undecorated case $\mathcal{HI}_g$.

\begin{theorem} \label{theorem:kerneltauclosed}
    When $g \geq 6$, the kernel of $\tau^*: H^2(\UQ; \Q) \rightarrow H^2(\mathcal{HI}_g; \Q)$ decomposes into $\operatorname{SL}_g(\Q)$ representations as \begin{align*}
\Q \oplus \Phi_{0,1,0, \dots ,0,1,0} \oplus \Phi_{0,2,0, \dots ,0} \oplus \Phi_{1,0, \dots ,0,1} \oplus \Phi_{1,0, \dots ,0,1,1} \oplus \Phi_{1,1,0, \dots ,0,1} \oplus \Phi_{2,0, \dots ,0} \oplus \Phi_{2,0, \dots ,0,2}.
\end{align*}
\end{theorem}

In this section we prove the following proposition, from which half of Theorem \ref{theorem:kerneltauclosed} will follow. Namely, we find a lower bound on $\operatorname{ker}(\tau^*)$ by showing that it contains each of the modules in the statement of Theorem \ref{theorem:kerneltauclosed}.

\begin{proposition} \label{proposition:lowerboundtau}
    When $g \geq 5$, the kernel of $\tau^*: H^2(\UQ; \Q) \rightarrow H^2(\mathcal{HI}_g; \Q)$ contains the following $\operatorname{SL}_g(\Q)$ representations: \begin{align*}
\Q \oplus \Phi_{0,1,0, \dots ,0,1,0} \oplus \Phi_{0,2,0, \dots ,0} \oplus \Phi_{1,0, \dots ,0,1} \oplus \Phi_{1,0, \dots ,0,1,1} \oplus \Phi_{1,1,0, \dots ,0,1} \oplus \Phi_{2,0, \dots ,0} \oplus \Phi_{2,0, \dots ,0,2}.
\end{align*}
\end{proposition}

\begin{proof}[Proof of Proposition \ref{proposition:lowerboundtau}] A module is in $\operatorname{ker}(\tau^*)$ if its dual module is in $\operatorname{coker}(\tau_*)$. We recall that the Johnson kernel $\mathcal{K}_g$ is the kernel of the Johnson homomorphism on the entire Torelli group $\mathcal{I}_g$, and we denote by $\mathcal{HK}_g$ its intersection with the handlebody group. From the five term exact sequence of the extension $$1 \rightarrow \mathcal{HK}_g \rightarrow \mathcal{HI}_g \rightarrow \UQ \rightarrow 1,$$ we get an exact sequence 
$$H_2(\mathcal{HI}_g; \mathbb{Q}) \xrightarrow{\tau_*} H_2(\UQ; \mathbb{Q}) \xrightarrow{b} H_1(\mathcal{HK}_g; \Q)_{\mathcal{HI}_g}.$$ Identifying $H_2(\UQ; \mathbb{Q}) \cong \bigwedge^2 \UQ_{\Q}$ and $H_1(\mathcal{HK}_g; \Q)_{\mathcal{HI}_g} \cong (\mathcal{HK}_g / [\mathcal{HI}_g, \mathcal{HK}_g]) \otimes \Q$ turns this into an exact sequence $$H_2(\mathcal{HI}_g; \mathbb{Q}) \xrightarrow{\tau_*} \bigwedge^2 \UQ_{\Q} \xrightarrow{b} (\mathcal{HK}_g / [\mathcal{HI}_g, \mathcal{HK}_g]) \otimes \Q.$$ We therefore have $$\operatorname{coker}(\tau_*) \cong \operatorname{im}(b).$$ The map $b$ is induced from the following map. Given $x \wedge y \in \bigwedge^2 \UQ$, let $\varphi$ and $\psi$ in $\mathcal{HI}_g$ be preimages of $x$ and $y$, respectively. Then we send $$x \wedge y \mapsto \overline{[\varphi,\psi]},$$ where $\overline{[\varphi,\psi]}$ is the class of the commutator $[\varphi,\psi]$ in the quotient $\mathcal{HK}_g / [\mathcal{HI}_g, \mathcal{HK}_g]$.

To detect modules in $\operatorname{im}(b)$, we examine the image of the composition $\tau_2^{\Q} \circ b$, where $\tau_2^{\Q}$ is the second Johnson homomorphism defined in Section \ref{subsection:higherMod} tensored with $\Q$. In Section \ref{subsection:higherMod}, we defined a bracket map $$[ \cdot, \cdot]: \bigwedge^2 \tau^{\Q}(\mathcal{I}_g) \rightarrow \tau_2^{\Q}(\mathcal{K}_g).$$ Tensoring with $\Q$ and identifying $\bigwedge^2 \tau^{\Q}(\mathcal{I}_g) \cong \bigwedge^2 ((\bigwedge^3 H_{\Q}) / H_{\Q})$ and $\tau_2^{\Q}(\mathcal{K}_g) \cong \operatorname{Sym}^2 (\bigwedge^2 H_{\Q})$ gives us

$$[ \cdot, \cdot]: \bigwedge^2 ((\bigwedge^3 H_{\Q}) / H_{\Q}) \rightarrow \operatorname{Sym}^2( \bigwedge^2 H_{\Q}).$$ Let $(a \wedge b) \leftrightarrow (c \wedge d)$ denote the vector $(a \wedge b) \otimes (c \wedge d) + (c \wedge d) \otimes (a \wedge b) \in \operatorname{Sym}^2(\bigwedge^2 H_{\Q})$. The bracket map $[ \cdot, \cdot]$ is defined on generators of $\bigwedge^2 ((\bigwedge^3 H_{\Q}) / H_{\Q})$ as follows: \begin{multline*} [a \wedge b \wedge c] \wedge [d \wedge e \wedge f] \mapsto \omega(a,d)(b \wedge c) \leftrightarrow (e \wedge f) - \omega(a,e)(b \wedge c)\leftrightarrow (d \wedge f) + \omega(a,f)(b \wedge c) \leftrightarrow (d \wedge e) \\ - \omega(b,d)(a \wedge c) \leftrightarrow (e \wedge f) + \omega(b,e)(a \wedge c) \leftrightarrow (d \wedge f) - \omega(b,f)(a \wedge c) \leftrightarrow (d \wedge e) \\ + \omega(c,d)(a \wedge b) \leftrightarrow (e \wedge f) - \omega(c,e)(a \wedge b) \leftrightarrow (d \wedge f) + \omega(c,f)(a \wedge b) \leftrightarrow (d \wedge e). \end{multline*} See \cite[Equation 4.2]{MoritaLinear}. We will consider the restriction of the bracket map $$[ \cdot, \cdot ]: \bigwedge^2 \UQ_{\Q} \rightarrow \tau_2^{\Q}(\mathcal{HK}_g).$$ It follows from the definition that $[\cdot, \cdot]$ is exactly the composition $\tau_2^{\Q} \circ b$. 

Additionally, in this section, we let $$p: \operatorname{Sym}^2( \bigwedge^2 H_{\Q}) \rightarrow \bigwedge^4 H_{\Q}$$ be the projection $(a \wedge b) \leftrightarrow (c \wedge d) \mapsto a \wedge b \wedge c \wedge d$. \\

We now show that each irreducible module in the statement of Proposition \ref{proposition:lowerboundtau} is in $\operatorname{im}([ \cdot, \cdot ])$ by finding a highest weight vector for each in $\operatorname{im}([ \cdot, \cdot ])$. Since $[\cdot, \cdot]$ is an $\mathcal{H}_g$-equivariant map, Schur's Lemma tells us that the $\operatorname{SL}_g(\Q)$ modules in $\operatorname{im}([\cdot, \cdot])$ are in $\operatorname{im}(b)$, and thus are in $\operatorname{coker}(\tau_*)$. \\

$\boxed{\Q}$: We will show that the trivial module $\Q$ is in the image of $[ \cdot, \cdot]$. First we apply $[\cdot, \cdot]$ to the vector $[a_1 \wedge a_2 \wedge b_3] \wedge [b_1 \wedge b_2 \wedge a_3]$ to get $$(a_2 \wedge b_3) \leftrightarrow (b_2 \wedge a_3) + (a_1 \wedge b_3) \leftrightarrow (b_1 \wedge a_3) - (a_1 \wedge a_2) \leftrightarrow (b_1 \wedge b_2).$$ Next, we apply a composition of $\mathcal{H}_g$-equivariant maps. First we apply the projection $p$ to get $$a_2 \wedge b_2 \wedge a_3 \wedge b_3 + a_1 \wedge b_1 \wedge a_3 \wedge b_3 + a_1 \wedge b_1 \wedge a_2 \wedge b_2.$$ Then we apply the contraction $C_4$ to get $$2 a_1 \wedge b_1 + 2a_2 \wedge b_2 + 2 a_3 \wedge b_3,$$ and contract again with $C_2$ to get $$6.$$ Since we found a highest weight vector\footnote{See Appendix \ref{appendix:hwv} and \ref{appendix:constructhwv} for details on highest weight vectors.} for $\Q$ in the image of the composition $C_2 \circ C_4 \circ p \circ [\cdot, \cdot]$, the entire module $\Q$ must lie in $\operatorname{im}(C_2 \circ C_4 \circ p \circ [\cdot, \cdot])$ because it is irreducible. By Schur's Lemma, there is a copy of $\Q$ in $\operatorname{im}([\cdot, \cdot])$. The trivial representation $\Q$ is isomorphic to its dual, and by taking duals we find a copy of $\Q$ in $\operatorname{ker}(\tau^*)$ when $g \geq 3$. \\

$\boxed{(0, \dots ,0,2)}$: As we did for $\Q$, we find a highest weight vector for $\Phi_{0, \dots ,0,2}$ in $\operatorname{im}([\cdot, \cdot])$. Apply $[\cdot, \cdot]$ to the vector $$\sum_{1 \leq i,j \leq g-1} [b_g \wedge a_i \wedge a_j] \wedge [b_g \wedge b_i \wedge b_j]$$ to get $$2(g-2) \sum_{i=1}^{g-1} (b_g \wedge a_i) \leftrightarrow (b_g \wedge b_i),$$ which is a highest weight vector for $\Phi_{0, \dots ,0,2}$. The dual $\Phi_{2,0, \dots, 0}$ lies in $\operatorname{ker}(\tau^*)$ when $g \geq 3$. \\

$\boxed{(0, \dots ,0,2,0)}$: Apply $[\cdot, \cdot]$ to the vector $$[b_{g-1} \wedge b_g \wedge a_1] \wedge [b_{g-1} \wedge b_g \wedge b_1]$$ to get $$(b_{g-1} \wedge b_g) \leftrightarrow (b_{g-1} \wedge b_g),$$ which is a highest weight vector for $\Phi_{0, \dots ,0,2,0}$. The dual $\Phi_{0,2,0, \dots, 0}$ lies in $\operatorname{ker}(\tau^*)$ when $g \geq 3$. \\

$\boxed{(0,1,0, \dots ,0,1,0)}$: Apply $[\cdot, \cdot]$ to the vector $$[b_{g-1} \wedge b_g \wedge a_3] \wedge [b_3 \wedge a_1 \wedge a_2]$$ to get $$(b_{g-1} \wedge b_g) \leftrightarrow (a_1 \wedge a_2),$$ which is a highest weight vector for $\Phi_{0,1,0, \dots ,0,1,0}$. The dual $\Phi_{0,1,0, \dots ,0,1,0}$ lies in $\operatorname{ker}(\tau^*)$ when $g \geq 5$. \\

$\boxed{(1,0, \dots ,0,1)}$: Apply the composition $C_4 \circ p \circ [ \cdot, \cdot]$ to the vector $$[b_g \wedge a_2 \wedge a_3] \wedge [a_1 \wedge b_2 \wedge b_3]$$ to get $$2 a_1 \wedge b_g,$$ which is a highest weight vector for $\Phi_{1,0, \dots ,0,1}$. The dual $\Phi_{1,0, \dots ,0,1}$ lies in $\operatorname{ker}(\tau^*)$ when $g \geq 5$. \\

$\boxed{(1,0, \dots ,0,1,1)}$: Apply $[\cdot, \cdot]$ to the vector $$[a_1 \wedge b_2 \wedge b_g] \wedge [a_2 \wedge b_{g-1} \wedge b_g]$$ to get $$(a_1 \wedge b_g) \leftrightarrow (b_{g-1} \wedge b_g),$$ which is a highest weight vector for $\Phi_{1,0, \dots ,0,1,1}$. The dual $\Phi_{1,1,0, \dots ,0,1}$ lies in $\operatorname{ker}(\tau^*)$ when $g \geq 4$. \\

$\boxed{(1,1,0, \dots ,0,1)}$: Apply $[\cdot, \cdot]$ to the vector $$[a_1 \wedge a_2 \wedge b_3] \wedge [a_3 \wedge a_1 \wedge b_g]$$ to get $$-(a_1 \wedge a_2) \leftrightarrow (a_1 \wedge b_g),$$ which is a highest weight vector for $\Phi_{1,1,0, \dots ,0,1}$. The dual $\Phi_{1,0, \dots, 0,1,1}$ lies in $\operatorname{ker}(\tau^*)$ when $g \geq 4$. \\

$\boxed{(2,0, \dots ,0,2)}$: Apply $[\cdot, \cdot]$ to the vector $$[a_1 \wedge b_g \wedge a_2] \wedge [a_1 \wedge b_g \wedge b_2]$$ to get $$(a_1 \wedge b_g) \leftrightarrow (a_1 \wedge b_g),$$ which is a highest weight vector for $\Phi_{2,0, \dots ,0,2}$. The dual $\Phi_{2,0, \dots ,0,2}$ lies in $\operatorname{ker}(\tau^*)$ when $g \geq 3$. \\

We have shown that every module in the statement of Proposition \ref{proposition:lowerboundtau} for $\mathcal{HI}_g$ lies in $\operatorname{ker}(\tau^*)$ when $g \geq 5$. \end{proof}

In Section \ref{section:upperboundtau}, we complete the proof of Theorem \ref{theorem:kerneltauclosed} by demonstrating that the remaining modules in the decomposition of $\bigwedge^2 \UQ_{\Q}^*$ do not lie in $\operatorname{ker}(\tau^*)$ when $g \geq 6$.

\subsection{Analogous results for the Torelli group} \label{subsection:hain}

In \cite{Hain}, Hain solves the analogous problem for the Torelli group $\mathcal{I}_g$. 
 
\begin{theorem}[{Hain, \cite{Hain}}] When $g \geq 3$, the kernel of $\tau^*: H^2((\bigwedge^3 H) / H; \Q) \rightarrow H^2(\mathcal{I}_g; \Q)$ is the $\operatorname{Sp}_{2g}(\Q)$ representation $\Gamma_{0,2} \oplus \Q$. \end{theorem}

We briefly explain these results and compare them to ours for $\mathcal{HI}_g$. From the five term exact sequence of the extension $$1 \rightarrow \mathcal{K}_g \rightarrow \mathcal{I}_g \rightarrow (\bigwedge^3 H) / H \rightarrow 1,$$ we get a map $$b: \bigwedge^2 ((\bigwedge^3 H_{\Q}) / H_{\Q}) \rightarrow ( \mathcal{K}_g / [\mathcal{I}_g, \mathcal{K}_g] ) \otimes \Q.$$ We have $\operatorname{coker}(\tau_*) \cong \operatorname{im}(b)$, and we can detect elements in $\operatorname{im}(b)$ by composing $b$ with other $\operatorname{Sp}_{2g}(\Q)$ maps. \\

The composition $\tau_2^{\Q} \circ b$ is exactly the bracket map $[ \cdot, \cdot]$. Here, the map $$[\cdot, \cdot]: \bigwedge^2 ((\bigwedge^3 H_{\Q}) / H_{\Q}) \rightarrow \tau_2^{\Q}(( \mathcal{K}_g / [\mathcal{I}_g, \mathcal{K}_g] ) \otimes \Q) \cong \Gamma_{0,2}$$ is surjective. We note that $\Gamma_{0,2}$ is isomorphic as an $\operatorname{Sp}_{2g}(\Q)$ representation to its dual. \\

In \cite[Section 3]{MoritaCasson}, Morita uses the Casson invariant to construct the surjective map $$\lambda^*: \mathcal{K}_g \rightarrow \Z.$$ The trivial module $\Q$ comes from a similar argument with this map. Restricting to $\mathcal{HI}_g$, we get $\lambda^*(\mathcal{HK}_g) = 0$, so the Casson invariant does \textit{not} give us an interesting quotient of $\mathcal{HK}_g / [\mathcal{HI}_g, \mathcal{HK}_g]$. 

\begin{remark} Note that $\bigwedge^2 \UQ_{\Q}$ has one fewer copy of the trivial representation $\Q$ in its decomposition than does $\bigwedge^2 ((\bigwedge^3 H_{\Q})/H_{\Q})$. \end{remark}

\begin{remark}
    The $\operatorname{Sp}_{2g}(\Q)$ representation $\Gamma_{0,2}$ branches into $\operatorname{SL}_g(\Q)$ representations as the direct sum

\hspace*{-1.3cm}
\begin{minipage}{\textwidth} \begin{align*}
\Q \oplus \Phi_{0, \dots ,0,2} \oplus  \Phi_{0, \dots ,0,2,0} \oplus \Phi_{0,1,0, \dots ,0,1,0} \oplus \Phi_{0,2,0, \dots ,0} \oplus \Phi_{1,0, \dots ,0,1} \oplus \Phi_{1,0, \dots ,0,1,1} \oplus \Phi_{1,1,0, \dots ,0,1} \oplus \Phi_{2,0, \dots ,0} \oplus \Phi_{2,0, \dots ,0,2}.
\end{align*} \end{minipage} Neither $\Phi_{0, \dots ,0,2}$ nor $\Phi_{0, \dots ,0,2,0}$ appears in the decomposition of $\bigwedge^2 \UQ_{\Q}^*$. With this background, we can restate Theorem \ref{theorem:kerneltauclosed} as $$\operatorname{ker}(\tau^*) = \bigwedge^2 \UQ_{\Q}^* \cap (\Gamma_{0,2} \oplus \Q).$$
\end{remark}

\section{Upper bound on $\operatorname{ker}(\tau^*)$ for $\mathcal{HI}_g$} \label{section:upperboundtau}

In this section, we prove the following proposition.

\begin{proposition} \label{proposition:upperboundtau}
    When $g \geq 6$, the modules \[ \Q \oplus \Phi_{0,1,0, \dots ,0,1,0} \oplus \Phi_{0,2,0, \dots ,0} \oplus \Phi_{1,0, \dots ,0,1} \oplus \Phi_{1,0, \dots ,0,1,1} \oplus \Phi_{1,1,0, \dots ,0,1} \oplus \Phi_{2,0, \dots ,0} \oplus \Phi_{2,0, \dots ,0,2} \] are the only irreducible $\operatorname{SL}_g(\Q)$ modules in $\bigwedge^2 \UQ_{\Q}^*$ that lie in $\operatorname{ker}(\tau^*)$.
\end{proposition} Theorem \ref{theorem:kerneltauclosed} will follow from Propositions \ref{proposition:lowerboundtau} and \ref{proposition:upperboundtau}. \\

As we discussed in Section \ref{section:upperboundtheta}, if a module is in the image of the induced map on homology $$\tau_*: H_2(\mathcal{HI}_g; \Q) \rightarrow \bigwedge^2 \UQ_{\Q},$$ then its dual module is not in $\operatorname{ker}(\tau^*)$. To prove Proposition \ref{proposition:upperboundtau}, we show that every irreducible $\operatorname{SL}_g(\Q)$ module in $\bigwedge^2 \UQ_{\Q}$ except for the duals of those in Proposition \ref{proposition:upperboundtau} lies in $\operatorname{im}(\tau_*)$. \\ 

We recall that $\bigwedge^2 \UQ_{\Q}$ is a submodule of $\bigwedge^2 ((\bigwedge^3 H_{\Q}) / H_{\Q})$, which has the structure of an $\operatorname{Sp}_{2g}(\Q)$ module. See Section \ref{subsection:branching}. Even though $\bigwedge^2 \UQ_{\Q}$ is not an $\operatorname{Sp}_{2g}(\Q)$ module, we can use the structure of $\bigwedge^2 ((\bigwedge^3 H_{\Q}) / H_{\Q})$ to keep organized the $\operatorname{SL}_g(\Q)$ modules in the decomposition of $\bigwedge^2 \UQ_{\Q}$ (see Table \ref{table:decompAlt2U}). We have the $\operatorname{Sp}_{2g}(\Q)$ decomposition $$\bigwedge^2 ((\bigwedge^3 H_{\Q}) / H_{\Q}) = \Gamma_{0,1,0,1} \oplus \Gamma_{0,2} \oplus \Gamma_{0,0,0,0,0,1} \oplus \Gamma_{0,0,0,1} \oplus \Gamma_{0,1} \oplus \mathbb{Q}$$ and we can restate Proposition \ref{proposition:lowerboundtau} as $$\bigwedge^2 \UQ_{\Q} \cap (\Gamma_{0,2} \oplus \Q) \subset \operatorname{coker}(\tau_*).$$ To prove Proposition \ref{proposition:upperboundtau}, we will show that $$\bigwedge^2 \UQ_{\Q} \cap (\Gamma_{0,1,0,1} \oplus \Gamma_{0,0,0,0,0,1} \oplus \Gamma_{0,0,0,1} \oplus \Gamma_{0,1}) \subset \operatorname{im}(\tau_*).$$ 

We will intersect each irreducible $\operatorname{Sp}_{2g}(\Q)$ representation with $\bigwedge^2 \UQ_{\Q}$ in its own subsection. We then detect the modules in $\operatorname{im}(\tau_*)$ using the method of abelian cycles as defined in Section \ref{subsection:abeliancycles}. In the following lemma, we introduce three vector types that lie in $\operatorname{im}(\tau_*)$. 

\begin{lemma} \label{lemma:3types} Let $1 \leq i,j,k, \ell, m,n \leq g$ be distinct. Additionally, let $x_i \in \{a_i, b_i \}$ and $x_j \in \{ a_j, b_j \}$, etc. Vectors of the following types are in $\operatorname{im}(\tau_*)$: \textnormal{\begin{align*}
    \text{Type 1:} \quad & [a_i \wedge b_i \wedge x_j - \frac{1}{g-1} x_j \wedge \omega] \wedge [a_k \wedge b_k \wedge x_{\ell} - \frac{1}{g-1} x_{\ell} \wedge \omega] \quad (g \geq 4) \\
    \text{Type 2:} \quad & [x_i \wedge x_j \wedge b_k] \wedge [x_{\ell} \wedge x_m \wedge b_n] \quad (g \geq 6) \\
    \text{Type 3:} \quad & [a_i \wedge b_i \wedge x_j - \frac{1}{g-1} x_j \wedge \omega] \wedge [\underbrace{(a_i \wedge b_i + \dots + a_j \wedge b_j)}_{N \text{ terms}} \wedge x_k - \frac{N}{g-1} x_k \wedge \omega] \quad (g \geq 3) 
\end{align*}}
\end{lemma}

\begin{proof} Consider the following commuting bounding pair annulus twists $T_{x} T_{x'}^{-1}$ and $T_{y} T_{y'}^{-1}$ and $T_{z} T_{z'}^{-1}$. 

{\centerline{\includegraphics[scale=.5]{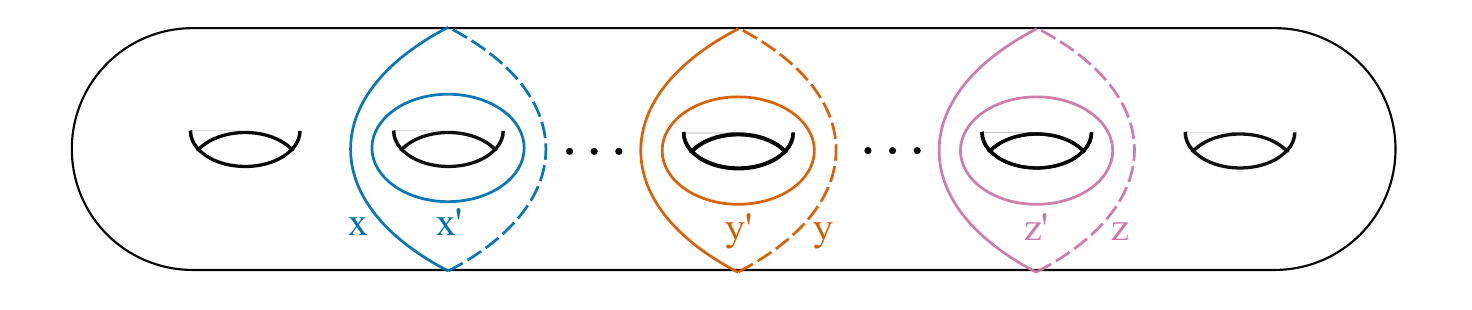}}}

We evaluate $\tau_*$ on the abelian cycle $\{T_{x} T_{x'}^{-1}, T_{z} T_{z'}^{-1} \}$: $$\tau_*(\{T_{x} T_{x'}^{-1}, T_{z} T_{z'}^{-1} \}) = [a_1 \wedge b_1 \wedge a_2 - \frac{1}{g-1} a_2 \wedge \omega] \wedge [a_g \wedge b_g \wedge a_{g-1} - \frac{1}{g-1} a_{g-1} \wedge \omega].$$ We can obtain any Type 1 vector from here via an appropriate composition\footnote{Recall that handleswaps permute indices, and $X_i$ sends $a_i \rightarrow b_i$ and kills all other generators.} of handleswaps and Lie algebra maps $X_i$. Since finite dimensional representations of $\overline{\Psi(\mathcal{H}_g)}$ are in natural bijection with those of its Lie algebra, all Type 1 vectors are in $\operatorname{im}(\tau_*)$. \\

Next, apply a composition of Lie algebra maps to $\tau_*(\{T_{x} T_{x'}^{-1}, T_{z} T_{z'}^{-1} \})$: $$-E_{14} \circ E_{g5} \circ \tau_*(\{T_{x} T_{x'}^{-1}, T_{z} T_{z'}^{-1} \}) = [a_1 \wedge a_2 \wedge b_g] \wedge [a_{g-1} \wedge a_g \wedge b_5].$$ We obtain any Type 2 vector  via an appropriate composition of handleswaps and $X_i$. \\

If the curve $y$ represents the homology class $a_{N+1}$, then we evaluate: $$\tau_*(\{T_{x} T_{x'}^{-1}, T_{y} T_{y'}^{-1} \}) = [a_1 \wedge b_1 \wedge a_2 - \frac{1}{g-1} a_2 \wedge \omega] \wedge [(\sum_{j=1}^N a_j \wedge b_j) \wedge a_{N+1} - \frac{N}{g-1} a_{N+1} \wedge \omega].$$ From here, we obtain any Type 3 vector via an appropriate choice of $N \in \{2, \dots, g-1 \}$ and composition of handleswaps and $X_i$. \end{proof}

We will detect modules in $\operatorname{im}(\tau_*)$ by finding their highest weight vectors in the $\mathcal{H}_g$-orbit of our vector types from Lemma \ref{lemma:3types}. First, we consider the vector $$v := [(g-1) a_1 \wedge b_1 \wedge a_2 - a_2 \wedge \omega] \wedge [(g-1) a_3 \wedge b_3 \wedge a_4 - a_4 \wedge \omega].$$ This vector $v$ is a Type 1 vector, so by Lemma \ref{lemma:3types} it lies in $\operatorname{im}(\tau_*)$. We apply the projection map $\pi: \bigwedge^2 ((\bigwedge^3 H_{\Q}) / H_{\Q}) \rightarrow \bigwedge^6 H_{\Q}$ to get $$\pi(v) = ((g-1)a_1 \wedge b_1 \wedge a_2 - a_2 \wedge \omega) \wedge ((g-1) a_3 \wedge b_3 \wedge a_4 - a_4 \wedge \omega).$$ In the following sections, we will
find highest weight vectors for each irreducible $\operatorname{SL}_g(\Q)$ module in $$\bigwedge^2 \UQ_{\Q} \cap [\Gamma_{0,0,0,0,0,1} \oplus \Gamma_{0,0,0,1} \oplus \Gamma_{0,1}]$$ in the $\mathcal{H}_g$-orbit of $\pi(v)$. Since $\pi$ is an $\mathcal{H}_g$-equivariant map, if we find a highest weight vector of weight $(w_1, w_2, \dots, w_{g-1})$ in the $\mathcal{H}_g$-orbit of $\pi(v)$, then $\Phi_{w_1, w_2, \dots, w_{g-1}} \subset \operatorname{im}(\tau_*)$ by Schur's Lemma. By projecting to $\bigwedge^6 H_{\Q}$ like this, we greatly simplify calculations because it is easier to express highest weight vectors as elements of $\bigwedge^6 H_{\Q}$ than as elements of $\bigwedge^2 \UQ_{\Q}$. 

In Section \ref{subsection:0101}, we will address $\bigwedge^2 \UQ_{\Q} \cap \Gamma_{0,1,0,1}$ which lies in $\operatorname{ker}(\pi)$.

%= (g-1)^2 a_1 \wedge b_1 \wedge a_2 \wedge a_3 \wedge b_3 \wedge a_4 - (g-1) a_1 \wedge b_1 \wedge a_2 \wedge a_4 \wedge \omega - (g-1) a_2 \wedge \omega \wedge a_3 \wedge b_3 \wedge a_4 + a_2 \wedge \omega \wedge a_4 \wedge \omega. 

\subsection{The module $\boldsymbol{\Gamma_{0,0,0,0,0,1}}$} \label{subsection:000001} The $\operatorname{Sp}_{2g}(\Q)$ representation $\Gamma_{0,0,0,0,0,1}$ branches into irreducible $\operatorname{SL}_g(\Q)$ representations as: \begin{align*}
\Phi_{0,0,0,0,0,1,0, \dots ,0} \oplus \Phi_{0,0,0,0,1,0, \dots ,0,1} \oplus \Phi_{0,0,0,1,0, \dots ,0,1,0} \oplus \Phi_{0,0,1,0, \dots ,0,1,0,0} \\ \oplus \Phi_{0,1,0, \dots , 0,1,0,0,0} \oplus \Phi_{1,0, \dots ,0,1,0,0,0,0} \oplus \Phi_{0, \dots, 0,1,0,0,0,0,0}.  \end{align*} Neither $\Phi_{0,0,0,0,0,1,0, \dots ,0}$ nor $\Phi_{0,0,0,0,1,0, \dots ,0,1}$ lies in $\bigwedge^2 \UQ_{\Q}$. We find the remaining modules in the $\mathcal{H}_g$-orbit of $\pi(v)$. These results hold for all $g \geq 6$. \\

$\boxed{(0,0,0,1,0, \dots ,0,1,0)}$: We use the fact that finite dimensional representations of $\overline{\Psi(\mathcal{H}_g)}$ are in natural bijection with those of its Lie algebra, and we apply a composition of Lie algebra maps to $$\pi(v) = ((g-1)a_1 \wedge b_1 \wedge a_2 - a_2 \wedge \omega) \wedge ((g-1) a_3 \wedge b_3 \wedge a_4 - a_4 \wedge \omega).$$ We describe these Lie algebra maps in Appendix \ref{appendix:liealgebra}. First we apply $E_{1g}$ to get $$E_{1g} \circ \pi(v) = - ((g-1)a_1 \wedge b_g \wedge a_2) \wedge ((g-1) a_3 \wedge b_3 \wedge a_4 - a_4 \wedge \omega).$$ We then apply $E_{3(g-1)}$ to get $$E_{3(g-1)} \circ E_{1g} \circ \pi (v) = -(g-1)^2 a_1 \wedge a_2 \wedge a_3 \wedge a_4 \wedge b_{g-1} \wedge b_g,$$ which is a highest weight vector for $\Phi_{0,0,0,1,0, \dots ,0,1,0}$. \\

$\boxed{(0,0,1,0, \dots ,0,1,0,0)}$: Building on previous work, we apply a composition of Lie algebra maps to $E_{3(g-1)} \circ E_{1g} \circ \pi (v)$: $$X_{g-2} \circ E_{(g-2)4} \circ E_{3(g-1)} \circ E_{1g} \circ \pi(v) = -(g-1)^2 a_1 \wedge a_2 \wedge a_3 \wedge b_{g-2} \wedge b_{g-1} \wedge b_g$$ is a highest weight vector for $\Phi_{0,0,1,0, \dots ,0,1,0,0}$.  \\

$\boxed{(0,1,0, \dots, 0,1,0,0,0)}$: $$X_{g-3} \circ E_{(g-3)3} \circ X_{g-2} \circ E_{(g-2)4} \circ E_{3(g-1)} \circ E_{1g} \circ \pi(v) = -(g-1)^2 a_1 \wedge a_2 \wedge b_{g-3} \wedge b_{g-2} \wedge b_{g-1} \wedge b_g$$ is a highest weight vector for $\Phi_{0,1,0, \dots, 0,1,0,0,0}$. \\

$\boxed{(1,0, \dots, 0,1,0,0,0,0)}$: $$X_{g-4} \circ E_{(g-4)2} \circ X_{g-3} \circ E_{(g-3)3} \circ X_{g-2} \circ E_{(g-2)4} \circ E_{3(g-1)} \circ E_{1g} \circ \pi(v) = -(g-1)^2 a_1 \wedge b_{g-4} \wedge b_{g-3} \wedge b_{g-2} \wedge b_{g-1} \wedge b_g$$ is a highest weight vector for $\Phi_{1,0, \dots, 0,1,0,0,0,0}$. \\

$\boxed{(0, \dots ,0,1,0,0,0,0,0)}$: 
\begin{align*}
    X_{g-5} \circ E_{(g-5)1} \circ X_{g-4} \circ E_{(g-4)2} \circ X_{g-3} \circ E_{(g-3)3} \circ X_{g-2} \circ E_{(g-2)4} \circ E_{3(g-1)} \circ E_{1g} \circ \pi(v) \\ = -(g-1)^2 b_{g-5} \wedge b_{g-4} \wedge b_{g-3} \wedge b_{g-2} \wedge b_{g-1} \wedge b_g
\end{align*} is a highest weight vector for $\Phi_{0, \dots ,0,1,0,0,0,0,0}$.

\subsection{The module $\boldsymbol{\Gamma_{0,0,0,1}}$} \label{subsection:0001} The $\operatorname{Sp}_{2g}(\Q)$ representation $\Gamma_{0,0,0,1}$ branches into irreducible $\operatorname{SL}_g(\Q)$ representations as $$\Phi_{0,0,0,1,0,\dots,0} \oplus \Phi_{0,0,1,0, \dots ,0,1} \oplus \Phi_{0,1,0, \dots ,0,1,0} \oplus \Phi_{1,0, \dots ,0,1,0,0} \oplus \Phi_{0, \dots ,0,1,0,0,0}.$$ The module $\Phi_{0,0,0,1,0,\dots,0}$ is not in $\bigwedge^2 \UQ_{\Q}$. We find the remaining modules in the $\mathcal{H}_g$-orbit of $$C_6 \circ \pi(v) = 2(g-1)(a_1 \wedge b_1 + a_3 \wedge b_3) \wedge a_2 \wedge a_4 - 2a_2 \wedge a_4 \wedge \omega$$ for all $g \geq 4$. \\

$\boxed{(0,0,1,0, \dots ,0,1)}$: $$E_{34} \circ E_{1g} \circ C_6 \circ \pi(v) = -2(g-1) a_1 \wedge a_2 \wedge a_3 \wedge b_g$$ is a highest weight vector for $\Phi_{0,0,1,0, \dots ,0,1}$. \\

$\boxed{(0,1,0, \dots ,0,1,0)}$: $$X_{g-1} \circ E_{(g-1)3} \circ E_{34} \circ E_{1g} \circ C_6 \circ \pi(v) = -2(g-1) a_1 \wedge a_2 \wedge b_{g-1} \wedge b_g$$ is a highest weight vector for $\Phi_{0,1,0, \dots ,0,1,0}$. \\

$\boxed{(1,0, \dots ,0,1,0,0)}$: $$X_{g-2} \circ E_{(g-2)2} \circ X_{g-1} \circ E_{(g-1)3} \circ E_{34} \circ E_{1g} \circ C_6 \circ \pi(v) = -2(g-1) a_1 \wedge b_{g-2} \wedge b_{g-1} \wedge b_g$$ is a highest weight vector for $\Phi_{1,0, \dots ,0,1,0,0}$. \\

$\boxed{(0, \dots ,0,1,0,0,0)}$: $$X_{g-3} \circ E_{(g-3)1} \circ X_{g-2} \circ E_{(g-2)2} \circ X_{g-1} \circ E_{(g-1)3} \circ E_{34} \circ E_{1g} \circ C_6 \circ \pi(v) = -2(g-1) b_{g-3} \wedge b_{g-2} \wedge b_{g-1} \wedge b_g$$ is a highest weight vector for $\Phi_{0, \dots ,0,1,0,0,0}$. 

\subsection{The module $\boldsymbol{\Gamma_{0,1}}$} \label{subsection:01} The $\operatorname{Sp}_{2g}(\Q)$ representation $\Gamma_{0,1}$ branches into irreducible $\operatorname{SL}_g(\Q)$ representations as $$\Phi_{0,1,0,\dots,0} \oplus \Phi_{1,0,\dots,0,1} \oplus \Phi_{0,\dots,0,1,0}.$$ We find all of these modules in the $\mathcal{H}_g$-orbit of $$C_4 \circ C_6 \circ \pi(v) = 2g a_2 \wedge a_4$$ for all $g \geq 4$. \\

$\boxed{(0,1,0, \dots ,0)}$: $$E_{41} \circ C_4 \circ C_6(v) = -2g a_1 \wedge a_2$$ is a highest weight vector for $\Phi_{0,1,0, \dots ,0}$. \\

$\boxed{(1,0, \dots ,0,1)}$: $$X_g \circ E_{g2} \circ E_{41} \circ C_4 \circ C_6 \circ \pi(v) = -2g a_1 \wedge b_g$$ is a highest weight vector for $\Phi_{1,0, \dots ,0,1}$. \\

$\boxed{(0, \dots ,0,1,0)}$: 

$$X_g \circ X_{g-1} \circ E_{g4} \circ E_{(g-1)2} \circ C_4 \circ C_6 \circ \pi(v) = 2g b_{g-1} \wedge b_g$$ is a highest weight vector for $\Phi_{0, \dots ,0,1,0}$. 

\subsection{The module $\boldsymbol{\Gamma_{0,1,0,1}}$} \label{subsection:0101} The $\operatorname{Sp}_{2g}(\Q)$ representation $\Gamma_{0,1,0,1}$ branches into a direct sum of $\operatorname{SL}_g(\Q)$ modules. In this section, we show that each module in $\Gamma_{0,1,0,1} \cap \bigwedge^2 \UQ_{\Q}$ belongs to $\operatorname{im}(\tau_*)$. First, we use our work from previous sections to determine $\Gamma_{0,1,0,1} \cap \bigwedge^2 \UQ_{\Q}$. We recall two $\operatorname{Sp}_{2g}(\Q)$ maps. We have $$\pi: \bigwedge^2 ((\bigwedge^3 H_{\Q}) / H_{\Q}) \rightarrow \bigwedge^6 H_{\Q}$$ which evaluates on generators as $[a \wedge b \wedge c] \wedge [d \wedge e \wedge f] \mapsto a \wedge b \wedge c \wedge d \wedge e \wedge f$. We also have the bracket map $$[ \cdot, \cdot ]: \bigwedge^2 ((\bigwedge^3 H_{\Q}) / H_{\Q}) \rightarrow \Gamma_{0,2}.$$ We decomposed $\bigwedge^2 \UQ_{\Q}$ in Table \ref{table:decompAlt2U} and determined the images of $[\cdot, \cdot]|_{\bigwedge^2 \UQ_{\Q}}$ and $\pi|_{\bigwedge^2 \UQ_{\Q}}$ in Sections \ref{section:lowerboundtau} and \ref{section:upperboundtau}, respectively. With this information, we determine that the modules in Table \ref{table:0101} are {\textit{not}} in $\Gamma_{0,1,0,1} \cap \bigwedge^2 \UQ_{\Q}$. 

\begin{table}[h!]
\centering
\begin{tabular}{|c|c|c|c|}
\hline
\multirow{2}{*}{\textbf{Highest Weight}} & \multicolumn{3}{c|}{\textbf{Multiplicity}} \\
\cline{2-4}
& \rule{0pt}{2.6ex} $\bigwedge^2 \UQ_{\mathbb{Q}}$ & $(\bigwedge^2 \UQ_{\mathbb{Q}}) \cap \Gamma_{0,2}$ & $(\bigwedge^2 \UQ_{\mathbb{Q}}) \cap (\bigwedge^6 H_{\mathbb{Q}})$ \\
\hline
$(0,0,0,1,0,\dots,0,1,0)$ & 1 & 0 & 1 \\
$(0,0,1,0,\dots,0,1)$ & 1 & 0 & 1 \\
$(0,0,1,0,\dots,0,1,0,0)$ & 1 & 0 & 1 \\
$(0,1,0,\dots,0)$ & 1 & 0 & 1 \\
$(0,1,0,\dots,0,1,0)$ & 2 & 1 & 1 \\
$(0,1,0,1,0,\dots,0)$ & 0 & 0 & 0 \\
$(0,1,1,0,\dots,0,1)$ & 0 & 0 & 0 \\
$(1,0,\dots,0,1)$ & 2 & 1 & 1 \\
$(1,0,0,1,0,\dots,0,1)$ & 0 & 0 & 0 \\
$(1,0,1,0,\dots,0)$ & 0 & 0 & 0 \\
$(1,0,1,0,\dots,0,1,0)$ & 0 & 0 & 0 \\
$(1,1,0,\dots,0,1)$ & 1 & 1 & 0 \\
\hline
\end{tabular}
\caption{The $\operatorname{SL}_g(\Q)$ modules in $\Gamma_{0,1,0,1} - \bigwedge^2 \UQ_{\Q}$}
\label{table:0101}
\end{table}

By finding them in $\operatorname{im}(\tau_*)$, we show that the remaining $\operatorname{SL}_g(\Q)$ modules in the branched decomposition of $\Gamma_{0,1,0,1}$ are in $\Gamma_{0,1,0,1} \cap \bigwedge^2 \UQ_{\Q}$. We get the decomposition \begin{align*} 
    \Gamma_{0,1,0,1} \cap \bigwedge^2 \UQ_{\Q} \cong \Phi_{0, \dots ,0,1,0} \oplus \Phi_{0, \dots ,0,1,0,1} \oplus \Phi_{0, \dots ,0,1,0,1,0} \oplus \Phi_{0,0,1,0, \dots ,0,1,1} \oplus \Phi_{0,1,0, \dots ,0,2} \oplus \Phi_{0,1,0, \dots ,0,2,0} \\
    \oplus \Phi_{0,1,0, \dots ,0,1,0,1} \oplus \Phi_{0,1,0, \dots ,0,1,0,0,0} \oplus \Phi_{0,2,0, \dots ,0,1,0} \oplus \Phi_{1,0, \dots ,0,1,1} \oplus \Phi_{1,0, \dots ,0,1,0,0} \oplus \Phi_{1,0, \dots ,0,1,1,0} \\
    \oplus \Phi_{1,0, \dots ,0,1,0,0,1} \oplus \Phi_{1,0,1,0, \dots,0,2} \oplus \Phi_{1,1,0,\dots,0,1,1} \oplus \Phi_{1,1,0, \dots ,0,1,0,0} \oplus \Phi_{2,0, \dots ,0,1,0} \oplus \Phi_{2,0, \dots ,0,1,0,1}.
\end{align*} Now we show that $\Gamma_{0,1,0,1} \cap \bigwedge^2 \UQ_{\Q} \subset \operatorname{im}(\tau_*)$ when $g \geq 6$. \\

$\boxed{(0, \dots, 0,1,0)}$: Fix distinct $i,j,k \in \{1, \dots, g-2 \}$. We define a vector $x_{ij}$ (which does not depend on our choice of $k$): \begin{align*} X_i \circ X_{g-1} \circ E_{(g-1)g} \circ E_{ik} \underbrace{[a_g \wedge b_g \wedge a_k - \frac{1}{g-1} a_k \wedge \omega] \wedge [a_i \wedge a_j \wedge b_j - \frac{1}{g-1} a_i \wedge \omega]}_{\text{Type 1}}  \\ = \underbrace{[b_{g-1} \wedge b_g \wedge b_i] \wedge [a_i \wedge a_j \wedge b_j - \frac{1}{g-1} a_i \wedge \omega ] + [b_{g-1} \wedge b_g \wedge a_i] \wedge [b_i \wedge a_j \wedge b_j - \frac{1}{g-1} b_i \wedge \omega ]}_{x_{ij}}. \end{align*} Note that $x_{ij} \in \operatorname{im}(\tau_*)$ since it is in the $\mathcal{H}_g$-orbit of a Type 1 vector. \\

Fix $i \in \{1, \dots, g-2 \}$. We define a vector $y_i$: 

\begin{align*}
    X_i \circ E_{i(g-1)} \underbrace{[a_{g-1} \wedge b_{g-1} \wedge b_g - \frac{1}{g-1} b_g \wedge \omega] \wedge [(a_{g-1} \wedge b_{g-1} + a_g \wedge b_g) \wedge a_i - \frac{2}{g-1} a_i \wedge \omega]}_{\text{Type 3}} \\ = y_i := [b_i \wedge b_{g-1} \wedge b_g] \wedge [(a_{g-1} \wedge b_{g-1} + a_g \wedge b_g) \wedge a_i - \frac{2}{g-1} a_i \wedge \omega ] \\ + [a_i \wedge b_{g-1} \wedge b_g] \wedge [(a_{g-1} \wedge b_{g-1} + a_g \wedge b_g) \wedge b_i - \frac{2}{g-1} b_i \wedge \omega].
\end{align*}

Fix distinct $i,j \in \{1, \dots, g-2 \}$. We define a vector $z_i$: 

\begin{align*}
    -E_{jg} \underbrace{[b_{g-1} \wedge a_g \wedge b_g - \frac{1}{g-1} b_{g-1} \wedge \omega] \wedge [b_j \wedge a_i \wedge b_i - \frac{1}{g-1} b_j \wedge \omega]}_{\text{Type 1}} \\ - E_{j(g-1)} \underbrace{[a_{g-1} \wedge b_{g-1} \wedge b_g - \frac{1}{g-1} b_g \wedge \omega] \wedge [b_j \wedge a_i \wedge b_i - \frac{1}{g-1} b_j \wedge \omega]}_{\text{Type 1}} \\ = z_i := [b_{g-1} \wedge a_g \wedge b_g - \frac{1}{g-1} b_{g-1} \wedge \omega] \wedge [b_g \wedge a_i \wedge b_i - \frac{1}{g-1} b_g \wedge \omega] \\ + [a_{g-1} \wedge b_{g-1} \wedge b_g - \frac{1}{g-1} b_g \wedge \omega] \wedge [b_{g-1} \wedge a_i \wedge b_i - \frac{1}{g-1} b_{g-1} \wedge \omega].
\end{align*}

The sum $$\sum_{1 \leq i \neq j \leq g-2} x_{ij} + \sum_{i=1}^{g-2} (y_i + z_i)$$ is a highest weight vector for $\Phi_{0, \dots, 0,1,0}$. \\

$\boxed{(0, \dots ,0,1,0,1)}$: Fix distinct $i,j,k \in \{1, \dots, g-3 \}$. We define a vector $x_i:$

\begin{align*}
    - X_i \circ E_{ik} \circ E_{jg} \underbrace{[b_{g-1} \wedge b_g \wedge a_i] \wedge [b_{g-2} \wedge b_j \wedge a_k]}_{\text{Type 2}} \\ = \underbrace{[b_{g-1} \wedge b_g \wedge a_i] \wedge [b_{g-2} \wedge b_g \wedge b_i] + [b_{g-1} \wedge b_g \wedge b_i] \wedge [b_{g-2} \wedge b_g \wedge a_i]}_{x_i}.
\end{align*}

We define a vector $y$:

\begin{align*}
    X_{g-2} \circ E_{(g-2)g} \underbrace{[a_{g-2} \wedge b_{g-2} \wedge b_{g-1} - \frac{1}{g-1} b_{g-1} \wedge \omega] \wedge [(a_{g-2} \wedge b_{g-2} + a_{g-1} \wedge b_{g-1}) \wedge b_g - \frac{2}{g-1} b_g \wedge \omega]}_{\text{Type 3}} \\ = \underbrace{[b_{g-2} \wedge b_{g-1} \wedge b_g] \wedge [(a_{g-2} \wedge b_{g-2} + a_{g-1} \wedge b_{g-1}) \wedge b_g - \frac{2}{g-1} b_g \wedge \omega]}_{\text{y}}.
\end{align*}

The sum $$(\sum_{i=1}^{g-3} x_i) + y$$ is a highest weight vector for $\Phi_{0, \dots ,0,1,0,1}$. \\

$\boxed{(0,\dots,0,1,0,1,0)}$: $$E_{1(g-1)} \circ E_{2g} \underbrace{[b_{g-2} \wedge b_{g-1} \wedge b_g] \wedge [b_{g-3} \wedge b_1 \wedge b_2]}_{\text{Type 2}} = [b_{g-2} \wedge b_{g-1} \wedge b_g] \wedge [b_{g-3} \wedge b_{g-1} \wedge b_g]$$ is a highest weight vector for $\Phi_{0,\dots,0,1,0,1,0}$. \\

$\boxed{(0,0,1,0,\dots ,0,1,1)}$: \begin{align*} - E_{4g} \underbrace{([a_1 \wedge a_2 \wedge b_g] \wedge [a_3 \wedge b_{g-1} \wedge b_4] - [a_1 \wedge a_3 \wedge b_g] \wedge [a_2 \wedge b_{g-1} \wedge b_4] - [a_1 \wedge b_{g-1} \wedge b_g] \wedge [a_2 \wedge a_3 \wedge b_4])}_{\text{Type 2}} \\ = [a_1 \wedge a_2 \wedge b_g] \wedge [a_3 \wedge b_{g-1} \wedge b_g] + [a_1 \wedge a_3 \wedge b_g] \wedge [a_2 \wedge b_{g-1} \wedge b_g] + [a_1 \wedge b_{g-1} \wedge b_g] \wedge [a_2 \wedge a_3 \wedge b_g] \end{align*} is a highest weight vector for $\Phi_{0,0,1,0,\dots ,0,1,1}$. \\

$\boxed{(0,1,0, \dots ,0,2)}$: Fix $i \in \{3, \dots, g-1 \}$. We define a vector $x_i$:

\begin{align*}
    - X_i \circ E_{i4} \circ E_{5g} \underbrace{[a_1 \wedge a_i \wedge b_g] \wedge [a_2 \wedge a_4 \wedge b_5]}_{\text{Type 2}} \\ = \underbrace{[a_1 \wedge a_i \wedge b_g] \wedge [a_2 \wedge b_i \wedge b_g] + [a_1 \wedge b_i \wedge b_g] \wedge [a_2 \wedge a_i \wedge b_g]}_{x_i}.
\end{align*}

We define a vector $y$: \begin{align*}
    - E_{2g} \underbrace{[a_1 \wedge a_2 \wedge b_2 - \frac{1}{g-1} a_1 \wedge \omega] \wedge [(a_1 \wedge b_1 + a_2 \wedge b_2) \wedge b_g - \frac{2}{g-1} b_g \wedge \omega]}_{\text{Type 3}} \\
    = \underbrace{[a_1 \wedge a_2 \wedge b_g] \wedge [(a_1 \wedge b_1 + a_2 \wedge b_2) \wedge b_g - \frac{2}{g-1} b_g \wedge \omega]}_{y}.
\end{align*}

The sum $$(\sum_{i=3}^{g-1} x_i) + y$$ is a highest weight vector for $\Phi_{0,1,0, \dots ,0,2}$. \\

$\boxed{(0,1,0, \dots ,0,2,0)}$: 

 $$E_{3(g-1)} \circ E_{4g} \underbrace{[a_1 \wedge b_{g-1} \wedge b_g] \wedge [a_2 \wedge b_3 \wedge b_4]}_{\text{Type 2}} = [a_1 \wedge b_{g-1} \wedge b_g] \wedge [a_2 \wedge b_{g-1} \wedge b_g]$$ is a highest weight vector for $\Phi_{0,1,0, \dots ,0,2,0}$. \\

$\boxed{(0,1,0, \dots ,0,1,0,1)}$: 

$$-E_{3g} \underbrace{[a_1 \wedge a_2 \wedge b_g] \wedge [b_{g-2} \wedge b_{g-1} \wedge b_3]}_{\text{Type 2}} = [a_1 \wedge a_2 \wedge b_g] \wedge [b_{g-2} \wedge b_{g-1} \wedge b_g]$$ is a highest weight vector for $\Phi_{0,1,0, \dots ,0,1,0,1}$. \\

$\boxed{(0,1,0, \dots ,0,1,0,0,0)}$: Consider the Type 2 vector
\begin{align*}
x := [a_1 \wedge a_2 \wedge b_{g-3}] \wedge [b_{g-2} \wedge b_{g-1} \wedge b_g] - [a_1 \wedge a_2 \wedge b_{g-2}] \wedge [b_{g-3} \wedge b_{g-1} \wedge b_g] \\ + [a_1 \wedge a_2 \wedge b_{g-1}] \wedge [b_{g-3} \wedge b_{g-2} \wedge b_g] - [a_1 \wedge a_2 \wedge b_g] \wedge [b_{g-3} \wedge b_{g-2} \wedge b_{g-1}].
\end{align*}

\noindent The vector $x$ is a highest weight vector for $\Phi_{0,1,0, \dots ,0,1,0,0,0}$ and $$\pi(x) = 4 a_1 \wedge a_2 \wedge b_{g-3} \wedge b_{g-2} \wedge b_{g-1} \wedge b_g \in \bigwedge^6 H_{\Q}.$$ 

Consider also the Type 2 vector \begin{align*}
    y := [a_1 \wedge b_{g-1} \wedge b_g] \wedge [a_2 \wedge b_{g-3} \wedge b_{g-2}] - [a_1 \wedge b_{g-2} \wedge b_g] \wedge [a_2 \wedge b_{g-3} \wedge b_{g-1}] \\ + [a_1 \wedge b_{g-3} \wedge b_g] \wedge [a_2 \wedge b_{g-2} \wedge b_{g-1}] + [a_1 \wedge b_{g-2} \wedge b_{g-1}] \wedge [a_2 \wedge b_{g-3} \wedge b_g] \\ - [a_1 \wedge b_{g-3} \wedge b_{g-1}] \wedge [a_2 \wedge b_{g-2} \wedge b_g] + [a_1 \wedge b_{g-3} \wedge b_{g-2}] \wedge [a_2 \wedge b_{g-1} \wedge b_g].
\end{align*} The vector $y$ is a highest weight vector for $\Phi_{0,1,0, \dots ,0,1,0,0,0}$ and $$\pi(y) = 6 a_1 \wedge a_2 \wedge b_{g-3} \wedge b_{g-2} \wedge b_{g-1} \wedge b_g \in \bigwedge^6 H_{\Q}.$$ The linear combination $3x - 2y$ satisfies $$\pi(3x - 2y) = 0.$$ Since we detected the first copy of $\Phi_{0,1,0, \dots ,0,1,0,0,0} \subset \operatorname{im}(\tau_*)$ with the projection map $\pi$ from Section \ref{subsection:000001}, the vector $2x - 3y$ is a highest weight vector for a second copy of $\Phi_{0,1,0, \dots ,0,1,0,0,0}$. \\

$\boxed{(0,2,0, \dots,0,1,0)}$: 

$$E_{24} \circ E_{13} \underbrace{[a_1 \wedge a_2 \wedge b_{g-1}] \wedge [a_3 \wedge a_4 \wedge b_g]}_{\text{Type 2}} = [a_1 \wedge a_2 \wedge b_{g-1}] \wedge [a_1 \wedge a_2 \wedge b_g]$$ is a highest weight vector for $\Phi_{0,2,0, \dots,0,1,0}$. \\

$\boxed{(1,0, \dots ,0,1,1)}$: Fix distinct $i,j,k \in \{2, \dots, g-3 \}$. We define a vector $x_i$: 

\hspace*{-1.5cm}
\begin{minipage}{\textwidth}

\begin{align*}
    - X_i \circ E_{ij} \circ E_{kg} \underbrace{[a_i \wedge b_{g-1} \wedge b_g] \wedge [a_1 \wedge a_j \wedge b_k]}_{\text{Type 2}} = \underbrace{[a_i \wedge b_{g-1} \wedge b_g] \wedge [a_1 \wedge b_i \wedge b_g] + [b_i \wedge b_{g-1} \wedge b_g] \wedge [a_1 \wedge a_i \wedge b_g]}_{x_i}.
\end{align*} \end{minipage}

We define a vector $y$: \begin{align*}
    E_{1g} \underbrace{[(a_1 \wedge b_1 + a_{g-1} \wedge b_{g-1}) \wedge b_g - \frac{2}{g-1} b_g \wedge \omega] \wedge [a_1 \wedge b_1 \wedge b_{g-1} - \frac{1}{g-1} b_{g-1} \wedge \omega]}_{\text{Type 3}} \\ = \underbrace{[(a_1 \wedge b_1 + a_{g-1} \wedge b_{g-1}) \wedge b_g - \frac{2}{g-1} b_g \wedge \omega] \wedge [a_1 \wedge b_{g-1} \wedge b_g]}_{y}.
\end{align*} The sum \begin{align*} (\sum_{i=2}^{g-3} x_i) - y  \end{align*} is a highest weight vector for $\Phi_{1,0, \dots ,0,1,1}$. \\ 

$\boxed{(1,0, \dots ,0,1,0,0)}$: Fix distinct $i,j \in \{2, \dots, g-3 \}$. We define a vector $x_i$: 

\hspace*{-1.3cm}
\begin{minipage}{\textwidth} \begin{align*}
    X_i \circ E_{ij} (\underbrace{[a_1 \wedge a_i \wedge b_{g-2}] \wedge [a_j \wedge b_{g-1} \wedge b_g] - [a_1 \wedge a_i \wedge b_{g-1}] \wedge [a_j \wedge b_{g-2} \wedge b_g] + [a_1 \wedge a_i \wedge b_g] \wedge [a_j \wedge b_{g-2} \wedge b_{g-1}]}_{\text{Type 2}})
\end{align*} \end{minipage} equals \begin{align*}
    x_i := [a_1 \wedge a_i \wedge b_{g-2}] \wedge [b_i \wedge b_{g-1} \wedge b_g] + [a_1 \wedge b_i \wedge b_{g-2}] \wedge [a_i \wedge b_{g-1} \wedge b_g] \\
    - [a_1 \wedge a_i \wedge b_{g-1}] \wedge [b_i \wedge b_{g-2} \wedge b_g] - [a_1 \wedge b_i \wedge b_{g-1}] \wedge [a_i \wedge b_{g-2} \wedge b_g] \\ 
    + [a_1 \wedge a_i \wedge b_g] \wedge [b_i \wedge b_{g-2} \wedge b_{g-1}] + [a_1 \wedge b_i \wedge b_g] \wedge [a_i \wedge b_{g-2} \wedge b_{g-1}].
\end{align*}

We define a vector $y$:  \begin{align*}
    X_g \circ E_{g(g-2)} \underbrace{[a_1 \wedge (a_{g-2} \wedge b_{g-2} + a_{g-1} \wedge b_{g-1} + a_g \wedge b_g) - \frac{3}{g-1} a_1 \wedge \omega] \wedge [a_{g-2} \wedge b_{g-2} \wedge b_{g-1} - \frac{1}{g-1} b_{g-1} \wedge \omega]}_{\text{Type 3}} \\ = \underbrace{[a_1 \wedge (a_{g-2} \wedge b_{g-2} + a_{g-1} \wedge b_{g-1} + a_g \wedge b_g) - \frac{3}{g-1} a_1 \wedge \omega] \wedge [b_{g-2} \wedge b_{g-1} \wedge b_g]}_{y}.
\end{align*}

We define a vector $z_{g-2}$: \begin{align*}
    - E_{1(g-1)} \underbrace{[(a_1 \wedge b_1 + a_{g-1} \wedge b_{g-1} + a_g \wedge b_g) \wedge b_{g-2} - \frac{3}{g-1} b_{g-2} \wedge \omega] \wedge [a_1 \wedge b_1 \wedge b_g - \frac{1}{g-1} b_g \wedge \omega]}_{\text{Type 3}} \\ = \underbrace{[(a_1 \wedge b_1 + a_{g-1} \wedge b_{g-1} + a_g \wedge b_g) \wedge b_{g-2} - \frac{3}{g-1} b_{g-2} \wedge \omega] \wedge [a_1 \wedge b_{g-1} \wedge b_g]}_{z_{g-2}}.
\end{align*}

We define a vector $z_{g-1}$: \begin{align*}
    - E_{1(g-2)} \underbrace{[(a_1 \wedge b_1 + a_{g-2} \wedge b_{g-2} + a_g \wedge b_g) \wedge b_{g-1} - \frac{3}{g-1} b_{g-1} \wedge \omega] \wedge [a_1 \wedge b_1 \wedge b_g - \frac{1}{g-1} b_g \wedge \omega]}_{\text{Type 3}} \\ = \underbrace{[(a_1 \wedge b_1 + a_{g-2} \wedge b_{g-2} + a_g \wedge b_g) \wedge b_{g-1} - \frac{3}{g-1} b_{g-1} \wedge \omega] \wedge [a_1 \wedge b_{g-2} \wedge b_g ]}_{z_{g-1}}.
\end{align*}

We define a vector $z_g$: \begin{align*}
    - E_{1(g-2)} \underbrace{[(a_1 \wedge b_1 + a_{g-2} \wedge b_{g-2} + a_{g-1} \wedge b_{g-1}) \wedge b_g - \frac{3}{g-1} b_g \wedge \omega] \wedge [a_1 \wedge b_1 \wedge b_{g-1} - \frac{1}{g-1} b_{g-1} \wedge \omega]}_{\text{Type 3}}  \\ = \underbrace{[(a_1 \wedge b_1 + a_{g-2} \wedge b_{g-2} + a_{g-1} \wedge b_{g-1}) \wedge b_g - \frac{3}{g-1} b_g \wedge \omega] \wedge [a_1 \wedge b_{g-2} \wedge b_{g-1}]}_{z_g}.
\end{align*}

The sum \begin{align*} (\sum_{i=2}^{g-3} x_i ) + y + z_{g-2} - z_{g-1} + z_g  \end{align*} is a highest weight vector for $\Phi_{1,0, \dots ,0,1,0,0}$. \\

$\boxed{(1,0, \dots,0,1,1,0)}$: $$E_{2(g-1)} \circ E_{3g} \underbrace{[a_1 \wedge b_{g-1} \wedge b_g] \wedge [b_{g-2} \wedge b_2 \wedge b_3]}_{\text{Type 2}} = [a_1 \wedge b_{g-1} \wedge b_g] \wedge [b_{g-2} \wedge b_{g-1} \wedge b_g]$$ is a highest weight vector for $\Phi_{1,0, \dots,0,1,1,0}$. \\

$\boxed{(1,0, \dots ,0,1,0,0,1)}$: 

\hspace*{-1.5cm}
\begin{minipage}{\textwidth}
\begin{align*} - E_{2g} (\underbrace{[b_{g-2} \wedge b_{g-1} \wedge b_g] \wedge [a_1 \wedge b_{g-3} \wedge b_2] + [b_{g-3} \wedge b_{g-1} \wedge b_g] \wedge [a_1 \wedge b_{g-2} \wedge b_2] - [b_{g-3} \wedge b_{g-2} \wedge b_g] \wedge [a_1 \wedge b_{g-1} \wedge b_2]}_{\text{Type 2}}) \\ = [b_{g-2} \wedge b_{g-1} \wedge b_g] \wedge [a_1 \wedge b_{g-3} \wedge b_g] - [b_{g-3} \wedge b_{g-1} \wedge b_g] \wedge [a_1 \wedge b_{g-2} \wedge b_g] + [b_{g-3} \wedge b_{g-2} \wedge b_g] \wedge [a_1 \wedge b_{g-1} \wedge b_g]
\end{align*} \end{minipage} \vspace{1mm}

is a highest weight vector for $\Phi_{1,0, \dots ,0,1,0,0,1}$. \\

$\boxed{(1,0,1,0, \dots ,0,2)}$: $$- E_{5g} \circ E_{14} \underbrace{[a_1 \wedge a_2 \wedge b_g] \wedge [a_4 \wedge a_3 \wedge b_5]}_{\text{Type 2}} = [a_1 \wedge a_2 \wedge b_g] \wedge [a_1 \wedge a_3 \wedge b_g]$$ is a highest weight vector for $\Phi_{1,0,1,0, \dots ,0,2}$. \\ 

$\boxed{(1,1,0, \dots ,0,1,1)}$: 

$$- E_{4g} \circ E_{13} \underbrace{[a_1 \wedge a_2 \wedge b_g] \wedge [a_3 \wedge b_{g-1} \wedge b_4]}_{\text{Type 2}} = [a_1 \wedge a_2 \wedge b_g] \wedge [a_1 \wedge b_{g-1} \wedge b_g]$$ is a highest weight vector for $\Phi_{1,1,0, \dots ,0,1,1}$. \\

$\boxed{(1,1,0, \dots ,0,1,0,0)}$: 

\begin{align*}
    E_{13} (\underbrace{[a_1 \wedge a_2 \wedge b_{g-2}] \wedge [a_3 \wedge b_{g-1} \wedge b_g] - [a_1 \wedge a_2 \wedge b_{g-1}] \wedge [a_3 \wedge b_{g-2} \wedge b_g] + [a_1 \wedge a_2 \wedge b_g] \wedge [a_3 \wedge b_{g-2} \wedge b_{g-1}]}_{\text{Type 2}}) \\ = [a_1 \wedge a_2 \wedge b_{g-2}] \wedge [a_1 \wedge b_{g-1} \wedge b_g] - [a_1 \wedge a_2 \wedge b_{g-1}] \wedge [a_1 \wedge b_{g-2} \wedge b_g] + [a_1 \wedge a_2 \wedge b_g] \wedge [a_1 \wedge b_{g-2} \wedge b_{g-1}]
\end{align*} is a highest weight vector for $\Phi_{1,1,0, \dots ,0,1,0,0}$.  \\

$\boxed{(2,0, \dots ,0,1,0)}$: Fix $i \in \{ 2, \dots , g-2 \}$. We define a vector $x_i$:

\begin{align*}
    X_i \circ E_{i4} \circ E_{13} \underbrace{[a_1 \wedge a_i \wedge b_g] \wedge [a_3 \wedge a_4 \wedge b_{g-1}]}_{\text{Type 2}} \\ = \underbrace{[a_1 \wedge a_i \wedge b_g] \wedge [a_1 \wedge b_i \wedge b_{g-1}] + [a_1 \wedge b_i \wedge b_g] \wedge [a_1 \wedge a_i \wedge b_{g-1}]}_{x_i}.
\end{align*}

We define a vector $y$:

\begin{align*}
     E_{1(g-1)} \underbrace{[a_{g-1} \wedge b_{g-1} \wedge b_g - \frac{1}{g-1} b_g \wedge \omega] \wedge [a_1 \wedge (a_{g-1} \wedge b_{g-1} + a_g \wedge b_g) - \frac{2}{g-1} a_1 \wedge \omega]}_{\text{Type 3}} \\ =
    \underbrace{[a_1 \wedge b_{g-1} \wedge b_g] \wedge [a_1 \wedge (a_{g-1} \wedge b_{g-1} + a_g \wedge b_g) - \frac{2}{g-1} a_1 \wedge \omega]}_{y}.
\end{align*}

The sum
 \begin{align*} (\sum_{i=2}^{g-2} x_i) + y  \end{align*} is a highest weight vector for $\Phi_{2,0, \dots ,0,1,0}$. \\

$\boxed{(2,0, \dots ,0,1,0,1)}$: 
$$ - E_{3g} \circ E_{12} \underbrace{[a_1 \wedge b_{g-1} \wedge b_g] \wedge [a_2 \wedge b_{g-2} \wedge b_3]}_{\text{Type 2}} = [a_1 \wedge b_{g-1} \wedge b_g] \wedge [a_1 \wedge b_{g-2} \wedge b_g]$$ is a highest weight vector for $\Phi_{2,0, \dots ,0,1,0,1}$.

\section{Determination of $\operatorname{ker}(\tau^*)$ for $\mathcal{HI}_{g,1}$} \label{section:taupuncture}

We now extend our results to the case of the handlebody Torelli group $\mathcal{HI}_{g,1}$ with one marked point. We recall that the Johnson homomorphism for $\mathcal{I}_{g,1}$ has image $\bigwedge^3 H$ and the restriction to the handlebody group has image $U := \tau(\mathcal{HI}_{g,1})$. See Proposition \ref{proposition:U}. Since $U = \UQ \oplus H$, where $\UQ := \tau(\mathcal{HI}_g)$ as in the previous sections, we have $$H^2(U; \Q) \cong \bigwedge^2 U_{\Q}^* = \bigwedge^2 \UQ_{\Q}^* \oplus [\UQ_{\Q}^* \otimes H_{\Q}] \oplus \bigwedge^2 H_{\Q}.$$ Recall that $H_{\Q}$ is canonically isomorphic to its dual, and we denote both groups by $H_{\Q}$. Each of $\bigwedge^2 \UQ_{\Q}^*$ and $\UQ_{\Q} \otimes H_{\Q}$ and $\bigwedge^2 H_{\Q}$ contains a copy of $(\bigwedge^2 H_{\Q}) / \Q$. For the rest of this section, we distinguish between the Johnson homomorphism on $\mathcal{HI}_g$ $$\tau_g: \mathcal{HI}_g \rightarrow \UQ$$ and on $\mathcal{HI}_{g,1}$ $$\tau_{g,1}: \mathcal{HI}_{g,1} \rightarrow U.$$

\begin{theorem} \label{theorem:kernelpunctured} The kernel of $\tau_{g,1}^*: H^2(U; \Q) \rightarrow H^2(\mathcal{HI}_{g,1}; \Q)$ decomposes into $\operatorname{SL}_g(\Q)$ representations as $$\operatorname{ker}(\tau_g^*) \oplus (\bigwedge^2 H_{\Q}) / \Q.$$ Moreover, $(\bigwedge^2 H_{\Q}) / \Q$ lies diagonally in $\bigwedge^2 \UQ_{\Q}^* \oplus [\UQ_{\Q}^* \otimes H_{\Q}] \oplus \bigwedge^2 H_{\Q}$. \end{theorem}

Theorem \ref{theorem:kernelpunctured} and its proof closely resemble the case for $\mathcal{I}_{g,1}$.

\begin{theorem}[{Morita, \cite[Theorem 6.5]{MoritaLinear}}] \label{theorem:moritapunctured} The kernel of $\tau_{\mathcal{I}_{g,1}}^*: H^2(\bigwedge^3 H; \Q) \rightarrow H^2(\mathcal{I}_{g,1}; \Q)$ decomposes into $\operatorname{Sp}_{2g}(\Q)$ representations as $$\operatorname{ker}(\tau_{\mathcal{I}_g}^*) \oplus (\bigwedge^2 H_{\Q}) / \Q.$$ Moreover, $(\bigwedge^2 H_{\Q}) / \Q$ lies diagonally in $\bigwedge^2 ((\bigwedge^3 H_{\Q}) / H_{\Q}) \oplus [((\bigwedge^3 H_{\Q}) / H_{\Q}) \otimes H_{\Q}] \oplus \bigwedge^2 H_{\Q}$.
    
\end{theorem}

To show that $(\bigwedge^2 H_{\Q}) / \Q$ is in $\operatorname{ker}(\tau^*_{g,1})$, we follow the method from Section \ref{section:lowerboundtau} and show it is in the image of the bracket map $[\cdot, \cdot]$. In the following subsection, we set ourselves up to do this calculation.

\subsection{The $\operatorname{SL}_g(\Q)$ maps $[\cdot, \cdot]$ and $p$ and $q$} We recall from Section \ref{subsection:higherMod} that there is a bracket map $$[\cdot, \cdot]: \bigwedge^2 \tau^{\Q}(\mathcal{I}_{g,1}) \rightarrow \tau_2^{\Q}(\mathcal{K}_{g,1}).$$ Expressing the images of these Johnson homomorphisms as $\operatorname{SL}_g(\Q)$ vector spaces gives us a map $$[\cdot, \cdot]: \bigwedge^2 (\bigwedge^3 H_{\Q}) \rightarrow \operatorname{Sym}^2(\bigwedge^2 H_{\Q}).$$ Letting $(a \wedge b) \leftrightarrow (c \wedge d)$ denote the vector $(a \wedge b) \otimes (c \wedge d) + (c \wedge d) \otimes (a \wedge b)$, the bracket map $[\cdot, \cdot]$ evaluates on generators of $\bigwedge^2 (\bigwedge^3 H_{\Q})$ as \begin{multline*} [a \wedge b \wedge c] \wedge [d \wedge e \wedge f] \mapsto \omega(a,d)(b \wedge c) \leftrightarrow (e \wedge f) - \omega(a,e)(b \wedge c)\leftrightarrow (d \wedge f) + \omega(a,f)(b \wedge c) \leftrightarrow (d \wedge e) \\ - \omega(b,d)(a \wedge c) \leftrightarrow (e \wedge f) + \omega(b,e)(a \wedge c) \leftrightarrow (d \wedge f) - \omega(b,f)(a \wedge c) \leftrightarrow (d \wedge e) \\ + \omega(c,d)(a \wedge b) \leftrightarrow (e \wedge f) - \omega(c,e)(a \wedge b) \leftrightarrow (d \wedge f) + \omega(c,f)(a \wedge b) \leftrightarrow (d \wedge e). \end{multline*} See \cite[Equation 4.2]{MoritaLinear}. The image $\tau_2^{\Q}(\mathcal{K}_{g,1})$ contains the module $(\bigwedge^2 H_{\Q}) / \Q$.
We will restrict $[\cdot, \cdot]$ to $\bigwedge^2 U_{\Q}$ and see that the image still contains $(\bigwedge^2 H_{\Q}) / \Q$. \\

Let $\gamma_2(\pi) / \gamma_3( \pi) := \gamma_2(\pi_1(\Sigma_{g,1})) / \gamma_3(\pi_1(\Sigma_{g,1}))$ be the second term in the graded quotient of the lower central series\footnote{The group $\bigoplus_{n \geq 1} \gamma_n(\pi) / \gamma_{n+1}( \pi)$ is naturally isomorphic as a Lie algebra to the free Lie algebra of $H$. See \cite{Serre}. We can thus express its elements as linear combinations of $[x,[y,z]]$ for $x,y,z \in H$.} of $\pi_1(\Sigma_{g,1})$. We define the map $$p: \operatorname{Sym}^2(\bigwedge^2 H_{\Q}) \rightarrow H_{\Q} \otimes \gamma_2(\pi) / \gamma_3( \pi)$$ on generators as $$(a \wedge b) \leftrightarrow (c \wedge d) \mapsto a \otimes [b, [c,d]] - b \otimes [a,[c,d]] + c \otimes [d, [a,b]] - d \otimes [c,[a,b]].$$

We also define the map $$q: H_{\Q} \otimes \gamma_2(\pi) / \gamma_3( \pi) \rightarrow \bigwedge^2 H_{\Q}$$ on generators as $$a \otimes [b, [c,d]] \mapsto 2 \omega(a,b) c \wedge d + \omega(a,c) b \wedge d - \omega(a,d) b \wedge c.$$ This sets up the following lemma.

\begin{lemma}[{Morita, \cite[Lemma 6.4]{MoritaLinear}}] \label{lemma:compositionbracketpq} The following composition of $\operatorname{SL}_g(\Q)$ maps detects $\bigwedge^2 H_{\Q}$ in the image of $[\cdot, \cdot]$: $$\bigwedge^2 \bigwedge^3 H_{\Q} \xrightarrow{[\cdot, \cdot]} \operatorname{Sym}^2(\bigwedge^2 H_{\Q}) \xrightarrow{p} H_{\Q} \otimes \gamma_2(\pi) / \gamma_3( \pi) \xrightarrow{q} \bigwedge^2 H_{\Q}.$$
    
\end{lemma}

In the proof of Theorem \ref{theorem:kernelpunctured}, we will apply the composition $q \circ p \circ [\cdot, \cdot]$ to vectors from each of $\bigwedge^2 \UQ_{\Q}$ and $\UQ_{\Q} \otimes H_{\Q}$ and $\bigwedge^2 H_{\Q}$. To simplify these calculations, we introduce Lemma \ref{lemma:evaluateQP}.

\begin{lemma} \label{lemma:evaluateQP}
    Vectors in $\operatorname{Sym}^2(\bigwedge^2 H_{\Q})$ evaluate in the following ways under the composition $q \circ p$: 
    \textnormal{\begin{align*}  
\text{Type } \pm 1 \; (i=1,2): \quad & \pm (a_1 \wedge a_2) \leftrightarrow (a_i \wedge b_i) \xrightarrow{q \circ p} \pm 6 a_1 \wedge a_2 \\
\text{Type } \pm 2 \; (3 \leq i \leq g): \quad & \pm (a_1 \wedge a_2) \leftrightarrow (a_i \wedge b_i) \xrightarrow{q \circ p} \pm 4 a_1 \wedge a_2 \\
\text{Type } \pm 3 \; (3 \leq i \leq g): \quad & \pm (a_1 \wedge a_i) \leftrightarrow (a_2 \wedge b_i) \xrightarrow{q \circ p} \pm 2 a_1 \wedge a_2 \\
\text{Type } \pm 4 \; (3 \leq i \leq g): \quad & \pm  (a_1 \wedge b_i) \leftrightarrow (a_2 \wedge a_i) \xrightarrow{q \circ p} \mp 2 a_1 \wedge a_2. \\
\end{align*}}
\end{lemma}

\begin{proof} We show this by direct computation for the Type $+1$ vector when $i=1$, and the other computations follow similarly. We apply the composition $q \circ p$. We use colored text to help keep track of pieces, but the color is not necessary for following the computation.
\begin{align*}
    (a_1 \wedge a_2) \leftrightarrow (a_1 \wedge b_1) \xrightarrow{p} \green{ a_1} \otimes [a_2, [a_1, \green{b_1}]] - a_2 \otimes [a_1, [a_1, b_1]] + \purple{a_1} \otimes [\purple{b_1},[a_1,a_2]] - \blue{b_1} \otimes [\blue{a_1},[\blue{a_1},a_2]] \\ \xrightarrow{q} \green{- \; a_2 \wedge a_1} \purple{\; + \; 2a_1 \wedge a_2}  \blue{\; + \; 2a_1 \wedge a_2 + a_1 \wedge a_2} = 6a_1 \wedge a_2. 
\end{align*} \end{proof}

We are now ready to prove Theorem \ref{theorem:kernelpunctured}.

\begin{proof}[Proof of Theorem \ref{theorem:kernelpunctured}] The Hochschild-Serre spectral sequence of the rational cohomology of the extension $$1 \rightarrow \pi_1 (\Sigma_g) \rightarrow \mathcal{HI}_{g,1} \rightarrow \mathcal{HI}_g \rightarrow 1$$ degenerates at the $E_2$ page \cite[Proposition 3.1]{MoritaBundle} so we have an isomorphism\footnote{Note that $\Sigma_g$ is a $K(\pi_1(\Sigma_g),1)$.} $$H^2(\mathcal{HI}_{g,1}; \Q) \cong  H^2(\mathcal{HI}_{g}; \Q) \oplus [H^1(\mathcal{HI}_{g}; \Q) \otimes H^1(\Sigma_g; \Q)] \oplus H^2(\Sigma_g; \Q).$$

This decomposition respects the cup product structure of $H^*(\mathcal{HI}_{g,1}; \Q)$, so the structure in the $H^2(\mathcal{HI}_{g}; \Q)$ piece will be the same as we found in Theorem \ref{theorem:kerneltauclosed}. Specifically, the quotient $$H_1(\mathcal{HI}_{g,1}; \Q) \xrightarrow{\tau_{g,1}^{\Q}} U_{\Q} \rightarrow U_{\Q} /H_{\Q} = \UQ_{\Q}$$ dualizes to a subspace $\UQ_{\Q}^* \subset H^1(\mathcal{HI}_{g,1}; \Q)$, and the kernel of the restriction of the cup product $$\cup: \bigwedge^2 \UQ_{\Q}^* \rightarrow H^2(\mathcal{HI}_{g,1}; \Q)$$ is exactly $\operatorname{ker}(\tau_g^*)$. \\

Now consider the submodule $\UQ_{\Q}^* \otimes H_{\Q}$ of $H^2(U; \Q)$, which is generated by elements $x \cup y$ for $x \in \UQ_{\Q}^*$ and $y \in H_{\Q} \subset H^1(\mathcal{HI}_{g,1}; \Q)$. This submodule injects into $H^1(\mathcal{HI}_{g}; \Q) \otimes  H^1(\Sigma_g; \Q)$, so $x \cup y \neq 0$ for $x,y \neq 0$. \\

The decomposition respects the cup product structure of $H^*(\Sigma_g; \Q)$: $$\cup: \bigwedge^2 H^1(\Sigma_g; \Q) \rightarrow H^2(\Sigma_g; \Q) \cong \Q.$$
The vector space $\bigwedge^2 H^1(\Sigma_g; \Q) \cong \bigwedge^2 H_{\Q}$ decomposes into $\operatorname{SL}_g(\Q)$ modules as $$\bigwedge^2 H_{\Q} = \Phi_{0,1,0, \dots, 0} \oplus \Phi_{1,0, \dots, 0,1} \oplus \Phi_{0, \dots, 0,1,0} \oplus \Q,$$ so there is a copy of $$(\bigwedge^2 H_{\Q}) / \Q = \Phi_{0,1,0, \dots, 0} \oplus \Phi_{1,0, \dots, 0,1} \oplus \Phi_{0, \dots, 0,1,0}$$ inside of $\operatorname{ker}(\tau_{g,1}^*)$. We show that each of $\bigwedge^2 \UQ_{\Q}$ and $\UQ_{\Q} \otimes H_{\Q}$ and $\bigwedge^2 H_{\Q}$ maps surjectively to $(\bigwedge^2 H_{\Q}) / \Q$. We do this by applying the composition $q \circ p \circ [ \cdot, \cdot]$ from Lemma \ref{lemma:compositionbracketpq}.  Following arguments from Section \ref{section:lowerboundtau}, this suffices to prove Theorem \ref{theorem:kernelpunctured}. 

The vector $a_1 \wedge a_2$ is a highest weight vector for $\Phi_{0,1,0, \dots, 0}$, and it is easy to see that the highest weight vectors for $\Phi_{1,0, \dots, 0,1}$ and $\Phi_{0, \dots, 0,1,0}$ are in its $\mathcal{H}_{g,1}$-orbit.\footnote{See our calculation for $\Gamma_{0,1}$ in Section \ref{subsection:01}.} It suffices to map to $a_1 \wedge a_2$ from each of $\bigwedge^2 \UQ_{\Q}$ and $\UQ_{\Q} \otimes H_{\Q}$ and $\bigwedge^2 H_{\Q}$. 

In the following calculations, we will make use of the results from Lemma \ref{lemma:evaluateQP}. We use colored text to help keep track of pieces, but the color is not necessary for following the computation.

$\boxed{\bigwedge^2 H_{\Q}}:$ We will show that \begin{equation} \label{equation:wedge2H} (q \circ p \circ [ \cdot, \cdot])([a_1 \wedge \omega] \wedge [a_2 \wedge \omega]) = (-4g -4) a_1 \wedge a_2. \end{equation}

First, we expand \begin{align*} [a_1 \wedge \omega] \wedge [a_2 \wedge \omega] = \sum_{3 \leq i \neq j \leq g} [a_1 \wedge a_i \wedge b_i] \wedge [a_2 \wedge a_j \wedge b_j] + {\green{[a_1 \wedge a_2 \wedge b_2] \wedge [a_2 \wedge a_1 \wedge b_1]}} \\ + \sum_{i=3}^g (\purple{[a_1 \wedge a_i \wedge b_i] \wedge [a_2 \wedge a_i \wedge b_i]} + \blue{[a_1 \wedge a_i \wedge b_i] \wedge [a_2 \wedge a_1 \wedge b_1]} + \orange{[a_1 \wedge a_2 \wedge b_2] \wedge [a_2 \wedge a_i \wedge b_i]}) \end{align*} Applying the bracket map $[ \cdot, \cdot]$, we get \begin{align*}
    [a_1 \wedge \omega] \wedge [a_2 \wedge \omega] \xrightarrow{[\cdot, \cdot]} {\green{\underbrace{(a_2 \wedge b_2) \leftrightarrow (a_2 \wedge a_1)}_{{\text{Type } -1}} \underbrace{-(a_1 \wedge a_2) \leftrightarrow (a_1 \wedge b_1)}_{\text{Type } -1}}} \\ + \sum_{i=3}^g [\purple{\underbrace{- (a_1 \wedge b_i) \leftrightarrow (a_2 \wedge a_i)}_{\text{Type } -4} \underbrace{+ (a_1 \wedge a_i) \leftrightarrow (a_2 \wedge b_i)}_{\text{Type } +3}} \blue{\underbrace{+ (a_i \wedge b_i) \leftrightarrow (a_2 \wedge a_1)}_{\text{Type } -2}} \orange{\underbrace{- (a_1 \wedge a_2) \leftrightarrow (a_i \wedge b_i)}_{\text{Type } +2}}].
\end{align*} By Lemma \ref{lemma:evaluateQP}, $(q \circ p \circ [ \cdot, \cdot])([a_1 \wedge \omega] \wedge [a_2 \wedge \omega])$ equals the sum $${\green{- \; 6a_1 \wedge a_2 \; - \; 6a_1 \wedge a_2}} + (g-2)[\purple{2a_1 \wedge a_2 + 2a_1 \wedge a_2} \blue{\; - \; 4a_1 \wedge a_2} \orange{\; - \; 4a_1 \wedge a_2}],$$ which simplifies to $(-4g-4)a_1 \wedge a_2$ as desired. \\

$\boxed{\bigwedge^2 \UQ_{\Q}}:$ We will show that $$(q \circ p \circ [ \cdot, \cdot])([a_1 \wedge a_3 \wedge b_3 - \frac{1}{g-1} a_1 \wedge \omega] \wedge [a_2 \wedge a_4 \wedge b_4 - \frac{1}{g-1} a_2 \wedge \omega]) = \frac{-8}{(g-1)^2} a_1 \wedge a_2.$$

First, we expand \begin{align*} [a_1 \wedge a_3 \wedge b_3 - \frac{1}{g-1} a_1 \wedge \omega] \wedge [a_2 \wedge a_4 \wedge b_4 - \frac{1}{g-1} a_2 \wedge \omega] = [a_1 \wedge a_3 \wedge b_3] \wedge [a_2 \wedge a_4 \wedge b_4]  \\ - \frac{1}{g-1}(\green{[a_1 \wedge a_3 \wedge b_3] \wedge [a_2 \wedge a_1 \wedge b_1]} \purple{+ [a_1 \wedge a_3 \wedge b_3] \wedge [a_2 \wedge a_3 \wedge b_3]} + \sum_{i=4}^g [a_1 \wedge a_3 \wedge b_3] \wedge [a_2 \wedge a_i \wedge b_i]) \\ - \frac{1}{g-1}(\blue{[a_1 \wedge a_2 \wedge b_2] \wedge [a_2 \wedge a_4 \wedge b_4]} \orange{+ [a_1 \wedge a_4 \wedge b_4] \wedge [a_2 \wedge a_4 \wedge b_4]} + \sum_{\substack{i=3 \\ i \ne 4}}^g [a_1 \wedge a_i \wedge b_i] \wedge [a_2 \wedge a_4 \wedge b_4]) \\ + \frac{1}{(g-1)^2} \darkpurple{[a_1 \wedge \omega] \wedge [a_2 \wedge \omega]}.
\end{align*} 

\noindent We already know that $(q \circ p \circ [\cdot, \cdot])(\darkpurple{[a_1 \wedge \omega] \wedge [a_2 \wedge \omega]}) = \darkpurple{(-4g-4) a_1 \wedge a_2}$ (see Equation \ref{equation:wedge2H}), so we will carry this term along until the end of our calculation. Applying the bracket map $[ \cdot, \cdot]$, we get \begin{align*} [a_1 \wedge a_3 \wedge b_3 - \frac{1}{g-1} a_1 \wedge \omega] \wedge [a_2 \wedge a_4 \wedge b_4 - \frac{1}{g-1} a_2 \wedge \omega] \xrightarrow{[\cdot,\cdot]} - \frac{1}{g-1} [\green{\underbrace{(a_3 \wedge b_3) \leftrightarrow (a_2 \wedge a_1)}_{\text{Type } -1}} \\ \purple{\underbrace{ \; - \; (a_1 \wedge b_3) \leftrightarrow (a_2 \wedge a_3)}_{\text{Type } -4} \underbrace{ \; + \; (a_1 \wedge a_3) \leftrightarrow (a_2 \wedge b_3)}_{\text{Type } +3} } \blue{\underbrace{ \; - \; (a_1 \wedge a_2) \leftrightarrow (a_4 \wedge b_4)}_{\text{Type } -1}} \orange{\underbrace{ \; - \; (a_1 \wedge b_4) \leftrightarrow (a_2 \wedge a_4)}_{\text{Type } -4}}  \\ \orange{\underbrace{ \; + \; (a_1 \wedge a_4) \leftrightarrow (a_2 \wedge b_4)}_{\text{Type } +3}}]  + \frac{1}{(g-1)^2} [\cdot, \cdot] (\darkpurple{[a_1 \wedge \omega] \wedge [a_2 \wedge \omega]}).  \end{align*} By Lemma \ref{lemma:evaluateQP}, $$(q \circ p \circ [ \cdot, \cdot])([a_1 \wedge a_3 \wedge b_3 - \frac{1}{g-1} a_1 \wedge \omega] \wedge [a_2 \wedge a_4 \wedge b_4 - \frac{1}{g-1} a_2 \wedge \omega])$$ equals \begin{align*} - \frac{1}{g-1}[\green{ \; - \; 6a_1 \wedge a_2} \purple{ \; + \; 2a_1 \wedge a_2 \; + \; 2a_1 \wedge a_2} \blue{ \; - \; 6a_1 \wedge a_2} \orange{ \; + \; 2a_1 \wedge a_2 \; + \; 2a_1 \wedge a_2}] + \frac{\darkpurple{(-4g-4)}}{(g-1)^2} {\darkpurple{a_1 \wedge a_2}}, \end{align*} which simplifies to $\frac{-8}{(g-1)^2} a_1 \wedge a_2$ as desired. \\

$\boxed{\UQ_{\Q} \otimes H_{\Q}}:$ We will show that $$(q \circ p \circ [ \cdot, \cdot])([a_3 \wedge \omega] \wedge [a_1 \wedge a_2 \wedge b_3]) = 4g a_1 \wedge a_2.$$ First, we expand \begin{align*}
    [a_3 \wedge \omega] \wedge [a_1 \wedge a_2 \wedge b_3] = \green{[a_3 \wedge a_1 \wedge b_1] \wedge [a_1 \wedge a_2 \wedge b_3]} \purple{ \; + \; [a_3 \wedge a_2 \wedge b_2] \wedge [a_1 \wedge a_2 \wedge b_3]} \\ + \sum_{i=4}^g \blue{[a_3 \wedge a_i \wedge b_i] \wedge [a_1 \wedge a_2 \wedge b_3]}.
\end{align*}  Applying the bracket map $[ \cdot, \cdot]$, we get \begin{align*}
    [a_3 \wedge \omega] \wedge [a_1 \wedge a_2 \wedge b_3] \xrightarrow{[\cdot,\cdot]} \green{\underbrace{(a_1 \wedge b_1) \leftrightarrow (a_1 \wedge a_2)}_{\text{Type } +1} \underbrace{\; - \; (a_3 \wedge a_1) \leftrightarrow (a_2 \wedge b_3)}_{\text{Type } +3}} \purple{\underbrace{\; + \; (a_2 \wedge b_2) \leftrightarrow (a_1 \wedge a_2)}_{\text{Type } +1}} \\ \purple{\underbrace{\; + \; (a_3 \wedge a_2) \leftrightarrow (a_1 \wedge b_3)}_{\text{Type } -3}} + \sum_{i=4}^g \blue{\underbrace{(a_i \wedge b_i) \leftrightarrow (a_1 \wedge a_2)}_{\text{Type } +2}}.
\end{align*} By Lemma \ref{lemma:evaluateQP}, $(q \circ p \circ [ \cdot, \cdot])([a_3 \wedge \omega] \wedge [a_1 \wedge a_2 \wedge b_3])$ equals the sum $$\green{6a_1 \wedge a_2 \; + \; 2a_1 \wedge a_2} \purple{ \; + \; 6a_1 \wedge a_2 \; - \; 2 a_1 \wedge a_2} + (g-3) \blue{4 a_1 \wedge a_2},$$ which simplifies to $4g a_1 \wedge a_2$ as desired. \end{proof} 

\section{Determination of $\operatorname{ker}(\tau^*)$ for $\mathcal{HI}_g^1$} \label{section:tauboundary}

Now we extend our results to the case of the handlebody Torelli group $\mathcal{HI}_g^1$ with one embedded disk. In this section, we distinguish between the Johnson homomorphism on $\mathcal{HI}_{g,1}$ $$\tau_{g,1}: \mathcal{HI}_{g,1} \rightarrow U$$ and on $\mathcal{HI}_g^1$ $$\tau_g^1: \mathcal{HI}_g^1 \rightarrow U.$$ We prove the following theorem. 

\begin{theorem} \label{theorem:kernelboundary}
    The kernel of $$(\tau_g^1)^*: H^2(U; \Q) \rightarrow H^2(\mathcal{HI}_g^1; \Q)$$ decomposes into $\operatorname{SL}_g(\Q)$ representations as $$\operatorname{ker}((\tau_g^1)^*) \cong \operatorname{ker}(\tau_{g,1}^*) \oplus \Q.$$
\end{theorem}

This closely resembles results for the Torelli group $\mathcal{I}_g^1$. \begin{theorem}[{Habegger-Sorger, \cite[Theorem 2.2]{Habegger}}] \label{theorem:kernelboundarytorelli}
    The kernel of $$\tau_{\mathcal{I}_g^1}^*: H^2(\bigwedge^3 H; \Q) \rightarrow H^2(\mathcal{I}_g^1; \Q)$$ decomposes into $\operatorname{Sp}_{2g}(\Q)$ representations as $$\operatorname{ker}(\tau_{\mathcal{I}_g^1}^*) \cong \operatorname{ker}(\tau_{\mathcal{I}_{g,1}}^*) \oplus \Q.$$
\end{theorem}

Before we prove Theorem \ref{theorem:kernelboundary}, we introduce a series of lemmas. First, we show that the boundary twist (the Dehn twist about the embedded disk in $\Sigma_g^1$), which we denote $T_{\partial}$, dies in $H_1(\mathcal{HI}_g^1; \Q)$. We then use this fact to understand the differentials in a spectral sequence that sets us up to prove Theorem \ref{theorem:kernelboundary}.

\subsection{Boundary twists} We start with a conjugacy relation in Lemma \ref{lemma:conjugacy}. To state this lemma, we briefly introduce some concepts regarding symplectic spaces. See \cite{JohnsonConjugacy} for more details. In particular, the statement and proof of Lemma \ref{lemma:conjugacy} closely resemble those of \cite[Theorem 1B]{JohnsonConjugacy}.

Consider a bounding pair map $T_{\gamma} T_{\delta}^{-1}$ for which $\gamma \cup \delta$ bounds a subsurface $S \subset \Sigma_g^1$. Then: \begin{itemize}
    \item The subspace $H_1(S; \Z) \subset H$ has a primitive\footnote{While the details are not necessary for our proof, see the erratum in \cite{JohnsonIII} for the definition of primitive (the definition in \cite{JohnsonConjugacy} is incorrect).} bilinear form $\omega_S$ coming from the restriction of the intersection form $\omega$ on $H$.

    \item The radical of $H_1(S; \Z)$ is the set of all $x \in H_1(S; \Z)$ such that $\omega_S(x,y) =0$ for all $y \in H_1(S; \Z)$. The radical of $H_1(S; \Z)$ is generated by $[\gamma]$.
\end{itemize} This map $T_{\gamma} T_{\delta}^{-1}$ gives rise to a \textit{polarized subspace} of $H$: the primitive subspace $H_1(S; \Z)$ paired with the element $[\gamma]$ which generates the one dimensional radical of $H_1(S; \Z)$. \\

We are ready to state the following lemma.

\begin{lemma}[Conjugacy relation] \label{lemma:conjugacy}
    Let $T_{\gamma} T_{\delta}^{-1}$ and $T_{\gamma'} T_{\delta'}^{-1}$ be bounding pair maps about curves which bound disks in $\mathcal{V}_g^1$. Then $T_{\gamma} T_{\delta}^{-1}$ and $T_{\gamma'} T_{\delta'}^{-1}$ are conjugate in $\mathcal{HI}_g^1$ if and only if they have the same polarized subspace in $H$.
\end{lemma}

\begin{proof}
    The forward direction follows from the analogous conjugacy relation in $\mathcal{I}_g^1$ (see \cite[Theorem 1B]{JohnsonConjugacy}). Since we have $fT_{\gamma}f^{-1} = T_{f(\gamma)}$ (see \cite[Fact 3.7]{Primer}), it suffices to show that there exists $f \in \mathcal{HI}_g^1$ such that $f(\gamma) = \gamma'$ and $f(\delta) = \delta'$. Cutting along $\gamma \cup \delta$ (resp. $\gamma' \cup \delta'$) splits $\Sigma_g^1$ into subsurfaces $S$ and $\bar{S}$ (resp. $S'$ and $\bar{S}'$). Since they induce the same polarized subspace in $H$, we have the homeomorphisms $S \cong S' \cong \Sigma_{g_1}^2$ and $\bar{S} \cong \bar{S}' \cong \Sigma_{g_2}^3$ with $g_1 + g_2 = g-1$. The pairs $S, S'$ and $\bar{S}, \bar{S}'$ each bound homeomorphic handlebodies (with embedded disks) via homeomorphisms that send $\gamma$ to $\gamma'$ and $\delta$ to $\delta'$. Piece these homeomorphisms together along $\{ \gamma, \delta \}$ to get $F \in \mathcal{H}_g^1$ such that $F(S) = S'$ and $F(\bar{S}) = \bar{S}'$ and $F(\gamma) = \gamma'$ and $F(\delta) = \delta'$. Note that $F$ leaves invariant the homology class $[\gamma]$ and the subgroups $A := H_1(S; \Z)= H_1(S'; \Z)$ and $B := H_1(\bar{S}; \Z) = H_1(\bar{S}'; \Z)$.

    Now we address the fact that $F$ may not lie in $\mathcal{I}_g^1$. The induced map $F_*: H \rightarrow H$ preserves the subgroups $A$ and $B$. The restriction $F_*|_A: A \rightarrow A$ acts via a matrix in $\Psi(\mathcal{H}_{g_1}^2) \subset \operatorname{Sp}_{2g_1}(\Z)$ that fixes $\operatorname{rad}(A) = \langle [\gamma] \rangle$, so we can find\footnote{We denote by $\mathcal{H}(S)$ the mapping class group of the handlebody bounded by $S$ and contained in the original handlebody $\mathcal{V}_g^1$. We have $\mathcal{H}(S) \cong \mathcal{H}_{g_1}^2$.} $h \in \mathcal{H}(S)$ such that $h_* = F_*|_A$. Similarly, we can find $\bar{h} \in \mathcal{H}(\bar{S})$ such that $\bar{h}_* = F_*|_B$. Piecing $h, \bar{h}$ together along $\{ \gamma, \delta \}$, we get a map $H \in \mathcal{H}_g^1$ such that $H_*$ agrees with $F_*$ on $[\gamma]^{\perp}$. In particular, $F \circ H^{-1}$ acts trivially on $[\gamma]^{\perp} := \{ v \in H \; | \; \omega(v,[\gamma]) = 0 \}$ and so, by \cite[Lemma 2]{JohnsonConjugacy}, $\Psi(F \circ H^{-1}) = \Psi(T_{\gamma}^r)$ for some $r \in \Z$. Then $f := F \circ H^{-1} \circ T_{\gamma}^{-r}$ lies in $\mathcal{HI}_g^1$ and takes $\gamma$ to $\gamma'$ and $\delta$ to $\delta'$ as desired.
\end{proof} 

\begin{lemma}[Boundary twists are torsion] \label{lemma:boundarytwistdies}
    Let $g \geq 2$ and let $T_{\partial}$ be the Dehn twist about the embedded disk in $\Sigma_g^1$. Then $[T_{\partial}] = 0$ in $H_1(\mathcal{HI}_g^1; \Q).$
\end{lemma}

\begin{proof}
    In \cite[Section 6]{Putman}, Putman constructs two lantern relations: $$T_{\partial} = (T_{x_1} T_{x_2}^{-1}) (T_{y_1} T_{y_2}^{-1}) (T_{z_1} T_{z_2}^{-1})$$ and $$T_{\partial} = (T_{x_1'} T_{x_2'}^{-1}) (T_{y_1'} T_{y_2'}^{-1}) (T_{z_1'} T_{z_2'}^{-1}).$$ These are products of bounding pair maps around meridians so they lie in $\mathcal{HI}_g^1$. See \cite[Figure 11]{Putman}.  With this fact, Putman's proof \cite[Lemma 6.2]{Putman} applies completely to our situation. The details go as follows. We choose orientations so that, in $H$, we have $[x_1]= -[x_2]$ and $[x_1'] = -[x_2']$ and $[x_1] = [x_2']$. By Lemma \ref{lemma:conjugacy}, $T_{x_1} T_{x_2}^{-1}$ and $T_{x_2'} T_{x_1'}^{-1}$ are conjugate in $\mathcal{HI}_g^1$ (notice that we are using $T_{x_2'} T_{x_1'}^{-1}$ here instead of $T_{x_1'} T_{x_2'}^{-1}$). Passing to $H_1(\mathcal{HI}_g^1; \Q)$, we have $$[T_{x_1} T_{x_2}^{-1}] = - [T_{x_1'} T_{x_2'}^{-1}].$$ Similarly, we have $[T_{y_1} T_{y_2}^{-1}] = - [T_{y_1'} T_{y_2'}^{-1}]$ and $[T_{z_1} T_{z_2}^{-1}] = - [T_{z_1'} T_{z_2'}^{-1}].$ Combining these facts with Putman's lantern relations, we calculate: 
    \begin{align*}
     2 [T_{\partial}] = ([T_{x_1} T_{x_2}^{-1}] + [T_{y_1} T_{y_2}^{-1}] + [T_{z_1} T_{z_2}^{-1}]) \\ + ([T_{x_1'} T_{x_2'}^{-1}] + [T_{y_1'} T_{y_2'}^{-1}] + [T_{z_1'} T_{z_2'}^{-1}]) = 0. \end{align*}
\end{proof}

\subsection{Spectral sequence} Since the Hochschild-Serre spectral sequence of the extension $$1 \rightarrow \Z \rightarrow \mathcal{HI}_g^1 \rightarrow \mathcal{HI}_{g,1} \rightarrow 1$$ has two rows, it degenerates at the $E_3$ page and has $E_{2}^{0,2} = E_{\infty}^{0,2} = 0$. To understand $$H_2(\mathcal{HI}_g^1; \Q) = E_{\infty}^{1,1} \oplus E_{\infty}^{2,0},$$ it suffices to understand the differentials $$d_{02}: H_2(\mathcal{HI}_{g,1}; \Q) \rightarrow \Q$$ and $$d_{03}: H_3(\mathcal{HI}_{g,1}; \Q) \rightarrow H_1(\mathcal{HI}_{g,1}; \Q).$$ We have $E_{\infty}^{2,0} = \operatorname{ker}(d_{02})$ and $E_{\infty}^{1,1} = \operatorname{coker}(d_{03})$. We show in the following lemmas that $d_{02}$ and $d_{03}$ are surjective, and so $H_2(\mathcal{HI}_g^1; \Q) \cong H_2(\mathcal{HI}_{g,1}; \Q) / \Q$. 

\begin{lemma} \label{lemma:d02} When $g \geq 2$, the differential $d_{02}$ is surjective. \end{lemma}

\begin{proof} In the short exact sequence (see \cite[Section 7.6]{Brown}) $$0 \rightarrow \operatorname{coker}(d_{02}) \rightarrow H_1(\mathcal{HI}_g^1; \Q) \rightarrow H_1(\mathcal{HI}_{g,1}; \Q) \rightarrow 0,$$ the map $\Q \twoheadrightarrow \operatorname{coker}(d_{02}) \hookrightarrow H_1(\mathcal{HI}_g^1; \Q)$ is given by $1 \mapsto [T_{\partial}]$. By Lemma \ref{lemma:boundarytwistdies}, $\operatorname{coker}(d_{02}) = 0$. \end{proof}

\begin{lemma} \label{lemma:d03} When $g \geq 4$, the differential $d_{03}$ is surjective. \end{lemma}

\begin{proof} These differentials are $\mathcal{H}_{g,1}$-equivariant maps so it is enough to show that $d_{03}$ hits $[T_{C_1} T_{C_2}^{-1}]$, the homology class of bounding pair annulus twist as in Theorem \ref{theorem:Omori}. Equivalently, we show that $[T_{C_1} T_{C_2}^{-1}] \in H_1(\mathcal{HI}_{g,1}; \Q)$ vanishes in $H_2(\mathcal{HI}_g^1; \Q)$ in the following short exact sequence: \[\begin{tikzcd}[ampersand replacement=\&]
	0 \& {\operatorname{coker}(d_{03})} \& {H_2(\mathcal{HI}_g^1; \mathbb{Q})} \& {\operatorname{ker}(d_{02})} \& 0
	\arrow[from=1-1, to=1-2]
	\arrow[from=1-2, to=1-3]
	\arrow[from=1-3, to=1-4]
	\arrow[from=1-4, to=1-5]
\end{tikzcd}. \] We choose $C_1 \cup C_2$ to bound a genus 1 surface on the side without the embedded disk, and to bound a subsurface $T$ of genus at least $2$ on the other side.

\centerline{\includegraphics[scale=.5]{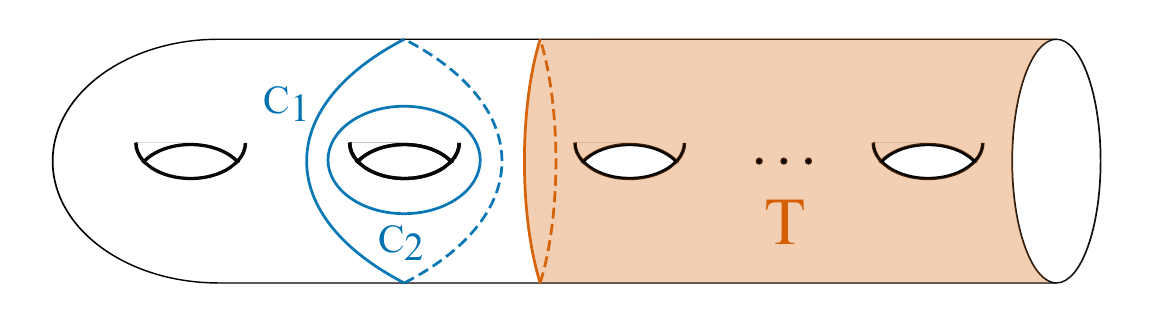}}
Consider the following commutative diagram of short exact sequences:

\[\begin{tikzcd}
	0 & {\mathbb{Z}} & {\langle T_{\partial}, T_{C_1} T_{C_2}^{-1} \rangle} & {\overline{\langle T_{C_1} T_{C_2}^{-1} \rangle}} & 0 \\
	1 & {\mathbb{Z}} & {\mathcal{HI}_g^1} & {\mathcal{HI}_{g,1}} & 1
	\arrow[from=1-1, to=1-2]
	\arrow[from=1-2, to=1-3]
	\arrow[from=1-2, to=2-2, equal]
	\arrow[from=1-3, to=1-4]
	\arrow["i", hook, from=1-3, to=2-3]
	\arrow[from=1-4, to=1-5]
	\arrow[hook, from=1-4, to=2-4]
	\arrow[from=2-1, to=2-2]
	\arrow[from=2-2, to=2-3]
	\arrow[from=2-3, to=2-4]
	\arrow[from=2-4, to=2-5] \end{tikzcd}\] The Hochschild-Serre spectral sequence $\{ E_r^{'p,q} \}$ of the top row degenerates at the $E_2'$ page. We have the isomorphism $\langle T_{\partial}, T_{C_1} T_{C_2}^{-1} \rangle \cong \Z^2$, and the inclusions in the commutative diagram induce a map of spectral sequences sending $$E_{\infty}^{'1,1} = H_1(\overline{\langle T_{C_1} T_{C_2}^{-1} \rangle}; \Q) \subset H_2(\Z^2; \Q)$$ to $$E_{\infty}^{1,1} = \operatorname{coker}(d_{03}) \subset H_2(\mathcal{HI}_g^1; \Q).$$ It suffices to show that the map $i_*: H_2(\Z^2; \Q) \rightarrow H_2(\mathcal{HI}_g^1; \Q)$ is the zero map. The map $i: \Z^2 \rightarrow \mathcal{HI}_g^1$ factors through $$\Z^2 \hookrightarrow \mathcal{HI}(T) \times \Z \rightarrow \mathcal{HI}_g^1$$ and $T_{\partial}$ dies in $H_1(\mathcal{HI}(T); \Q)$ by Lemma \ref{lemma:boundarytwistdies}. \end{proof} 

We are now ready to prove Theorem \ref{theorem:kernelboundary}.

\begin{proof}[Proof of Theorem~\ref{theorem:kernelboundary}]
    Consider the following commutative diagram: \[\begin{tikzcd}[ampersand replacement=\&]
	0 \& {\mathbb{Q}} \& {H^2(\mathcal{HI}_{g,1}; \mathbb{Q})} \& {H^2(\mathcal{HI}_g^1; \mathbb{Q})} \& 0 \\
	\&\& {\bigwedge^2 H^1(\mathcal{HI}_{g,1}; \mathbb{Q})} \& {\bigwedge^2 H^1(\mathcal{HI}_g^1; \mathbb{Q})} \\
	\&\& {\bigwedge^2 U_{\mathbb{Q}}^*} \& {\bigwedge^2 U_{\mathbb{Q}}^*}
	\arrow[from=1-1, to=1-2]
	\arrow[from=1-2, to=1-3]
	\arrow[from=1-3, to=1-4]
        \arrow[from=1-4, to=1-5]
	\arrow["\cup"', from=2-3, to=1-3]
	\arrow[equals, from=2-3, to=2-4]
	\arrow["\cup"', from=2-4, to=1-4]
	\arrow[hook, from=3-3, to=2-3]
	\arrow["{{\text{id}_{\bigwedge^2 U_{\mathbb{Q}}^*}}}", from=3-3, to=3-4]
	\arrow[hook, from=3-4, to=2-4]
        \arrow["\iota"', hook, bend right=55, from=1-2, to=3-3]
        \arrow["\tau_{g,1}^*", bend left=75, from=3-3, to=1-3]
        \arrow["(\tau_g^1)^*"', bend right=75, from=3-4, to=1-4]
\end{tikzcd}\] Combining results from Lemma \ref{lemma:d02} and Lemma \ref{lemma:d03}, the top row is exact. It follows that \begin{align*}
    (\tau_g^1)^* \circ \text{id}_{\bigwedge^2 U_{\mathbb{Q}}^*} \circ \iota(\Q) = 0. 
\end{align*} 
\end{proof}

Theorem \ref{theorem:kerneltau} follows from Theorem \ref{theorem:kerneltauclosed}, Theorem \ref{theorem:kernelpunctured}, and Theorem \ref{theorem:kernelboundary}.

\section{Decomposition of $\tau_2^{\Q}(\mathcal{HK}_{g,p}^b)$} \label{section:decomptau2}

In this section, we prove the following theorem.

\begin{theorem} \label{theorem:decomptau2}
    When $g \geq 5$, the image $\tau_2^{\Q}(\mathcal{HK}_{g,p}^b)$ decomposes into $\operatorname{SL}_g(\Q)$ representations as in Table \ref{table:decomptau2}. \end{theorem}

\begin{table}[h!]
\centering
\begin{tabular}{|c|c|c|c|}
\hline
\multirow{2}{*}{\textbf{Highest Weight}} & \multicolumn{3}{c|}{\textbf{Multiplicity}} \\
\cline{2-4}
& \rule{0pt}{2.6ex} \textbf{$\tau_2^{\Q}(\mathcal{HK}_g)$} & \textbf{$\tau_2^{\Q}(\mathcal{HK}_{g,1})$} & \textbf{$\tau_2^{\Q}(\mathcal{HK}_g^1)$} \\
\hline

$(0, \dots, 0)$ & 1 & 1 & 2 \\
$(0, \dots, 0, 2)$ & 1 & 1 & 1 \\
$(0, \dots, 0, 1, 0)$ & 0 & 1 & 1 \\
$(0, \dots, 0, 2, 0)$ & 1 & 1 & 1 \\
$(0, 1, 0, \dots, 0)$ & 0 & 1 & 1 \\
$(0, 1, 0, \dots, 0, 1, 0)$ & 1 & 1 & 1 \\
$(1, 0, \dots, 0, 1)$ & 1 & 2 & 2 \\
$(1, 0, \dots, 0, 1, 1)$ & 1 & 1 & 1 \\
$(1, 1, 0, \dots, 0, 1)$ & 1 & 1 & 1 \\
$(2, 0, \dots, 0, 2)$ & 1 & 1 & 1 \\

\hline
\end{tabular}
\caption{Decomposition of $\tau_2^{\Q}(\mathcal{HK}_{g,p}^b)$}
\label{table:decomptau2}
\end{table}

\begin{proof} Let $D_{g,p}^b$ be the decomposition given in Table \ref{table:decomptau2}. We show $\tau_2^{\Q}(\mathcal{HK}_{g,p}^b) = D_{g,p}^b$.

We know $\tau_2^{\Q}(\mathcal{HK}_{g,p}^b) \subset \tau_2^{\Q}(\mathcal{K}_{g,p}^b)$. For all $p+b \leq 1$, we decompose $\tau_2^{\Q}(\mathcal{K}_{g,p}^b)$ into irreducible $\operatorname{Sp}_{2g}(\Q)$ representations. These decompositions are related by the projection maps in the following diagram: \[\begin{tikzcd}
	{\tau_2^{\Q}(\mathcal{K}_g^1)} & {\Gamma_{0,2} \oplus \Gamma_{0,1} \oplus \mathbb{Q}} \\
	{\tau_2^{\Q}(\mathcal{K}_{g,1})} & {\Gamma_{0,2} \oplus \Gamma_{0,1}} \\
	{\tau_2^{\Q}(\mathcal{K}_g)} & {\Gamma_{0,2}}
	\arrow[equals, from=1-1, to=1-2]
	\arrow[from=1-1, to=2-1]
	\arrow["{\text{proj}}", from=1-2, to=2-2]
	\arrow[equals,from=2-1, to=2-2]
	\arrow[from=2-1, to=3-1]
	\arrow["{\text{proj}}", from=2-2, to=3-2]
	\arrow[equals, from=3-1, to=3-2]
\end{tikzcd}\] See \cite[Section 2.4]{Sakasai2}. The images $\tau_2^{\Q}(\mathcal{K}_{g,p}^b)$ branch into the following direct sums of $\operatorname{SL}_g(\Q)$ representations: $$\tau_2^{\Q}(\mathcal{K}_{g,p}^b) \cong D_{g,p}^b \oplus \Phi_{2,0, \dots, 0} \oplus \Phi_{0,2,0, \dots, 0}.$$ We show in Sections \ref{section:lowerboundtau} and \ref{section:taupuncture} that $D_{g,p}^b$ is exactly the image of the bracket map $$[ \cdot, \cdot]: \begin{cases}
    \bigwedge^2 \UQ_{\Q} \rightarrow \tau_2^{\Q}(\mathcal{HK}_g) \\
    \bigwedge^2 U_{\Q} \rightarrow \tau_2^{\Q}(\mathcal{HK}_{g,1}) 
   \end{cases}$$ when $b=0$ and $g \geq 5$, giving us a lower bound on $\tau_2^{\Q}(\mathcal{HK}_{g,p}^b)$. When $b=1$, the vector $\tau_2^{\Q}(T_{\partial})$ is a highest weight vector for the trivial representation $\Q$ (see \cite[Theorem A.3]{Faes}). Further, $2T_{\partial} \in [\mathcal{HI}_g^1, \mathcal{HI}_g^1]$ by Lemma \ref{lemma:boundarytwistdies} so $\tau_2(T_{\partial}) \in \operatorname{im}([\cdot,\cdot])$. It remains to show that the modules $\Phi_{2,0, \dots, 0}$ and $\Phi_{0,2,0, \dots, 0}$ are not in $\tau_2^{\Q}(\mathcal{HK}_{g,p}^b)$. In \cite[Theorem 5.1]{Faes}, Faes defines an $\mathcal{H}_g^1$-equivariant map $\operatorname{Tr}^A$ and proves that, for $g \geq 4$, we have $\tau_2(\mathcal{HK}_g^1) = \tau_2(\mathcal{K}_g^1) \cap \operatorname{ker}(\operatorname{Tr}^A)$. It is easy to check that the highest weight vectors for $\Phi_{2,0, \dots, 0}$ and $\Phi_{0,2,0, \dots, 0}$ are not in $\operatorname{ker}(\operatorname{Tr}^A)$, and our result follows.
\end{proof} 

\begin{corollary}
    When $g \geq 5$, the degree 2 part of $\bigoplus_k \tau_k^{\Q}(\mathcal{HK}_{g,p}^b)$ is generated as a Lie algebra by its degree 1 part.
\end{corollary} 

This decomposition is the setup needed to tackle the following problem. \begin{problem}
    Calculate the kernel of $$\tau_2^*:  H^2(\tau_2(\mathcal{HK}_{g,p}^b); \Q) \rightarrow H^2(\mathcal{HK}_{g,p}^b; \Q)$$ following Sakasai \cite{Sakasai2} for $\mathcal{K}_{g,p}^b$.
\end{problem}

\appendix

\section{Lie algebra actions} \label{appendix:liealgebra}

In this section, we give a more detailed overview of the actions of the Lie algebra $\mathfrak{sl}_g(\Q)$. Let $E_{ij} \in \mathfrak{sl}_g(\Q)$ be the matrix with a 1 in the $i^{th}$ row and $j^{th}$ column, and zeros everywhere else. Letting $\delta_{ij}$ be the Kronecker delta function, $E_{ij}$ acts on the standard $\mathfrak{sl}_g(\Q)$ representation $V_{\Q} = \Q \langle a_1, \dots, a_g \rangle$ by $$E_{ij}(a_k) = \delta_{jk} a_i$$ and on its dual $V^*_{\Q} = \Q \langle b_1, \dots, b_g \rangle$ by $$E_{ij}(b_k) = - \delta_{ik} b_j.$$ Table \ref{table:E12actionsonH} includes some examples of the action of $E_{12}$ on elements of $H_{\Q} = V_{\Q} \oplus V_{\Q}^*$.

\begin{table}[h!]
\centering
\begin{tabular}{|c|c|}
\hline
$\vec{v}$ & $E_{12}(\vec{v})$ \\
\hline
$a_1$ & $0$ \\
$a_2$ & $a_1$ \\
$a_3$ & $0$ \\
$b_1$ & $-b_2$ \\
$b_2$ & $0$ \\
$b_3$ & $0$ \\
\hline
\end{tabular}
\caption{Action of $E_{ij} \in \mathfrak{sl}_g(\Q)$ on vectors in $H_{\Q}$}
\label{table:E12actionsonH}
\end{table}

\begin{remark}
    $E_{ij}$ acts on $V_{\Q} \oplus V_{\Q}^*$ in the same way that the element $$\begin{bmatrix} E_{ij} & 0 \\ 0 & -E_{ij}^t \end{bmatrix} \in \mathfrak{sp}_{2g}(\Q)$$ acts on the standard $\mathfrak{sp}_{2g}(\Q)$ representation $H_{\Q}$.
\end{remark}

The matrix $E_{ij}$ acts on a vector in $H_{\Q}^{\otimes n}$ in the following way: $$E_{ij}(\vec{v_1} \otimes \vec{v_2} \otimes \dots \otimes \vec{v_n}) = E_{ij}(\vec{v_1}) \otimes \vec{v_2} \otimes \dots \otimes \vec{v_n} \; + \;  \vec{v_1} \otimes E_{ij}(\vec{v_2}) \otimes \dots \otimes \vec{v_n} \; + \dots + \; \vec{v_1} \otimes \vec{v_2} \otimes \dots \otimes E_{ij}(\vec{v_n}).$$ See \cite[Chapter 4.3.2]{Hall}. We include examples of the action of $E_{12}$ on various vectors in Table \ref{table:E12actions}. 

\begin{table}[h!]
\centering
\begin{tabular}{|c|c|}
\hline
$\vec{v}$ & $E_{12}(\vec{v})$ \\
\hline
$a_3 \otimes a_4 \otimes a_5$ & 0 \\
$a_2 \otimes a_3 \otimes a_4$ & $a_1 \otimes a_3 \otimes a_4$ \\
$a_1 \otimes a_2 \otimes a_3$ & $a_1 \otimes a_1 \otimes a_3$ \\
$b_1 \otimes b_3 \otimes b_4$ & $-b_2 \otimes b_3 \otimes b_4$ \\
$a_2 \otimes a_3 \otimes b_1$ & $a_1 \otimes a_3 \otimes b_1 - a_2 \otimes a_3 \otimes b_2$ \\
\hline
\end{tabular}
\caption{Action of $E_{12} \in \mathfrak{sl}_g(\Q)$ on vectors in $H_{\Q}^{\otimes 3}$}
\label{table:E12actions}
\end{table}

Since $\bigwedge^n H_{\Q}$ and $\operatorname{Sym}^n(H_{\Q})$ inject into $H_{\Q}^{\otimes n}$ via the maps $$x_1 \wedge \dots \wedge x_n \mapsto \sum_{\sigma \in S_n} \text{sgn}(\sigma) x_{\sigma(1)} \otimes \dots \otimes x_{\sigma(n)}$$ and $$x_1 \leftrightarrow \dots \leftrightarrow x_n \mapsto \sum_{\sigma \in S_n} x_{\sigma(1)} \otimes \dots \otimes x_{\sigma(n)},$$ we see that $E_{ij}$ acts on both $\bigwedge^n H_{\Q}$ and $\operatorname{Sym}^n(H_{\Q})$ via analogous product rules.

We could replace the tensor products $\otimes$ in Table \ref{table:E12actions} with the wedge product $\wedge$ or the symmetric product $\leftrightarrow$, and the action of $E_{12}$ is the same. Note that $E_{12}$ kills $a_1 \wedge a_2 \wedge a_3$ since $a_1 \wedge a_1 \wedge a_3 = 0$. This pattern is important in Sections \ref{appendix:hwv} and \ref{appendix:constructhwv}, in which we define highest weight vectors which are killed by {\it{all}} $E_{ij}$ with $i < j$. 

\section{Highest weight vectors} \label{appendix:hwv}

The group $\mathfrak{sl}_g(\Q)$ consists exactly of the $g \times g$ matrices with entries in $\Q$ and trace zero. We denote by $\mathfrak{h}$ the Cartan subalgebra of $\mathfrak{sl}_g(\Q)$, which is the set of matrices diagonal in $\mathfrak{sl}_g(\Q)$. Matrices $M \in \mathfrak{h}$ look like $$M =
\begin{bmatrix}
t_1 & 0 & \cdots & 0 & 0 \\
0 & t_2 & \cdots & 0 & 0 \\
\vdots & \vdots & \ddots & \vdots & \vdots \\
0 & 0 & \cdots & t_{g-1} & 0 \\
0 & 0 & \cdots & 0 & -t_1 - \dots - t_{g-1}
\end{bmatrix}. $$

\begin{definition}[Weight vector]
    A {\textit{weight vector}} $\vec{v}$ is a vector that is an eigenvector for $\mathfrak{h}$, i.e., any $M \in \mathfrak{h}$ scales $\vec{v}$. When $M$ scales $\vec{v}$ by $$w_1 t_1 + w_2(t_1 + t_2) + \dots + w_{g-1}(t_1 + \dots + t_{g-1}),$$ we say $\vec{v}$ has {\textit{weight}} $(w_1, w_2, \dots, w_{g-1})$. 
\end{definition} 

\begin{definition}[Highest weight vector]
    A {\textit{highest weight vector}} $\vec{v}$ is a weight vector that is killed by all $E_{ij}$ such that $i < j$.
\end{definition}

Any irreducible $\mathfrak{sl}_g(\Q)$ representation has a unique highest weight vector up to scaling. This highest weight $(w_1, \dots, w_{g-1})$ is always a $(g-1)$-tuple of nonnegative integers, and it determines the isomorphism class of the irreducible $\operatorname{SL}_g(\Q)$ representation. We denote this representation by $\Phi_{w_1, \dots, w_{g-1}}$. \\

We provide some examples of irreducible $\mathfrak{sl}_g(\Q)$ representations and their (highest) weight vectors.

\begin{example} Consider the vectors $a_k$ in the standard $\mathfrak{sl}_g(\Q)$ representation $V_{\Q} = \Q \langle a_1, \dots, a_g \rangle$. Then $$M \cdot a_k = t_k a_k,$$ so $a_k$ is a weight vector with weight $(0, \dots, 0, 1, 0, \dots, 0)$, where the $1$ is in the $k^{th}$ spot. Up to scaling, $a_1$ is the unique vector killed by all $E_{ij}$ with $i < j$, so $a_1$ is a highest weight vector for $V_{\Q} = \Phi_{1,0, \dots, 0}$.
\end{example}

\begin{nonexample}
    Consider the vector $a_1 + a_2 \in V_{\Q}$. We calculate $$M \cdot (a_1 + a_2) = t_1a_1 + t_2a_2,$$ which is not a scale of $a_1 + a_2$ for $t_1 \neq t_2$. Therefore, $a_1 + a_2$ is {\textit{not}} a weight vector.
\end{nonexample}

For our remaining examples, we will consider only highest weight vectors, and one can check that these vectors are all killed by $E_{ij}$ when $i < j$.

\begin{example} 
    The vector $a_1 \wedge a_2$ is a highest weight vector for $\bigwedge^2 V_{\Q}$ with weight $(0,1,0, \dots, 0)$ since $$M \cdot (a_1 \wedge a_2) = (t_1 a_1) \wedge a_2 + a_1 \wedge (t_2 a_2) = (t_1 + t_2) a_1 \wedge a_2.$$
\end{example}

In general, the vector $a_1 \wedge a_2 \wedge \dots \wedge a_n$ is a highest weight vector for $\bigwedge^n V_{\Q} = \Phi_{0, \dots, 0, 1, 0, \dots, 0}$, where the 1 is in the $n^{th}$ place.

\begin{example}
    The vector $b_g$ is a highest weight vector for the dual representation $V^*_{\Q}$ with weight $(0, \dots, 0,1)$. Our matrix $M \in \mathfrak{h}$ acts on $V_{\Q}^*$ via multiplication by its negative transpose: $$M \cdot b_g = -M^{t} b_g.$$ Since $$-M^{t} =
\begin{bmatrix}
-t_1 & 0 & \cdots & 0 & 0 \\
0 & -t_2 & \cdots & 0 & 0 \\
\vdots & \vdots & \ddots & \vdots & \vdots \\
0 & 0 & \cdots & -t_{g-1} & 0 \\
0 & 0 & \cdots & 0 & t_1 + \dots + t_{g-1}
\end{bmatrix}, $$ we have $$M \cdot b_g = (t_1 + \dots + t_{g-1})b_g$$ as desired.
\end{example}

\begin{example}
    The vector $b_{g-1} \wedge b_g$ is a highest weight vector for $\bigwedge^2 V^*_{\Q}$ with weight $(0, \dots, 0,1,0)$ since $$M \cdot b_{g-1} \wedge b_g = (-t_{g-1}b_{g-1}) \wedge b_g + b_{g-1} \wedge (t_1 + \dots + t_{g-1}) b_g = (t_1 + \dots + t_{g-2}) b_{g-1} \wedge b_g.$$
\end{example}

In general, the vector $b_{g-n+1} \wedge b_{g-n+2} \wedge \dots \wedge b_g$ is a highest weight vector for $\bigwedge^n V^*_{\Q} = \Phi_{0, \dots, 0, 1, 0, \dots, 0}$, where the 1 is in the $(g-n+1)^{st}$ place.

\begin{example}
   The vector $a_1 \otimes a_1$ is a highest weight vector for $\operatorname{Sym}^2(V_{\Q})$ with weight $(2,0, \dots,0)$ since $$M \cdot a_1 \otimes a_1 = (t_1 a_1) \otimes a_1 + a_1 \otimes (t_1 a_1) = 2t_1 a_1 \otimes a_1.$$
\end{example}

In general, $\underbrace{a_1 \otimes \dots \otimes a_1}_{n \text{ times}}$ is a highest weight vector for $\operatorname{Sym}^n(V_{\Q}) = \Phi_{n, \dots, 0}$. Similarly, $\underbrace{b_g \otimes \dots \otimes b_g}_{n \text{ times}}$ is a highest weight vector for $\operatorname{Sym}^n(V_{\Q}^*) = \Phi_{0, \dots, 0,n}$. \\

\begin{example} The trivial representation $\Q$ has weight $(0, \dots, 0)$ and
    the vector $a_k \otimes b_k$ is a weight vector for $\Q \subset V_{\Q} \otimes V_{\Q}^*$, since $$M \cdot a_k \otimes b_k = (t_k a_k) \otimes b_k + a_k \otimes (-t_k b_k) = 0.$$ Note that this is {\it{not}} a highest weight vector, since for fixed $1 \leq i < j \leq g$ we have \begin{itemize}
        \item $E_{ij}(a_i \otimes b_i) = -a_i \otimes b_j$
        \item $E_{ij}(a_j \otimes b_j) = a_i \otimes b_j$
        \item $E_{ij}(a_k \otimes b_k) = 0 \quad (k \neq i,j)$.
    \end{itemize}
    
It follows that the vector $$\sum_{k=1}^g a_k \otimes b_k$$ is a highest weight vector for $\Q \subset V_{\Q} \otimes V^*_{\Q}$. Similarly, the algebraic intersection form $$\omega := \sum_{k=1}^g a_k \wedge b_k$$ is a highest weight vector for $\Q$ expressed as a subrepresentation of $\bigwedge^2 H_{\Q} = \bigwedge^2 (V_{\Q} \oplus V^*_{\Q})$. 
\end{example}

\section{Constructing highest weight vectors} \label{appendix:constructhwv} Given a weight $(w_1, w_2, \dots, w_{g-1})$ and no other context, it is not hard to create highest weight vectors with this weight. The vector $$a_1^{\otimes w_1} \otimes (a_1 \wedge a_2)^{\otimes w_2} \otimes \dots \otimes (a_1 \wedge \dots \wedge a_{g-1})^{\otimes w_{g-1}}$$ works, as do $$a_1^{\otimes w_1} \otimes (a_1 \wedge a_2)^{\otimes w_2} \otimes \dots \otimes (a_1 \wedge \dots \wedge a_{g-1})^{\otimes w_{g-1}} \otimes (\sum_{i=1}^g a_i \wedge b_i)$$ and 
$$a_1^{\otimes w_1} \otimes (a_1 \wedge a_2)^{\otimes w_2} \otimes \dots \otimes (a_1 \wedge \dots \wedge a_{g-3})^{\otimes w_{g-3}} \otimes (b_{g-1} \wedge b_g)^{\otimes w_{g-2}} \otimes b_g^{\otimes w_{g-1}}.$$ In this paper, however, we must construct highest weight vectors that fit within the context of our problem. 

\subsection{The $\boldsymbol{\mathfrak{sl}_g(\Q)}$ representation $\boldsymbol{((\bigwedge^2 V_{\Q}) \otimes V_{\Q}^*) \otimes \operatorname{Sym}^2(V_{\Q}^*)}$} \label{appendix:thetaHWVs} In Section \ref{subsection:imJtensorimPsi}, we constructed highest weight vectors of given weights {\textit{expressed as elements of}} $((\bigwedge^2 V_{\Q}) \otimes V_{\Q}^*) \otimes \operatorname{Sym}^2(V_{\Q}^*)$. In the following examples, we describe this process. We use the fact that $((\bigwedge^2 V_{\Q}) \otimes V_{\Q}^*) \otimes \operatorname{Sym}^2(V_{\Q}^*)$ is generated by elements of the form $$[(a_i \wedge a_j) \otimes b_k] \otimes [b_{\ell} \otimes b_{\ell}]$$ and $$[(a_i \wedge a_j) \otimes b_k] \otimes [b_{\ell} \otimes b_m + b_m \otimes b_{\ell}],$$ where the indices are not necessarily distinct.

\begin{example}[Easy guess]
    We want a highest weight vector of weight $(0,1,0, \dots, 0,3)$ expressed as an element of $((\bigwedge^2 V_{\Q}) \otimes V_{\Q}^*) \otimes \operatorname{Sym}^2(V_{\Q}^*)$. Only one basis element of $((\bigwedge^2 V_{\Q}) \otimes V_{\Q}^*) \otimes \operatorname{Sym}^2(V_{\Q}^*)$ has the desired weight, and that is $$[(a_1 \wedge a_2) \otimes b_g] \otimes [b_g \otimes b_g].$$ We can check that all $E_{ij}$  with $i < j$ kill this vector, so we are done.
\end{example}

\begin{example}[Solve system of equations] \label{example:systemequations}
    We want a highest weight vector of weight $(0,1,0, \dots, 0,1,1)$ expressed as an element of $((\bigwedge^2 V_{\Q}) \otimes V_{\Q}^*) \otimes \operatorname{Sym}^2(V_{\Q}^*)$. To have the desired weight, our vector must be of the form $$v_{\lambda_1, \lambda_2} := \lambda_1 [(a_1 \wedge a_2) \otimes b_{g-1}] \otimes [b_g \otimes b_g] + \lambda_2 [(a_1 \wedge a_2) \otimes b_g] \otimes [b_{g-1} \otimes b_g + b_g \otimes b_{g-1}].$$ for some $\lambda_1, \lambda_2 \in \Q$. Because there is exactly one copy of $\Phi_{0,1,0, \dots, 0,1,1}$ in the decomposition of $((\bigwedge^2 V_{\Q}) \otimes V_{\Q}^*) \otimes \operatorname{Sym}^2(V_{\Q}^*)$, there is a unique choice (up to scaling) of $\lambda_1$ and $\lambda_2$ such that $E_{ij}(v_{\lambda_1, \lambda_2}) = 0$ for all $i < j$. To solve for $\lambda_1$ and $\lambda_2$, we apply all such $E_{ij}$ and solve the resulting system of equations.

    If we have $E_{ij} \neq E_{(g-1)g}$ such that $i < j$, then $$E_{ij}(v_{\lambda_1, \lambda_2}) = 0$$ as desired. Then we calculate $$E_{(g-1)g}(v_{\lambda_1, \lambda_2}) = (- \lambda_1 - 2 \lambda_2) [(a_1 \wedge a_2) \otimes b_g] \otimes [b_g \otimes b_g]$$ so we choose $\lambda_1 = -2$ and $\lambda_2 = 1$.

\end{example}

\subsection{A basis for $\mathbf{\bigwedge^2 \UQ_{\Q}}$} \label{appendix:wedge2Ubasis} In Section \ref{subsection:0101}, we constructed highest weight vectors expressed as elements of $\bigwedge^2 \UQ_{\Q}$. To do this, we followed the same process as in Example \ref{example:systemequations} and so we needed a basis for $\bigwedge^2 \UQ_{\Q}$. \\

It follows from Section \ref{section:setuphi} that $\UQ_{\Q}$ is the submodule of $(\bigwedge^3 H_{\Q}) / H_{\Q}$ with the following basis\footnote{If we added the ``triple-$a$" generators $a_i \wedge a_j \wedge a_k$ to this list, we would have a basis for $(\bigwedge^3 H_{\Q}) / H_{\Q}$.}: \begin{itemize}
    \item $a_i \wedge a_j \wedge b_k$ \quad  $(i < j; \; k \neq i,j)$ 
    \item $a_i \wedge b_j \wedge b_k$ \quad  $(i \neq j,k; \; j < k)$ 
    \item $b_i \wedge b_j \wedge b_k$ \quad $(i < j < k)$
    \item $a_i \wedge a_j \wedge b_j - \frac{1}{g-1} a_i \wedge \omega$ \quad $(i \neq j)$
    \item $b_i \wedge a_j \wedge b_j - \frac{1}{g-1} b_i \wedge \omega$ \quad $(i \neq j)$
\end{itemize}

\noindent A basis for $\bigwedge^2 \UQ_{\Q}$ consists of all wedge products of basis elements of $\UQ_{\Q}$.

\section{LiE commands} \label{appendix:LiE}

We include the commands we used in LiE \cite{LiE} to decompose our representations. We choose $g=9$, which is the smallest genus for which LiE's outputs are unambiguous.\footnote{For example, when $g=8$, the modules  $\Phi_{1,0,0,1,0, \dots, 0,1}$ and $\Phi_{1,0, \dots, 0,1,0,0,1}$ both have highest weight $(1,0,0,1,0,0,1)$ so LiE does not differentiate them. When $g=9$, these modules have weights $(1,0,0,1,0,0,0,1)$ and $(1,0,0,0,1,0,0,1)$, respectively.}

\subsection{Tensor products} To decompose the $\operatorname{SL}_g(\Q)$ representation $\Phi_{1,0, \dots ,0,1,0} \otimes \Phi_{0,1,0, \dots ,0,1}$: 

\medskip

{\begin{verbatim}
tensor([1,0,0,0,0,0,1,0],[0,1,0,0,0,0,0,1],(A8))
\end{verbatim}}

\medskip

\noindent This is a step in computing the decomposition in Table \ref{table:decompAlt2U}.

\subsection{Wedge products} To decompose the $\operatorname{SL}_g(\Q)$ representation $\bigwedge^2 (\bigwedge^3 V_{\Q})$: 

\medskip

{\begin{verbatim}
alt_tensor(2,[0,0,1,0,0,0,0,0],(A8))
\end{verbatim}}

\medskip

\noindent This is a step in computing the decomposition in Table \ref{table:decompAlt2U}.

\subsection{Branching} To decompose the $\operatorname{Sp}_{2g}(\Q)$ representation $\Gamma_{0,1,0,1}$ as a direct sum of irreducible $\operatorname{SL}_g(\Q)$ modules, as in Section \ref{subsection:0101}:

\medskip

{\begin{verbatim}
branch([0,1,0,1,0,0,0,0,0],A8,[[1,0,0,0,0,0,0,0],[0,1,0,0,0,0,0,0], 
[0,0,1,0,0,0,0,0],[0,0,0,1,0,0,0,0],[0,0,0,0,1,0,0,0],[0,0,0,0,0,1,0,0], 
[0,0,0,0,0,0,1,0],[0,0,0,0,0,0,0,1],[0,0,0,0,0,0,0,0]],(C9))
\end{verbatim}}

\end{document}